\documentclass[12pt]{amsart}

\usepackage[OT1]{fontenc}
\usepackage{amsmath, mathrsfs, amsfonts, amssymb, enumerate, caption, subcaption}
\usepackage[numbers]{natbib}
\usepackage[margin=0.75in]{geometry}
\usepackage{nameref}
\usepackage{zref-xr,zref-user}
\zxrsetup{toltxlabel}

\usepackage{tikz}
\usetikzlibrary{arrows, calc}
\usetikzlibrary{patterns, shapes, decorations.markings}

\usepackage[colorlinks,citecolor=blue,urlcolor=blue]{hyperref}
\usepackage{pgfplots}
\pgfplotsset{/pgf/number format/use comma,compat=newest}
\usetikzlibrary{external}
\numberwithin{equation}{section}

\numberwithin{equation}{section}
\theoremstyle{definition}
\newtheorem{definition}[equation]{Definition}
\theoremstyle{definition}
\newtheorem{remark}[equation]{Remark}
\theoremstyle{definition}

\theoremstyle{definition}

\theoremstyle{lemma}
\newtheorem{assumption}[equation]{Assumption}
\theoremstyle{lemma}
\newtheorem{lemma}[equation]{Lemma}
\theoremstyle{theorem}
\newtheorem{theorem}[equation]{Theorem}
\theoremstyle{proposition}
\newtheorem{proposition}[equation]{Proposition}
\theoremstyle{corollary}
\newtheorem{corollary}[equation]{Corollary}

\theoremstyle{corollary}
\theoremstyle{definition}

\theoremstyle{example}

\theoremstyle{proposition}

\theoremstyle{definition}

\newcommand{\E}{\mathbf{E}}
\renewcommand{\P}{\mathbf{P}}
\newcommand{\R}{\mathbf{R}}
\newcommand{\C}{\mathbf{C}}
\newcommand{\N}{\mathbf{N}}
\newcommand{\Z}{\mathbf{Z}}
\newcommand{\mcl}{\mathcal{L}}
\renewcommand{\d}{d}

\zexternaldocument*[II-]{polyZII}

\title[Noise-Induced Stabilization of Planar Flows I]{Noise-Induced Stabilization of Planar Flows I}

\author[D. P. Herzog]{David P. Herzog}
\author[J. C. Mattingly]{Jonathan C. Mattingly}
\makeindex
\begin{document}
\maketitle

\begin{abstract}
  We show that the complex-valued ODE ($n\geq 1$, $a_{n+1} \neq 0$):
  \begin{equation*}
    \dot z = a_{n+1}  z^{n+1} + a_n z^n+\cdots+a_0 , 
  \end{equation*}
  which necessarily has trajectories along which the dynamics blows up
  in finite time, can be stabilized by the addition of an arbitrarily
  small elliptic, additive Brownian stochastic term.  We also show
  that the stochastic perturbation has a unique invariant probability measure
  which is heavy-tailed yet is uniformly, exponentially attracting. The methods turn on the construction of Lyapunov
  functions. The techniques used in the construction are general and
  can likely be used in other settings where a Lyapunov function is
  needed.  This is a two-part paper.  This paper, Part I, focuses on
  general Lyapunov methods as applied to a special, simplified version
  of the problem.  Part II \cite{HerzogMattingly2013II} extends the
  main results to the general setting.
\end{abstract}



\section{Introduction}

We study the following complex-valued system
\begin{equation}\label{eqn:genpoly}
\left\{
  \begin{aligned}
    dz_t &=(a_{n+1} z_t^{n+1} +a_{n}z_t^n + \cdots + a_0) \, dt + \sigma \, dB_t\\
    z_0 &\in \C
  \end{aligned}\right.
\end{equation}
where $n\geq 1$ is an integer, $a_i \in \C$, $a_{n+1} \neq 0$, $\sigma \geq
0$ is constant, and $B_t=B_t^{(1)} + i B_t^{(2)}$ is a complex
Brownian motion defined on a probability space $(\Omega, \mathcal{F},
\P)$. (We in fact prove many results for the slightly more general form given
in \eqref{eqn:polyZglot}.)

When $\sigma=0$ in equation \eqref{eqn:genpoly}, the resulting
deterministic system blows up in finite time for some
susbset of the initial data\footnote{This can be intuited by noting that the asymptotic
equation $\dot{z}_t=a_{n+1} z_t^{n+1}$ at infinity has $n$ explosive
trajectories (see Figure \ref{fig:orbits})}.  In particular even with the addition of noise ($\sigma>0$ in \eqref{eqn:genpoly}), one cannot employ the general results in \cite{Ver_00}  
as the drift vector field does not point inward radially. 
Interestingly, however, we will see that when $\sigma >0$ in \eqref{eqn:genpoly}, not only does the system
stabilize so that solutions to \eqref{eqn:genpoly} exist for all initial conditions and all finite times, but the dynamics settles down
into a unique statistical steady state with corresponding invariant
measure.  Strikingly, too, we will find that this invariant measure
has the following two properties. First, it possesses an everywhere
positive density (with respect to Lebesgue measure on $\C$) which
decays polynomially in $|z|$ at infinity. Second, it attracts all
initial conditions exponentially fast in time. Note that these two
properties cannot be simultaneously realized in a gradient system with additive noise and nominal growth assumptions on the potential, but it is possible here because the system is
strongly non-reversible with a nontrivial probability flux in
equilibrium. See Remark~\ref{fatAndFast} for a further discussion of this point. 
\begin{figure}\label{fig:orbits}
        \centering
        \begin{subfigure}[b]{0.4\textwidth}
                \includegraphics[width=\textwidth]{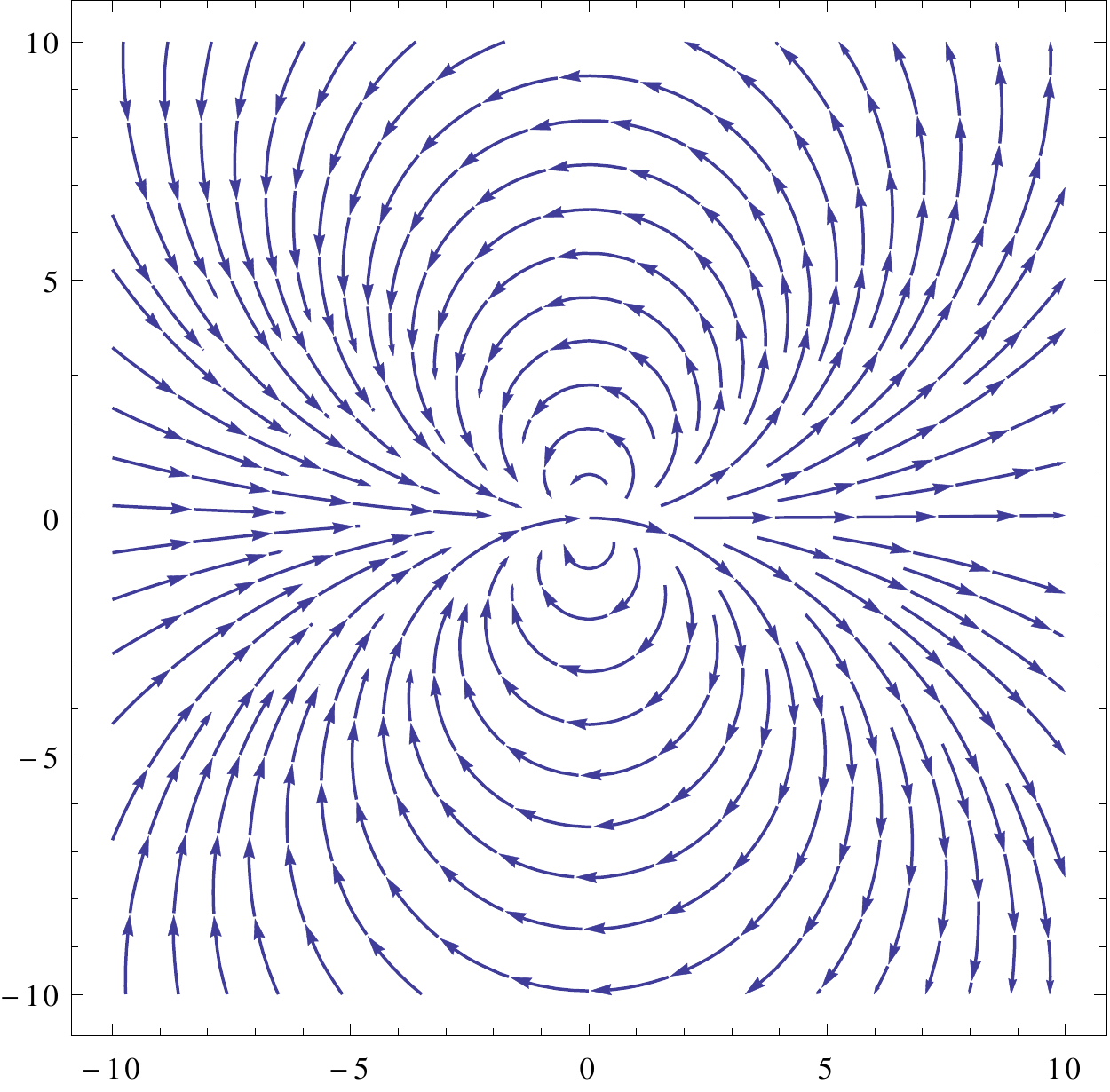}
                \caption{$\dot z = z^2$ (case: $n=1$)}
                \label{fig:n1}
        \end{subfigure}
        ~~
        \begin{subfigure}[b]{0.4\textwidth}
                \includegraphics[width=\textwidth]{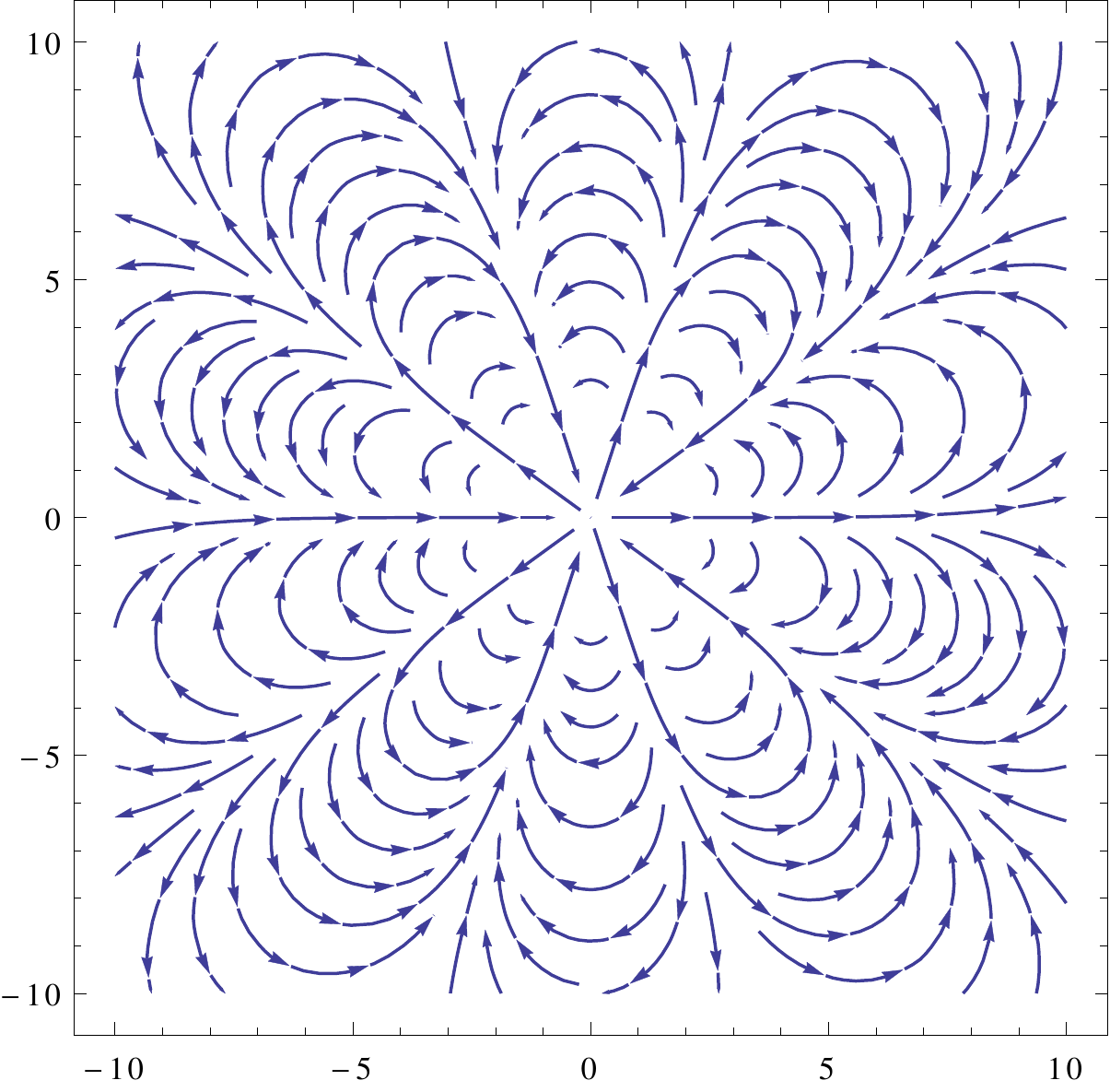}
                \caption{$\dot z = z^6$ (case: $n=5$)}
                \label{fig:n5}
        \end{subfigure}
        \caption{The orbits of $\dot z = z^{n+1}$.  Trajectories with initial condition $z_0=r e^{i\theta}$ satisfying $r>0$ and $\theta =\tfrac{2\pi k}{n}$, $k \in \Z$, explode to infinity in finite time.}\label{fig:orbits}
\end{figure}

We will see that the polynomial decay follows from a delicate
balance between the noise and unstable dynamics resulting in a global
circulation in equilibrium. Such circulation produces what might be
called ``intermittent'' behavior. Namely, the system spends long
periods of time in an order one region about the origin but at
approximately exponential times the system ``spikes'', making rapid
excursions to large values followed by equally rapid returns to
order one values (see Figure \ref{fig:timeseries}). The balance
of noise and explosion that produces the polynomial decay at infinity
also implies a specific scaling between the level of the spike and the
parameter which determines the distribution of times between spikes. This and its
relationship to the equilibrium flux 
are highlighted in the heuristic
discussions in Section~\ref{Heuristic Discussions}.

Although the results obtained here are specific to the equation
\eqref{eqn:genpoly} (and the generalization  \eqref{eqn:polyZglot}), the methods used should be applicable to a wide
range of problems. That is, to establish the main results, a sequence
of ``optimal" Lyapunov functions is constructed, and because of the
delicate interplay between the noise and the instabilities of the
underlying deterministic system, it forces one to know how to build
such functions well.  Although we do not claim to have a step-by-step
algorithmic procedure which would produce a Lyapunov function for a
given stochastic differential equation, we at least give a framework
that could be molded to handle a variety of situations. That is, we believe that
the core ideas can be applied quite broadly.

It is important to remark that the system \eqref{eqn:genpoly} and
other similar planar systems have been studied before \cite{Aarts_13, AGS_13, AKM_12,
  BodDoe_12, GHW_11, Her_11, Sch_93}.  In the case when $n=1$ in \eqref{eqn:genpoly}, the asymptotic behavior of the invariant density was first studied in \cite{GHW_11} to help extract information on the distribution of spacing between close, heavy particles transported by moderately turbulent flows.  Subsequent works \cite{Aarts_13, AGS_13, AKM_12, GHW_11, Her_11} have since investigated \eqref{eqn:genpoly} in special cases of the polynomial drift term.  In particular focusing on complex Langevin dynamics, the authors in \cite{Aarts_13, AGS_13} study numerically the equilibrium distribution when $n=2$ in equation \eqref{eqn:genpoly}.  In
addition to proving results for the general system
\eqref{eqn:genpoly}, we will also improve upon the existing ones in
these special cases.  Specifically, our results will be seen to be
optimal in the following sense: If $\mu$ denotes the unique invariant
measure for \eqref{eqn:genpoly}, then for any $\gamma \in (0,2 n)$ we
will succeed in constructing a Lyapunov function for
\eqref{eqn:genpoly} whose growth at infinity will imply that
\begin{align*}\int_{\C}  (1+|z|^\gamma ) d\mu(z) < \infty, \end{align*}
yet we will see that for any $\gamma \geq 2n$ \begin{align*} 
\int_{\C}  (1+|z|^\gamma ) d\mu(z) =\infty.
\end{align*} 
Hence it would be impossible to build a Lyapunov function with power
growth at infinity which grows faster than those we construct.

\begin{figure}[h]
\centering
\hspace{-.3in}\includegraphics[width=5in]{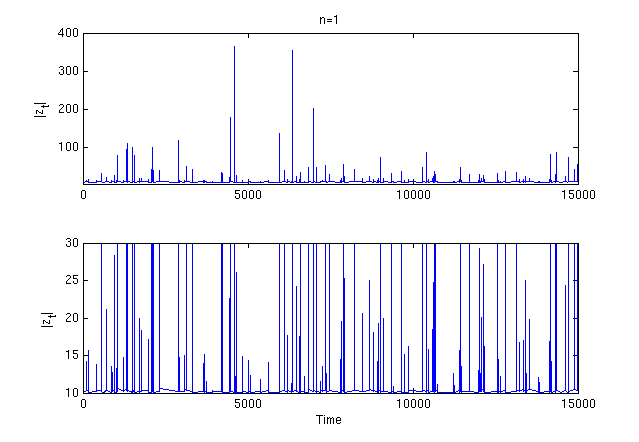}
\caption{A realization of the process $|z_t|$ plotted on the time
  interval $[0, 15000]$ where $z_t$ solves equation
  \eqref{eqn:genpoly} with $n=1$, $a_{2}=1$, $a_1=a_0=0$ and
  $\sigma=1$. See Section~\ref{sec:spacing} for a discussion of the
  spacing between the ``spikes.''}
\label{fig:timeseries}
\end{figure}

The layout of this paper is as follows.  In
Section~\ref{sec:formalInvM}, we give two non-rigorous arguments which suggest a possible rate of decay
at infinity for the invariant probability density function.  The first
argument given relies on formal asymptotic matching at the level of
the PDE solved by the invariant density and the second argument is
based on the stochastic dynamics near the point at infinity.  While
the first argument is shorter, the second is more informative in that
it not only allows us to understand the the decay rate of the
invariant density but it also allows us to understand the inter-spike
spacing distribution as displayed in Figure~\ref{fig:timeseries} and
confirm all results predicted with numerics. 
It is also almost certainly possible to make the second discussion rigorous.
 In Section~\ref{mainResults}, we state the main rigorous results
of this paper, Part I, and its continuation, Part II
\cite{HerzogMattingly2013II}, which henceforth will be referred to as
the sequel. In Section~\ref{sec:conseq_lyap}, we give the precise
definition of a Lyapunov function we will use and list some of the
consequences of its existence.  It is interesting to note that,
although our notion of a Lyapunov function is similar to that of Meyn
and Tweedie \cite{MT_93} and Khasminskii \cite{Has_12}, we only require Lyapunov functions to be
piecewise $C^2$ and globally continuous rather than globally
$C^2$. While this is useful in constructing Lyapunov functions,
it forces us to employ a generalized It\^{o}-Tanaka formula due to
Peskir \citep{Peskir_07} to estimate contributions along curves where
our Lyapunov function is not $C^2$.  This allows us to avoid smoothing
or mollifying along the boundaries which leads to a substantial
reduction in the complexity of  the argument when
compared to previous works \cite{AKM_12,GHW_11, Her_11, BodDoe_12}. In
Section~\ref{sec:mainResultsOutline}, we state the precise results we
will actually prove which, when combined with the results in
Section~\ref{sec:conseq_lyap}, will imply the main results as stated
in Section~\ref{mainResults}.  In Section~\ref{sec:GenOver}, the key
initial steps of the construction procedure that will produce the
required Lyapunov functions are discussed. In particular, we show how
we plan to apply Peskir's result \citep{Peskir_07}, allowing us to
work with less regular Lyapunov functions.
To illustrate our general methods, in
Section~\ref{sec:theconstruction} we build our Lyapunov functions
corresponding to the system \eqref{eqn:genpoly} assuming that there
are no ``significant" lower-order terms in the drift of
\eqref{eqn:genpoly}.  Section~\ref{sec:thedetails} finishes the
remaining details in this special case.  
%

In the sequel (Part II of this paper) \cite{HerzogMattingly2013II}, we
prove the results given in Section~\ref{sec:mainResultsOutline}
pertaining to \eqref{eqn:genpoly} but in the general setting without
this simplifying assumption.  Extending the result is subtly intricate
as the presence of higher-order polynomial terms with degree $\leq n$
drastically alters the nature of the process at infinity.  Moreover,
we will also give a short proof of a weaker version of Peskir's result
\cite{Peskir_07} suitable for our needs. In
Section~\ref{sec:conclusion}, we summarize what has been accomplished
in Part I of the paper with the advantage of hindsight.  We postpone
discussions of possible future directions to Part II of the work.

\section*{Acknowledgements}
The problem studied in this paper and the sequel was originally posed
by Jan Wehr, and we thank him for first asking the question and for
fruitful discussions regarding it.  We also thank Avanti Athreya,
Tiffany Kolba, Matti Leimbach, Scott Mckinley and David Schaeffer for fruitful discussions about this
and related problems. We thank Denis Talay for a talk in which he
brought  \cite{Peskir_07} to our attention and subsequent discussions
about the possible application to this setting. We would also like to acknowledge partial
support of the NSF through grant DMS-08-54879 and the Duke University Dean's office.


\section{Heuristic Discussions of the Decay at Infinity and Spike Spacing}
\label{Heuristic Discussions}
\label{sec:formal_asy}
\label{sec:formalInvM}

In this section, we present some heuristic, non-rigorous arguments
which give information about the possible structure of the invariant
measure at infinity as well as the structure of the inter-spike distribution displayed in Figure~\ref{fig:timeseries}.  Here we focus our efforts on the simplified equation
\begin{align}
\label{eqn:zn}
  d z_t = z^{n+1}_t \, dt +  \sigma dB_t
\end{align}
where $z_t \in \C$, $n\geq 1$, $\sigma >0$, and $B_t$ is a complex Brownian motion.  Although matters could be complicated by the presence of lower order terms in the drift, studying the equation above is a good starting point and it reveals much of the structure of the general equation \eqref{eqn:genpoly}.   

In Section~\ref{sec:scalingargumentp}, we start by giving an asymptotic matching argument which suggests a possible decay rate at infinity for the invariant probability density function.  In Section~\ref{sec:hueristicmodel}, we will then develop a heuristic model of the dynamics informed
by a detailed scaling analysis carried out later in Section~\ref{sec:PiecewiseI}. In Section~\ref{sec:tail}, we analyze this heuristic
model and see that it implies the same decay as predicted in Section~\ref{sec:scalingargumentp}. One advantage that the heuristic model has over the scaling argument given in Section~\ref{sec:scalingargumentp} is that it gives a more dynamic picture of the processes which lead to the polynomial decay at infinity of the stationary measure. Moreover, we will be able to, in Section~\ref{sec:spacing}, use the same heuristic model to explore the spike spacing illustrated in Figure~\ref{fig:timeseries} as well as 
validate many of the results obtained in this paper with numerical simulations.

\subsection{Scaling Argument}
\label{sec:scalingargumentp}
Let $\mathcal{L}$ denote the generator of the process $(x_t, y_t)$
where $x_t =\text{Re}(z_t)$, ${y_t=\text{Im}(z_t)}$ and $z_t$
satisfies \eqref{eqn:zn}.  Since the formal adjoint $\mathcal{L}^*$ of
$\mathcal{L}$ with respect to Lebesgue measure on $\R^2$ is uniformly
elliptic and has $C^\infty$ coefficients, any invariant probability
density function $\rho(x,y)$ with respect to Lebesgue measure on
$\R^2$ must be globally positive, $C^\infty$, and satisfy the equation  $\mathcal{L}^* \rho=0$ on $\R^2$.  To analyze the behavior of $\rho(x,y)$ as $|(x,y)|\rightarrow \infty$, we first convert the equation $\mathcal{L}^* \rho=0$ to polar coordinates. Letting $\tilde{\rho}(r, \theta)= \rho(r \cos(\theta), r \sin(\theta))$ we see that $\tilde{\rho}$ satisfies the following equation for $r>0$    
\begin{multline*}
 r^n\big[ (2 n+2)\cos(n \theta) \tilde{\rho} + r \cos(n \theta) \partial_r \tilde{\rho}  + 
 \sin(n\theta)\partial_\theta \tilde{\rho}\big]  -  \sigma^2\big[ \frac{1}{2r}\partial_r\tilde{ \rho}+ \frac{1}{2}\partial_r^2
 \tilde{\rho}   + \frac{1}{2
   r^2} \partial_\theta^2 \tilde{\rho}\big] =0\,.
\end{multline*}
Considering the effect of the scaling transformation $(r,\theta) \mapsto (\lambda r
,\theta)$ on the equation above produces
\begin{multline*}
\lambda^n r^n\big[ (2 n+2)\cos(n \theta) \tilde{\rho} + r \cos(n \theta) \partial_r \tilde{\rho}  + 
 \sin(n\theta)\partial_\theta \tilde{ \rho}\big] -  \frac{\sigma^2}{\lambda^2}\big[ \frac{1}{2r}\partial_r \tilde{\rho}+ \frac{1}{2}\partial_r^2
\tilde{ \rho}   + \frac{1}{2
   r^2} \partial_\theta^2 \tilde{\rho}\big] =0\,.
\end{multline*}
Observe that this transformation allows us to gauge the asymptotic behavior of $\tilde{\rho}$, hence $\rho$, along a fixed radial direction.  
Extracting the leading order $\lambda^n$ term assuming that
$\partial_\theta \tilde{\rho}=0$ and $\cos(n \theta)\neq 0$, we obtain the leading order  equation
\begin{align*}
  (2n+2)\tilde{ \rho} + r \partial_r \tilde{\rho} =0.
\end{align*}
Solving this equation produces
\begin{align}
  \label{eq:formalScaling}
  \rho(x, y)=\tilde{\rho}(r,\theta)  \sim \frac1{r^{2n+2}}\,.
\end{align}
We will see later that, in fact, that this scaling is essentially correct.

\begin{figure}
        \centering
        \begin{subfigure}[b]{0.3\textwidth}
                \includegraphics[width=\textwidth]{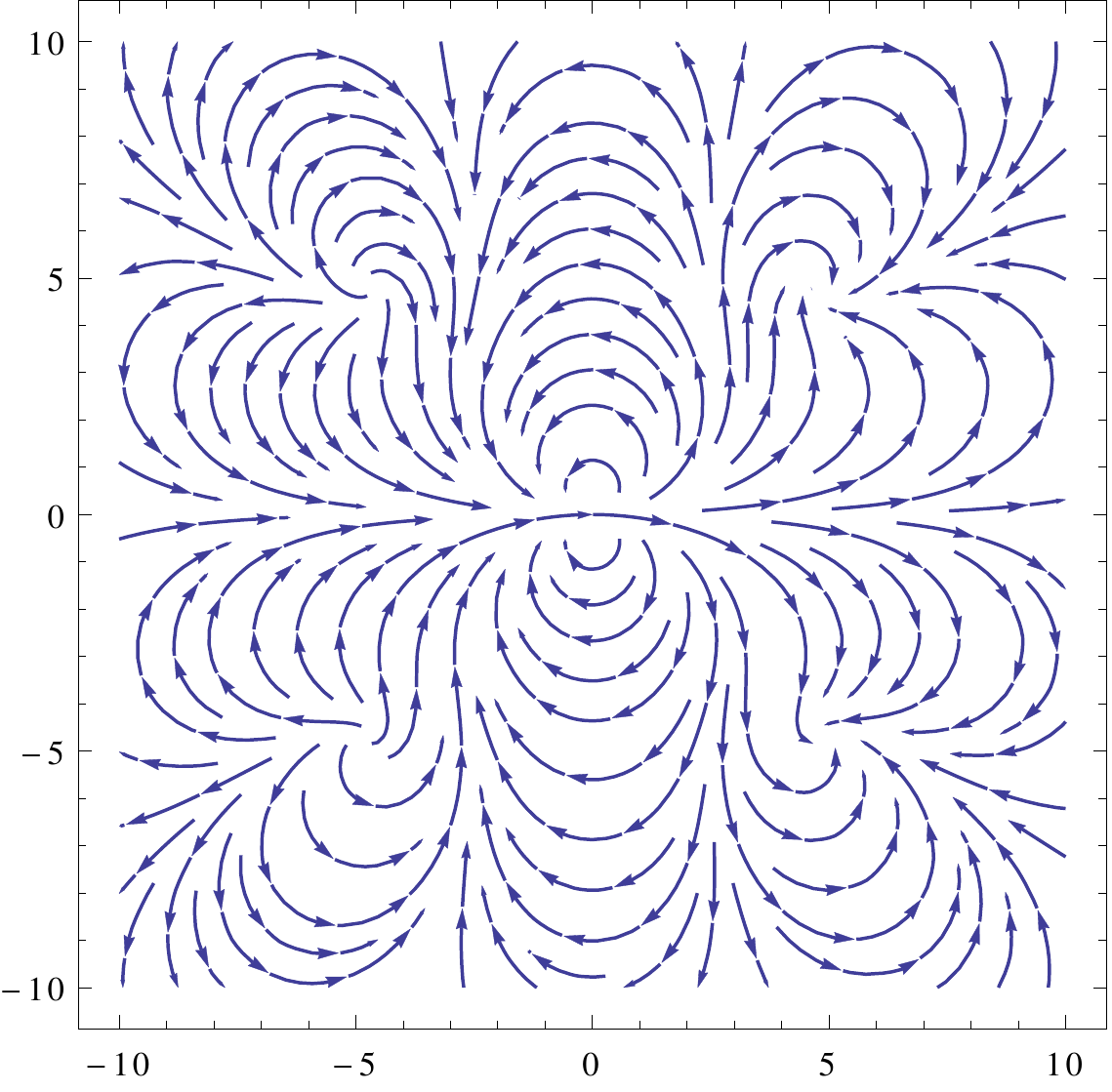}
        \end{subfigure}
        ~~
        \begin{subfigure}[b]{0.3\textwidth}
                \includegraphics[width=\textwidth]{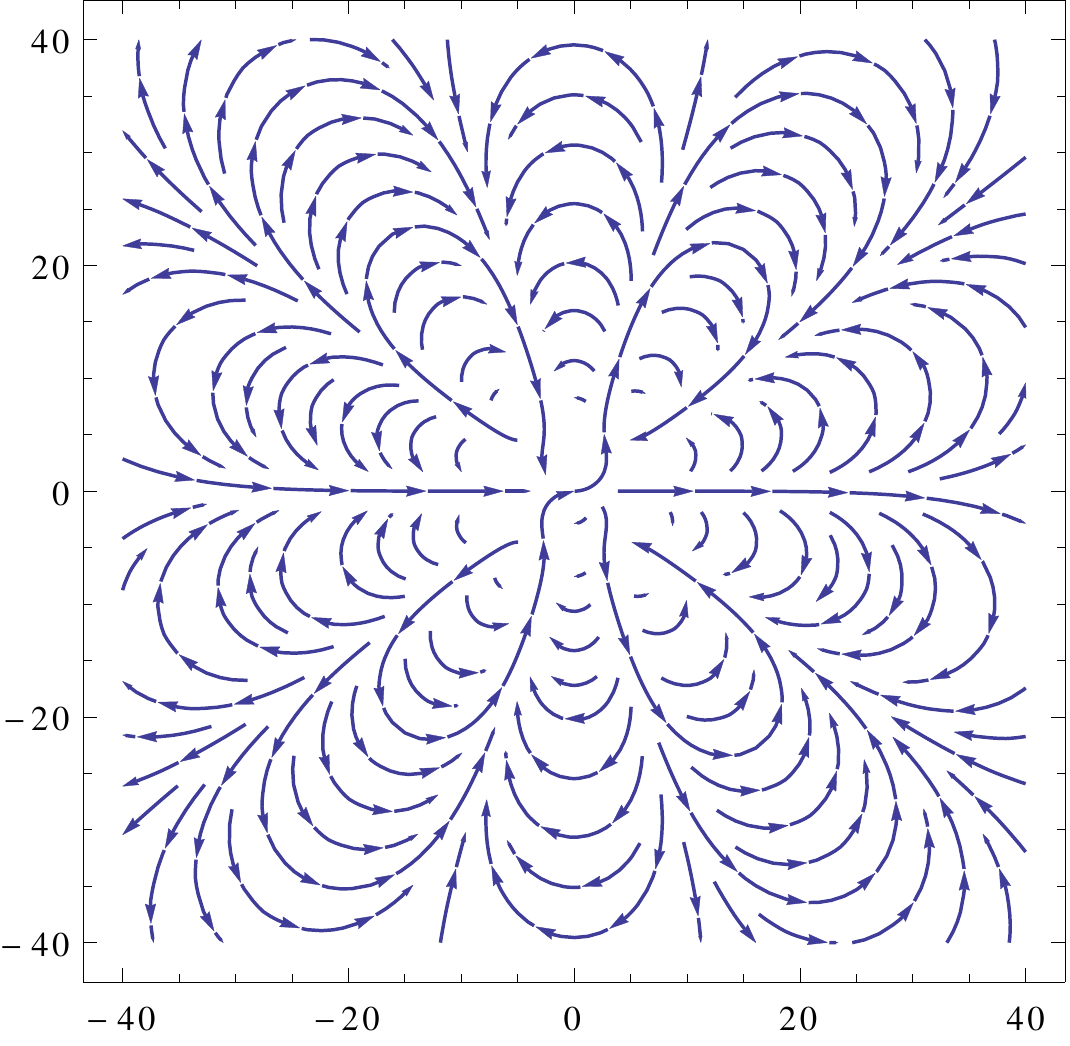}
        \end{subfigure}
        \caption{The trajectories of $\dot{z}=z^6 +2000 z^2$ plotted on $[-10,10]^2$ (left) and $[-40,40]^2$ (right).  For $|z|$ small, the dynamics qualitatively resembles that of $\dot{z}=z^2$ (see Figure \ref{fig:n1}).  As $|z|$ becomes larger, the dynamics starts to resemble that of $\dot{z}=z^6$ (see Figure \ref{fig:n5}).}
\end{figure}

\subsection{The Heuristic Model}
\label{sec:hueristicmodel}
First observe that the rotated process $Z_t:=e^{i\frac{2\pi k}{n}}z_t$ solves the equation
\begin{align*}
dZ_t = Z_t^{n+1} \, dt + \sigma \, d\widetilde{B}_t
\end{align*}
where $\widetilde{B}_t$ is also a complex Brownian motion.  In particular, this fact allows us to reduce our analysis of the process $z_t$ restricted to the set (in polar coordinates)  
\begin{align}
 \mathcal{R}=\{(r, \theta)\,: \,  r\geq r^*,\, -\tfrac{\pi}{n} \leq \theta \leq \tfrac{\pi}{n} \}
\end{align}
as the analysis in the remaining wedge-shaped regions $\mathcal{R}_k$, $k\in \Z$, given by
\begin{align*}
\mathcal{R}_k= \{(r, \theta) \, : \, (r, \theta- 2\pi k/n) \in \mathcal{R} \}\end{align*}
can be carried out in a similar fashion.

As we will see in Section~\ref{sec:PiecewiseI}, to leading order inside of $\mathcal{R}$ the noise only has an effect
inside of the region 
\begin{align*}
  \mathcal{S}(\eta^*,r^*)=\{(r, \theta) \in \mathcal{R}\, : \, |\theta| r^{\frac{n+2}{2}} \leq
  \eta^*\} \cap\{ (r,\theta) : r \geq r^*\}  
\end{align*} 
where $r^*, \eta^* >0$ are both large and fixed.  In order to study the process $z_t=r_t e^{i\theta_i}$ in this region, it is convenient to introduce the variable $$\eta_t :=
\theta_t r_t^{\frac{n+2}{2}}$$ which, when paired with $r_t$, 
still completely determines the state of the system at time $t$.  Moreover, after making the
time change 
\begin{align}\label{eq:timechangeIntro}
  t \mapsto \int_0^t r_s^n ds, 
\end{align}
we will see later that the $(\eta_t,r_t)$ dynamics are well-approximated in $\mathcal{S}(\eta^*,r^*)$ by
\begin{align}\label{eq:approxEtaR}
  d\eta_t = \Big(\frac32 n+1\Big) \eta_t dt + \sigma dW_t
  \quad\text{and}\quad\dot r_t = r_t
\end{align}
where $W_t$ is a standard scalar Brownian Motion.  In contrast,  when the process belongs to the set $ \mathcal{S}(\eta^*,r^*)$, it approximately follows the deterministic equation $\dot z = z^{n+1}$ (or rather $\dot z = z^{n+1}|z|^{-n}$ after the time change). The orbits of this system are
concentric loops which share a single common point $z=0$. Indexing
each loop by its maximal distance from the origin $K$, we see that the
$K$-th  loop is the locus of points
\begin{align}\label{Jset}
  \mathcal{J}(K)=\Big\{(r,\theta) \,: \, r  = K |\sin(n\theta)|^{\frac{1}{n}},\,
0<  |\theta| < \frac\pi{n} \Big\}\,.
\end{align}

Combining these two approximate dynamics, we obtain a model for the
behavior far from the origin. It is not hard to argue that the most
probable route to a point $\zeta \in  \mathcal{R}$ far from the
origin, with $|\zeta| >r^*$, is to first enter
$\mathcal{S}(\eta^*,r^*)$ through the set $\{ (r,\theta) : r =
r^*\}$. The trajectory system will then spend a random about
of time in $\mathcal{S}(\eta^*,r^*)$ before exiting at time 
\begin{align*}
   \tau=\inf\{ t >0 : z_t \not \in \mathcal{S}(\eta^*,r^*)\}=\inf\{ t >0 : |\eta_t| > \eta_*\}\,.
\end{align*}
Since we are using the approximate dynamics \eqref{eq:approxEtaR} in
$\mathcal{S}(\eta^*,r^*)$, it is not hard to approximate both the exit time $\tau$ and exit location from $\mathcal{S}(\eta^*,r^*)$.  Clearly, 
$\eta_\tau=\pm\eta^*$. Moreover, it is also easy to see from \eqref{eq:approxEtaR} that $r_\tau=r^* e^{\tau}$.
Since $$|\eta_\tau | =r_\tau^{\frac{n+2}{2}} |\theta_\tau| = \eta^*,$$ we
can also solve for $|\theta_\tau|$. More importantly, we can find the
value of $K$ which parametrizes the orbit  $\mathcal{J}(K)$ passing through
$(r_\tau,\theta_\tau)$, the point of exit from
$\mathcal{S}(\eta^*,r^*)$. There is a small ambiguity in the sign of
$\theta_\tau$, however none of the properties of interest will
depend on this sign, so we take to positive value for
definiteness. 

Letting $K_\tau$ denote the value of $K$ at the exit point, we see
that since $|\theta|$ is small in $\mathcal{S}(\eta^*, r^*)$
\begin{align}  \label{eq:Ktau}
  K_\tau=
\frac{r_\tau}{|\sin(n\theta_\tau)|^{\frac1n}}\approx
    \frac{r_\tau}{|n\theta_\tau |^{\frac1n}}=
    \frac{(r^*)^{\frac{3n+2}{2n}}}{(n\eta^*)^{\frac1n}}
    e^{\frac{3n+2}{2n} \tau}. 
\end{align}
Once outside of $\mathcal{S}(\eta^*,r^*)$ on the orbit
$\mathcal{J}(K_\tau)$, the dynamics is taken to be deterministic.
Hence the probability of ending up on any point on
$\mathcal{J}(K)$ is determined by the probability that $K_\tau=K$.  Therefore we now further explore the distribution of the exit time $\tau$.

In our model dynamics, $\eta_t$ is a one-dimensional (unstable)
Ornstein-Uhlenbeck process. Letting $q_t(\eta,\eta')$ be the
transition density of $\eta_t$ starting from $\eta$ and killed whenever $|\eta_t|=\eta^*$, 
it
is not hard to see that
\begin{align*}
  \P(\tau > t ) = \int_{-\eta^*}^{\eta^*}q_t(\eta,\,\eta')d\eta' \,.
\end{align*}
Standard spectral theory gives the existence of functions
$c_k(\eta,\eta')$ so that  if $0<\lambda_1< \lambda_2 <
\lambda_k< \cdots$ are the eigenvalues of 
\begin{align*}
  (\mathcal{Q} f)(\eta)= -\frac{\sigma^2}{2}\partial_\eta^2 f +
  \big(\tfrac{3n +2}{2}\big) \partial_\eta (\eta f)
\end{align*}
on the domain $[-\eta^*,\eta^*]$ with zero boundary conditions, then
\begin{align*}
  q_t(\eta,\eta') = \sum_{k=1}^\infty c_k(\eta,\eta') e^{-\lambda_k t}\,.
\end{align*}
Hence for $t$ large
\begin{align}
  \label{eq:asyTau}
  \P(\tau > t ) \approx   e^{-\lambda_1 t} \int_{-\eta^*}^{\eta^*}c_k(\eta,\eta')d\eta'\,.
\end{align}
Clearly, $\lambda_1$ is a function of $\eta^*$.  However, it is not
hard to see that as $\eta^* \rightarrow \infty$, $\lambda_1
\rightarrow \tfrac{3n +2}{2}$. 

Another way of seeing this is to note that the solution to
\eqref{eq:approxEtaR} with $\eta_0=0$ can be written as 
\begin{align*}
  \eta_\tau = \sigma e^{\frac{3n +2}{2} \tau} \int_0^\tau  e^{-\frac{3n +2}{2} s}dW_s .   
\end{align*}
Since we are interested in large $r^*$, by the scaling of the equation
we can consider $\sigma$ small for a fixed $r^*$. Since $|\eta_\tau|=\eta^*$, rearranging the above
equation produces $$\frac{\eta_*}{\sigma}   e^{-\frac{3n +2}{2} \tau} =
\Big|\int_0^\tau  e^{-\frac{3n +2}{2} s}dW_s\Big|.$$ When $\sigma$ is small,
$\tau$ will be large and effectively conditionally independent of the bulk of the
Brownian trajectory $\{ W_t : t \in [0,\tau]\}$, particularly  $\{ W_t
: t \in [0,\tau/2]\}$ which makes the dominate contribution for $\tau$
large. Making this leap, for large $\tau>0$ the righthand side of the equation above is approximately Gaussian with mean zero and variance $\approx \frac{1}{3n+2}$. Solving for $\tau$ gives
\begin{align}\label{eq:TauPath}
  \tau \approx \frac{2}{3n +2}\log(\frac{\eta_*}\sigma)
\end{align}
for $\sigma$ small (or equivalently $\eta^*$ big). 

Both this pathwise calculation and the more classical PDE calculation presented first
have their individual merits. We are not married to either.  Since we have chosen to take a
Lyapunov function approach, however, we tend to need expectations of various
quantities for which the PDE methods are well suited. For this reason
and the fact that we wish to avoid developing both modes of calculation in parallel, we
slightly favor the PDE calculations in this note.  Nevertheless, many ideas presented in this packaging were first developed by using a more pathwise reasoning. 


To summarize our model of the dynamics: In equilibrium, trajectories
are injected into the wedge $\mathcal{S}(\eta^*,r^*)$ a certain
rate through the boundary $\{r=r^*\}$. The most likely way for a trajectory to reach a large value is to spend enough time in the wedge
$\mathcal{S}(\eta^*,r^*)$, hence having a large enough exit time $\tau$. The
tails of this exit time are approximately exponential and are given by
\eqref{eq:asyTau}. Once the process exits $\mathcal{S}(\eta^*,r^*)$,
it follows the deterministic trajectory contained in
$\mathcal{J}(K_\tau)$ where the constant $K_\tau$ was set by the exit
point from $\mathcal{S}(\eta^*,r^*)$ using \eqref{eq:Ktau}.
We will now use this caricature of the dynamics to study the decay of
the invariant
measure and spike spacing distribution. 

\subsection{Tails of the invariant measure}\label{sec:tail}
As already mentioned, the maximal distance from the origin on a given orbit 
$\mathcal{J}(K)$ is $K$.  Note that this maximal distance is realized along this orbit precisely when $\theta=\pm \pi/(2n)$ as
$|\sin(n\theta)|^{1/n}=1$ when $\theta=\pm \pi/(2n)$. As previously noted, the most likely way to reach a point far from the  origin in  the principal wedge
$\mathcal{R}$ is to  pass though $\mathcal{S}(\eta^*,r^*)$ and exit  on  the curve $\mathcal{J}(K_\tau)$ with the parameter $K_\tau$ large. In turn, the trajectory upon exit from $\mathcal{S}(\eta^*, r^*)$ then follows a deterministic orbit eventually reaching a distance $K_\tau$ from the origin.

Since our Heuristic model was of the time-changed process obtained through \eqref{eq:timechangeIntro}, we will first study the tails of the invariant measure of the time-changed  process and then undo the time change at the end.  Let $\tilde \mu$ denote the 
stationary measure of the time-changed process, $\P_{\tilde\mu}$ the probability 
measure of the time-changed Markov process with initial distribution 
$\tilde{\mu}$, and $\P_{(r,\theta)}$ the probability measure of the 
time-changed Markov process with initial condition $(r,\theta)$.

In light of the heuristic model,  for  for $R>0$ large 
\begin{align*}
  \P_{\tilde\mu}(|z| \geq R) \approx c\P_{(r^*,0)}(K_\tau > R) 
\end{align*}
where $c>0$ is a positive constant capturing the flux into $\mathcal{S}(\eta^*,r^*)$ around $r^*$ in equilibrium. 
Using  the relation \eqref{eq:Ktau}, we find that for $R>0$ large 
\begin{align*}
  \P_{\tilde\mu}(|z| \geq R) \approx c \P_{(r^*,0)}( \tau > \tfrac{2n}{3n+2}
  \log \tfrac{R}{R_0}) 
\end{align*}
for some positive constant $R_0$. Now pick 
$\eta^* >0$ large enough so that $\lambda_1 \approx \frac{3n+2}{2}$.  Thus for $R>0$ large enough, \eqref{eq:asyTau} gives 
\begin{align*}
  \P_{\tilde\mu}(|z| \geq R) \approx cR^{-\tfrac{2n\lambda_1}{3n+2}}
  \approx cR^{-n}\,.
\end{align*}
One can also obtain the fact that $\P_{(r^*,0)}(K_\tau \geq R) \approx
cR^{-n}$ by combining the calculation used to obtain
\eqref{eq:TauPath} and \eqref{eq:Ktau}.

Letting $\rho(r,\theta)$ denote density of $\tilde \mu$, for 
$\theta \neq 0$ we have 
\begin{align*}
 - \frac{c}{R^{n+1}}\approx\frac{\partial\ }{\partial R}    \P_{\tilde\mu}(|z| \geq R)  =
  \frac{\partial\ }{\partial R}  \int_R^\infty \int_0^{2\pi}
  \rho(r,\theta) 
  \,d\theta\, r\, dr = - R \int_0^{2\pi}
  \rho(R,\theta) \, d\theta\,. 
\end{align*}
From this we conclude that 
\begin{align*}
  \P_{\tilde\mu}( |z| \in dR ) \approx \frac{c}{R^{n+2}} \,dR 
\end{align*}
for large $R$. Hence if $\mu$ denotes the stationary measure of the original system \eqref{eqn:zn} without 
the time change we see that 
\begin{align*}
   \P_{\mu}( |z| \in dR ) \approx \frac{c}{R^{2n+2}}\, dR 
\end{align*}
which agrees with the scaling argument of Section~\ref{sec:scalingargumentp}.  

\subsection{Spacing of the excursions}
\label{sec:spacing}
We now investigate the distribution of the time between ``spikes" of size $R>0$, as illustrated in Figure \ref{fig:timeseries}.

It is reasonable to assume that for a large but fixed value of $r^*>0$,
trajectories with high probability spend a random amount of time in
$\{r \leq r^*\}$
before passing through $\mathcal{S}(\eta^*,r^*)$ to hit larger radial values. For any $R \gg  2r^*$, define the following sequence of stopping times:
$S_0=0$ and for $i\geq 1$ set
\begin{align*}
  T_{i}=\inf\{ t \geq  S_{i-1} : r_t \geq R\} \qquad\text{and}\qquad
  S_i=\inf\{ t \geq T_i : r_t \leq r^*\}\,.
\end{align*}
We are interested in the distribution of the time between
``spikes"~$T_{i+1}- T_{i}$. We will see that $T_{i+1}-T_i$ is
distributed as a compound-geometric with geometric parameter that
scales like $R^{-n}$. Hence we expect
\begin{align*}
  \E \big(T_{i+1}- T_{i}\big) = c R^n
\end{align*}
for some positive constant $c$.  

We begin by defining successive exit times from the set $\{r <
2r^*\}$ with an intervening return to the set  $\{r \leq r^*\}$:  Let $s_0=0$ and for $i\geq 1$
\begin{align*}
  t_{i}=\inf\{ t \geq s_{i-1} : r_t \geq 2r^*\} \qquad\text{and}\qquad
  s_i=\inf\{ t \geq t_i : r_t \leq r^*\}\,.
\end{align*}
With high probability, each exit from $\{r < R\}$ happens through the
region $\mathcal{S}(\eta^*,2r^*)$.  In turn, the locations of the
exits through the boundary of $\mathcal{S}(\eta^*,2r^*)$ determine
which of the orbits $\mathcal{J}(K)$ the dynamics follows. As in the
discussion at the start of this section, let $\tau$ denote the exit
time from $\mathcal{S}(\eta^*,2r^*)$ and let $K_\tau$ be the value of
the constant $K$ used to index the orbit $\mathcal{J}(K)$ when the
dynamics leaves $\mathcal{S}(\eta^*,2 r^*)$.
\begin{figure}
\centering
\hspace{-.3in}\includegraphics[width=5in]{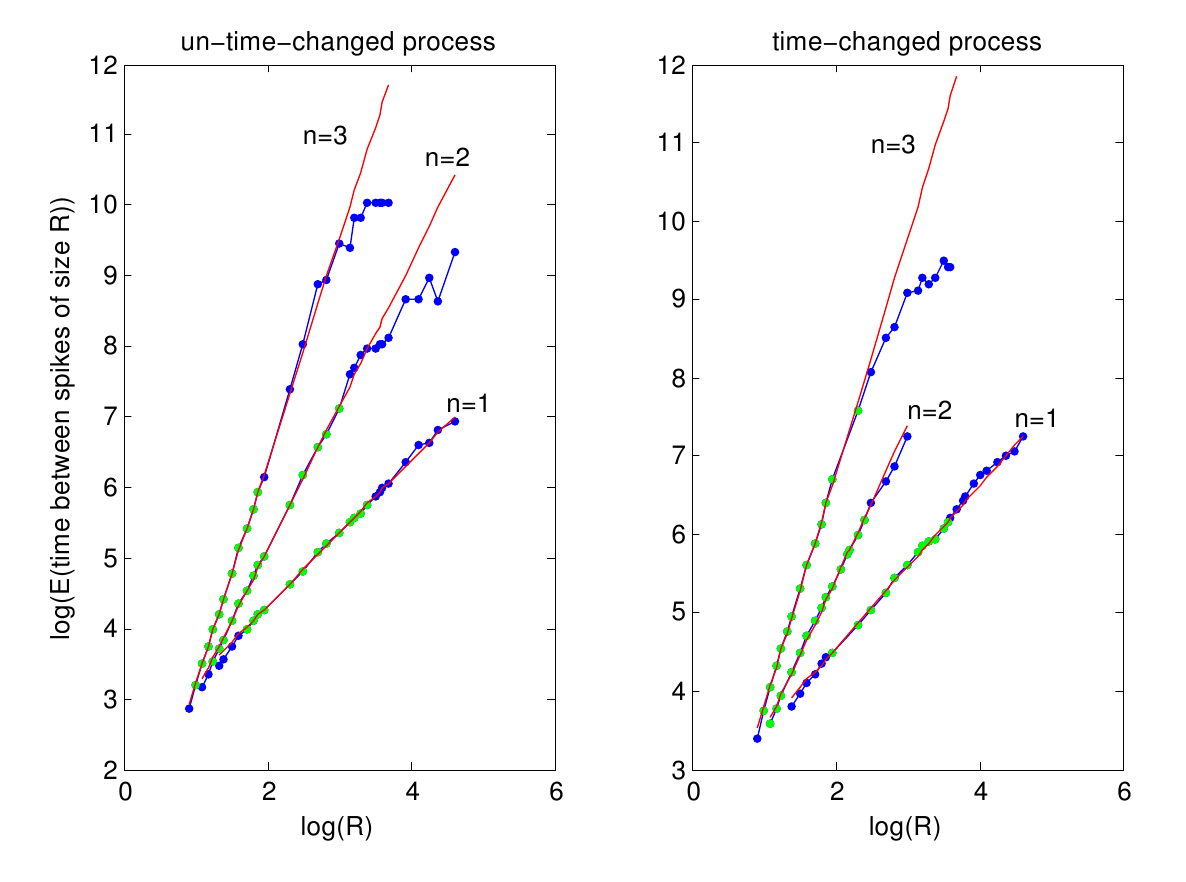}
\begin{center}
\begin{tabular}{|r|c|l|}
  \hline
  $n$ & Slope (no time change)  & Slope (time change) \\
\hline
  1 & 1.031 & 1.035\\
  \hline
  2 & 2.034 & 1.960\\
  \hline 
  3 & 3.175 & 3.001\\
  \hline
\end{tabular}
\end{center}\caption{Simulation results of $\log\E[T_{j+1}-T_j]$ versus $\log(R)$ without the time change (left) and with the time change (right).  Points sufficiently far from $0$ but with enough data (in green) are fitted with a least squares approximation.  Slopes of each line in either case are given in the table above and are as predicted.}   
\label{fig:meanspike}
\end{figure}

If we define
$p_R=\P(\text{height of spike is $\geq$ R})$ on a given entry
into $\mathcal{S}(\eta^*,2r^*)$ from the $\{r=2r^*\}$ boundary, then using the same logic as in the previous section
\begin{align*}
  p_R = \P( K_\tau >R) \approx cR^{-\tfrac{2n\lambda_1}{3n+2}} \approx
  cR^{-n}.  
\end{align*}
It is reasonable to assume that the $K_\tau^{(i)}$ associated to $i$th
entry into $\mathcal{S}(\eta^*,2r^*)$ is independent of the
$K_\tau^{(j)}$ with $j\neq i$. Letting $n_R^{(i)}$ denote the number
of excursions to level $2r^*$ needed to produce an spike greater than $R$ after time
$T_{i-1}$, we observe that under the independence assumption $n_R^{(i)}$ is geometric with parameter 
$p_R$.  Hence, $\E n_R^{(i)} = 1/p_R$.  Since 
\begin{align*}
  T_{j+1} - T_j = \sum_{i=N_j}^{N_{j+1}} (t_{i}  -t_{i-1})\qquad \text{where}
  \qquad N_j= N_{j-1} + n_R^{(j)}
\end{align*}
we see that $T_{j+1} - T_j$ is a compound-geometric random variable. Hence we have 
\begin{align*}
 \E\big[ T_{j+1} - T_j\big]\approx\E\big[ n_R^{(j)}\big] \E \big[  t_{2}-
 t_{1}\big]\approx \frac{\E \big[  t_{2}-
 t_{1}\big]}{p_R} \approx  c R^{n}
\end{align*}
Hence we expect the average spacing between peaks of spikes greater
than $R$ should grow like $R^{n}$.  Consult Figure \ref{fig:meanspike}
for numerical results which agree with this heuristic prediction.

All of the analysis above holds equally well even if the dynamics has
not been time-changed using \eqref{eq:timechangeIntro}. The critical quantity is the location at which
the process exits from $\mathcal{S}(\eta^*,2r^*)$ and this is
unchanged by the time-change. The time-change only affects the time
between entrances into $\mathcal{S}(\eta^*,2r^*)$ and
the time it takes to traverse the spike excursion out to level $R$
(see Figure \ref{fig:meanspike}), in particular it has no effect on
the probability $p_R$. The numerical  confirm that the predictions
still hold.

\section{Main Results: Ergodicity, Mixing, and the Behavior of the
  Stationary Measure at Infinity}
\label{mainResults}

Although we have thus far only discussed equation \eqref{eqn:genpoly}, we will see that our main results, to be stated in this section, hold for more general complex-valued SDEs.  In particular, our analysis can tolerate more general lower-order terms in the drift.  Therefore, throughout the remainder of this paper and the sequel~\cite{HerzogMattingly2013II} we assume that the complex-valued process $z_t$ satisfies more generally the following SDE
\begin{align}
\label{eqn:polyZglot}
dz_t = [ a_{n+1}   z_t^{n+1} + F(z_t, \bar{z}_t) ]\,dt + \sigma \, dB_t
\end{align}
where $a_{n+1}\in \C\setminus \{ 0\}$, $n\geq 1$, $\sigma>0$, $B_t=B_t^1 + i B_t^2$ is a complex Brownian motion and $F(z, \bar{z})$ is a complex polynomial in $(z, \bar{z})$ with $F(z, \bar{z}) =\mathcal{O}(|z|^n)$ as $|z|\rightarrow \infty$.  Notice that equation \eqref{eqn:genpoly} is the special case of equation \eqref{eqn:polyZglot} where $F(z, \bar{z})\equiv F(z)$ is a complex polynomial in the variable $z$ only with $\text{degree}(F)\leq n$.

The global-in-time existence of the Markov process induced by
\eqref{eqn:polyZglot} is neither obvious nor certain given the unstable
nature of the underlying deterministic dynamics.  Consequently, even
if it is shown that the process does not explode in finite time, the
existence of an invariant measure is still in question.  Assuming,
however, both issues can be settled, the formal asymptotic
calculations of Section \ref{sec:formalInvM} suggest that the
probability density function of the invariant probability measure has a certain
polynomial decay rate at infinity. The following result, one of the
principal rigorous results of this article and the sequel
\cite{HerzogMattingly2013II}, shows that these formal computations are essentially
correct.

\begin{theorem}\label{thm:decayStMeasure} The Markov process defined
  by \eqref{eqn:polyZglot} is non-explosive and possesses a unique
  stationary measure $\mu$.  In addition, $\mu$ satisfies:
\begin{align*}
\int_\C (1+|z|)^\gamma \, d\mu(z) < \infty \,\,\text{ if and only if }\,\, \gamma < 2n.  
\end{align*}
Furthermore, $\mu$ is
ergodic and has a probability density function $\rho$ with respect to
Lebesgue measure on $\R^2$ which is smooth and everywhere positive.
 \end{theorem}

 Given the existence and uniqueness of the stationary measure, it is
 also natural to explore if initial distributions converge
 to it and, if so, to determine the rate of convergence. To see what
 happens in the present context, given any measurable function $w\colon \C
 \rightarrow [1,\infty)$, let $\mathcal{M}_w(\C)$ denote the set of probability measures $\nu$ on $\C$ satisfying $w\in L^1(\nu)$  
  and define the weighted total variation metric
 $d_w$ on $\mathcal{M}_w(\C)$ by
\begin{align*}
  \d_w(\nu_1,\nu_2)=\sup_{\substack{\phi:\C \rightarrow \R
      \\|\phi(z)|\leq w(z)} }\bigg[\int \phi(z)\, \nu_1(dz)- \int
  \phi(z) \, \nu_2(dz)\bigg].  
\end{align*}

\begin{theorem}\label{thm:convergence} Let $P_t$ denote the Markov
  semi-group corresponding to \eqref{eqn:polyZglot} and let $\alpha \in (0, n)$ be arbitrary.   Then there exists a function $\Psi \colon \C
  \rightarrow [0,\infty)$ and positive constants $c, d, K$ such that 
\begin{align*}
  c |z|^\alpha \leq \Psi(z) \leq d |z|^{\alpha + \frac{n}{2}+1}  
\end{align*}
for all $|z| \geq K$ and such that if $w(z)=1+\beta \Psi(z)$ for some $\beta >0$, then ${\nu P_t \in \mathcal{M}_w(\C)}$ for all $t>0$ and any probability measure $\nu$ on $\C$.  Moreover, with the same choice of $w$, there exist positive constants $C, \gamma$ such that for any two probability measures $\nu_1, \nu_2$ on $\C$ and any $t\geq 1$
  \begin{align*}
   d_w(\nu_1 P_t,\nu_2 P_t) \leq C e^{-\gamma t} \| \nu_1-\nu_2\|_{TV}.  
 \end{align*} 
\end{theorem}
\begin{remark}
In the statement of Theorem~\ref{thm:convergence}, $\|\nu_1-\nu_2\|_{TV}$ denotes the
total variation distance between the measures $\nu_1$ and $\nu_2$.  Note that $\|\nu_1-\nu_2\|_{TV} = d_w(\nu_1,\nu_2)$ when $w\equiv 1$.  
\end{remark}

\begin{remark}
It is interesting to note that all results stated in Theorem~\ref{thm:convergence} hold for a fixed but arbitrary noise intensity $\sigma >0$.  It is therefore natural to wonder how exponential convergence to the invariant probability measure depends on the parameter $\sigma>0$.  Although we will not be able to extract this dependence rigorously in the general system \eqref{eqn:polyZglot}, we will be able to in the monomial case \eqref{eqn:zn} in the total variation distance.

To see the dependence in the monomial case, let $z(t, \sigma, z_0)$ denote the solution of \eqref{eqn:zn} with initial condition $z_0\in \C$ and noise intensity $\sigma>0$.  Here, we naturally emphasize the dependence of the solution $z(t, \sigma, z_0)$ on the noise intensity $\sigma >0$ and the initial condition $z_0$ because we will rescale both space and time, making use of the homogeneity of the drift term $z_t^{n+1}$.  Now observe that 
\begin{align*}
\tilde{z}(t,\sigma, z_0) : = \sigma^{l} z(t \sigma^{l'}, 1, \sigma^{-l} z_0), \,\,\, l= \frac{2}{n+2}, \,\,\, l'= \frac{2n}{n+2}
\end{align*}
has the same distribution as $z(t, \sigma, z_0)$ for all times $t\geq 0$.  Define a probability measure $\pi^\sigma$ on the Borel subsets $A$ of $\C$ by $\pi^\sigma(A)= \pi^{1}(\sigma^{-l}A)$ where $\pi^{1}$ denotes the invariant probability measure corresponding to the process \eqref{eqn:zn} with noise intensity $1$ and suppose that $C, \gamma$ and $w$ in the statement of Theorem~\ref{thm:convergence} correspond to the case \eqref{eqn:zn} when the noise intensity is $1$.  Then we have by Theorem~\ref{thm:convergence} 
\begin{align*}
\| \P(z(t, \sigma, z_0) \in \cdot\,) - \pi^\sigma(\,\cdot\,) \|_{TV} &= \|\P(z(t \sigma^{l'}, 1, \sigma^{-l} z_0)\in \cdot \,) - \pi^{1}(\, \cdot \,) \|_{TV}\\
& \leq C e^{-\gamma  \sigma^{l'} t} \|\delta_{\sigma^{-l} z}- \pi^{1}\|_{TV}\\
& \leq  C e^{-\gamma  \sigma^{l'}t}= Ce^{-\gamma_\sigma t}    
\end{align*}
where $\gamma_\sigma:=\gamma \sigma^{l'}$.  In particular, $\pi^\sigma$ is the unique invariant probability measure corresponding to the process \eqref{eqn:zn} with noise intensity $\sigma$ and, moreover, when measuring the convergence to equilibrium in the total variation distance, the only constant that depends on $\sigma$ is $\gamma_\sigma$, and it is related to the constant $\gamma$ corresponding to the process $z(t, 1, z_0)$ via $\gamma_\sigma= \gamma \sigma^{l'}$.

For the general system \eqref{eqn:polyZglot}, this simple argument fails since the drift is no longer homogeneous under radial scalings.  In particular, in trying to replicate this argument, one cannot rescale time and/or space to arrive at a process independent of $\sigma$ like we were able to above.  Therefore, uncovering the dependence on $\sigma$ is more nuanced for the general system \eqref{eqn:polyZglot}.  Although we will not do it here or in the sequel, however, one could obtain bounds on $C$ and $\gamma$ in the statement of Theorem~\ref{thm:convergence} in terms of $\sigma$ by carefully tracking the dependence of our Lyapunov functions on $\sigma$ and applying the results of \cite{HM_11}.       
\end{remark}

\begin{remark} \label{fatAndFast}
  As mentioned in the introduction, an interesting feature of the system \eqref{eqn:genpoly} is that the equilibrium attracts all initial conditions exponentially fast yet the
  equilibrium density only decays polynomially at infinity (and
  not exponentially as one might expect in a system with additive noise).
  Due to the structure of the invariant density, this cannot happen in a gradient system with additive noise under a nominal uniform growth assumption on the potential. To see why, consider the following equation on $\R^k$
  \begin{align*}
    dX_t =& -\nabla V(X_t) dt + dW_t 
  \end{align*}
  where $W_t=(W_t^1, \ldots, W_t^k)$ is a $k$-dimensional Brownian motion and the potential $V\in C^2$ satisfies $V(x) \geq c|x|$ for $|x|\geq K$ for some $K>0$.  Since the invariant density is proportional to $e^{- c V(x)}$ for some $c>0$, we see that the invariant density decays exponentially at infinity.  If, for example, $V(x)=p \log(|x|)$ for some constant $p>0$ sufficently large (depending on the dimension), then the system posses a unique invariant density decays polynomially at infinity.  However, this implies that the drift in the above SDE is of the form $-p x /|x|^{2}$. Such drifts are known to have slow return times to the ``center" of the phase space and this leads to sub-exponential convergence to equilibrium. See \cite{KlokovVeretennikov2004} for a general example and \cite{Hairer2009,HM_09} for some further interesting examples.
\end{remark}


\section{Consequences of Lyapunov Structure and Implications for the
  Invariant Measure}
\label{sec:conseq_lyap}

Most of the results in Section~\ref{mainResults} turn on the existence
of a certain type of Lyapunov function corresponding to the dynamics \eqref{eqn:polyZglot}.  Because we require the additional flexibility, we make use of a slightly more general notion of a Lyapunov function than usually employed in the context of diffusion processes.  In this section, therefore, we will define what we mean by a Lyapunov function and give some results which follow from its existence.  Because it is simpler, in this section we will work more generally within the context of a time-homogeneous
It\^{o} diffusion $\xi_t$ on $\R^k$ with smooth $(C^\infty)$ coefficients.  Also, because it is not clear that $\xi_t$ exists for all finite times, we make use of the stopping times $\tau_n = \inf\{t>0 \, : \, |\xi_t|\geq n \}$, $n\in \N$.    
\begin{definition}
\label{def:Lyapunov}
Let $\Psi, \Phi: \R^k\rightarrow [0, \infty)$ be continuous. We call $(\Psi, \Phi)$
a \textsc{Lyapunov pair corresponding to} $\xi_t$ if:
  \begin{enumerate}
  \item $\Psi(\xi)\wedge \Phi(\xi) \rightarrow \infty$ as
    $|\xi|\rightarrow \infty$;
  \item\label{lyac:b}  There exists a locally bounded and measurable function ${g: \R^k \rightarrow \R}$ such that the following equality holds for all $\xi_0 \in \R^k$, $n\in \N$ and all bounded stopping times $\upsilon$:
  \begin{align*}
   \E_{\xi_0} \Psi(\xi_{\upsilon \wedge \tau_n}) = \Psi(\xi_0) +
   \E_{\xi_0}\int_0^{\upsilon \wedge \tau_n} g(\xi_s) \, ds + \text{Flux}(\xi_0, \upsilon, n)  
 \end{align*}
 where $\text{Flux}(\xi_0, \upsilon, n)\in (-\infty, 0]$ and
 $\text{Flux}(\xi_0, t, l) \leq \text{Flux}(\xi_0, s, n)$ for all
 $0\leq s\leq t$, $n\leq l$, $\xi_0 \in \R^k$.    The meaning of the
 flux term and the reason for its name is discussed in Remark~\ref{rem:C2}.

\item \label{lyac:c}There exist constants $m, b>0$ such that for all $\xi \in \R^k$  
  \begin{align*}
  g(\xi) \leq - m \Phi(\xi) + b.    
  \end{align*} 
    \end{enumerate}
  The function $\Psi$ in a Lyapunov pair $(\Psi, \Phi)$ is called a
  \textsc{Lyapunov function.  }
  \end{definition}

 \begin{remark}\label{rem:C2}
One usually requires $\Psi$ in Definition \ref{def:Lyapunov} to be globally $C^2$, in which case conditions \ref{lyac:b}) and \ref{lyac:c}) are replaced by the global bound
\begin{align*}
\mathscr{L} \Psi \leq - m \Phi + b  
\end{align*}
where $\mathscr{L}$ denotes the infinitesimal generator of $\xi_t$.
Since our Lyapunov function $\Psi$ will only be globally continuous,
we will need the more general formulation given above in order to make
use of an extension of Tanaka's formula due to Peskir
\cite{Peskir_07}.  A consequence of this extended formula is that one
is now permitted to take the It\^{o} differential of $\Psi(\xi_t)$
above so long as $\Psi$ is continuous on $\R^k$ and $\Psi$ is $C^2$
everywhere except on a finite collection of non-intersecting,
sufficiently smooth $(k-1)$-dimensional surfaces.  In such cases, if
$\Psi$ is not globally $C^2$, then the differential $d\Psi(\xi_t)$
contains a finite number of local time contributions, each of which
corresponds to the local time of the flux (in the normal direction) of
$\Psi$ on the boundary of a given surface where $\Psi$ is not $C^2$.
Hence, this is precisely why the term $\text{Flux}(\xi_0, \upsilon,
n)$ appears in the formula above.  For further discussions on this
topic, we refer the reader to \cite{Peskir_07} (see also
Section~\ref{subsec:Bas_structI} of this paper and
Section~\ref{II-sec:peskir} of Part II \cite{HerzogMattingly2013II}).      
\end{remark}

\begin{remark} 
Another way to think about the generalized notion of a Lyapunov function given here is that it affords the structure needed to work with certain types of weak (as opposed to classical) sub-solutions of the PDE $\mathscr{L}\Psi = -m \Phi +b$.  The flux term naturally arises when integrating by parts to show that $\Psi$ is indeed a weak sub-solution of $\mathscr{L}\Psi=-m \Phi + b$ and, for $\Psi$ to be a sub-solution, we need the flux term to be $\leq 0$.  
  \end{remark}

\begin{remark}
To keep this section concise, most of the proofs in this section will be given in Appendix~\ref{appendixA}.  Although each proof is a somewhat natural extension of results in the references  \cite{AKM_12, HM_09, HM_11, Has_12, MT_93}, special care is taken precisely because our Lyapunov function will not be globally $C^2$.  
\end{remark} 
 
The first result we state gives the basic consequences of the existence of a Lyapunov pair $(\Psi, \Phi)$.  
 
\begin{lemma}\label{l:existInvM}
If $\xi_t$ possesses a Lyapunov
  pair $(\Psi, \Phi)$, then the following conclusions hold:
  \begin{enumerate}
  \item \label{l:existInvMa}$\xi_t$ is non-explosive; that is, if $\tau_\infty = \lim_{n\rightarrow \infty} \tau_n$, then 
  \begin{align*}
  \P_{\xi_0} [\tau_\infty = \infty ]=1
  \end{align*}
  for all $\xi_0 \in \R^k$.  In particular, for every initial condition ${\xi_0 \in \R^k}$, $\xi_t$ is well-defined for all finite times $t\geq 0$ almost surely.  
  \item \label{l:existInvMb}$\xi_t$ has an invariant probability measure $\pi$ satisfying
\begin{align*}  
\quad \int_{\R^k} \Phi(\xi) \pi(d\xi) < \infty.  
    \end{align*}  
  \end{enumerate}
\end{lemma}
\begin{proof}
See Appendix~\ref{appendixA}.         
\end{proof}

Motivated by the setup in the previous section, if $w: \R^k \rightarrow [1, \infty)$ is measurable, we let $\mathcal{M}_w(\R^k)$ denote the set of probability measures $\nu$ on $\R^k$ satisfying ${w\in L^1(\nu)}$ and define a metric $d_w$ on $\mathcal{M}_w(\R^k)$ by  
\begin{align*}
d_w(\nu_1, \nu_2) = \sup_{|\varphi| \leq w}\bigg[ \int_{\R^k} \varphi(\xi) \nu_1(d\xi)-\int_{\R^k}\varphi(\xi) \nu_2(d\xi)\bigg]
\end{align*}
where the supremum is taken over $\varphi: \R^k \rightarrow \R$ measurable satisfying the global bound $|\varphi|\leq w$.  We also have the following results which give uniqueness of the invariant measure and characterize the convergence rate of the process $\xi_t$ to this equilibrium.

\begin{theorem}\label{LyapConverge}  Suppose that $\xi_t$ has a uniformly elliptic diffusion matrix and a Lyapunov pair $(\Psi, \Phi)$, and let $\mathscr{P}_t$ denote the Markov semigroup corresponding to $\xi_t$.  Then the following conclusions also hold:
 \begin{enumerate}
  \item\label{LyapunovConvergea} $\xi_t$ possesses a unique invariant probability measure $\pi$.  Moreover, $\pi$ is ergodic, satisfies 
  \begin{align*}
  \int_{\R^k} \Phi(\xi) \, \pi(d\xi) < \infty,  
  \end{align*}
   and has a smooth and everywhere positive density with respect to Lebesgue measure on $\R^k$.
  \item\label{contraction} If $\Phi=\Psi$ and $w(\xi)=1+\beta \Psi(\xi)$ for some $\beta >0$, then there exist constants $C, \eta>0$ such that 
    \begin{align*}
      d_w(\nu_1 \mathscr{P}_t,\nu_2 \mathscr{P}_t) \leq C e^{-\eta t} d_w(\nu_1,\nu_2)
    \end{align*}
    for all times $t\geq 0$ and all $\nu_1, \nu_2 \in \mathcal{M}_w(\R^k)$.     
  \item\label{supcontraction} If $\Phi=\Psi^{1+\delta}$ for some $\delta>0$ and $w(\xi)=1+\beta \Psi(\xi)$ for some $\beta >0$, then the conclusion in part \ref{contraction}) also holds.  Moreover, $\nu \mathscr{P}_t \in \mathcal{M}_w(\R^k)$ for $t>0$ and any probability measure $\nu$ on $\R^k$ and there exist positive constants $\tilde{C}, \tilde{\eta}$ such that 
  \begin{align*}
  d_w(\nu_1 \mathscr{P}_t, \nu_2 \mathscr{P}_t) \leq \tilde{C} e^{-\tilde{\eta} t} \| \nu_1 - \nu_2 \|_{TV}
  \end{align*} 
  for all $t\geq 1$ and all probability measures $\nu_1, \nu_2$ on $\C$.  
  \end{enumerate}
\end{theorem} 
\begin{remark}
  The uniform ellipticity assumption is
  not needed for many of the results to hold. However, it simplifies
  our discourse significantly. See \cite{AKM_12,GHW_11} for examples
  considering degenerate noise in setting similar to this paper. 
\end{remark}

\begin{proof}[Proof of Theorem~\ref{LyapConverge}]
The existence in part \ref{LyapunovConvergea}) follows from the
  previous lemma.  Uniqueness of $\pi$ and the existence of a smooth
  and everywhere positive density are well-known consequences of
  uniform ellipticity of the diffusion matrix and the fact that $\xi_t$ satisfies an SDE with smooth coefficients.  The property
  \begin{align*}
  \int_{\R^k} \Phi(\xi) \, \pi(d \xi) <\infty
  \end{align*}  
  follows by uniqueness of the invariant probability measure and by Lemma \ref{l:existInvM} \ref{l:existInvMb}).  Parts \ref{contraction}) and \ref{supcontraction}) of the result are proven in Appendix~\ref{appendixA}.    
\end{proof}

As we saw in the previous lemma, if $\xi_t$ possesses a Lyapunov pair
of the form $(\Psi, \Psi^{1+\delta})$ for some $\delta >0$, then the
standard geometric ergodicity bound given in part \ref{contraction})
can be improved in the sense that the right-hand side no longer
depends on the initial state for $t\geq 1$.  This is also reflected in the following
theorem, as return times to large compact sets are small and independent
of where the process $\xi_t$ starts.

\begin{theorem}
\label{thm:entrance_time}
Suppose that $\xi_t$ has a Lyapunov pair $(\Psi, \Psi^{1+\delta})$ for some $\delta >0$ and that the diffusion matrix corresponding to $\xi_t$ is uniformly elliptic. Define ${\upsilon_\gamma =\inf\{t>0 \, : \, |\xi_t| \leq \gamma \}}$ for $\gamma>0$.  Then for each $\gamma >0$ sufficiently large
\begin{align*}
\inf_{\xi_0 \in \R^k}\P_{\xi_0} [\upsilon_\gamma<\infty]=1.   
\end{align*}
Moreover, for each $t, \epsilon>0$ there exists $\gamma >0$ large enough so that 
\begin{align*}
\sup_{\xi_0 \in \R^k}\P_{\xi_0}[\upsilon_\gamma \geq t ]\leq \epsilon.    
\end{align*}
\end{theorem}

\begin{proof}
See Appendix~\ref{appendixA}.  
\end{proof}


\section{Proving the Main Results: An Outline} 
\label{sec:mainResultsOutline}

We now give two results which, when combined with the results of the previous section, will yield all of main results of this paper and the sequel \cite{HerzogMattingly2013II}.  Theorem~\ref{thm:lyapfunmadness} below will provide the needed Lyapunov pair, allowing us to apply the results of Section~\ref{sec:conseq_lyap}. And, Theorem~\ref{thm:limit_bound} below will give the required lower bound on the density of the invariant probability measure whose existence will be now ensured by Theorem~\ref{thm:lyapfunmadness}.

First notice we may assume, without loss of generality, that $a_{n+1}=1$ in \eqref{eqn:polyZglot}.  Indeed, the system can be rescaled and rotated so that it is one and, since any rotation of $B_t$ is also a complex Brownian motion, the resulting system will be of the form \eqref{eqn:polyZglot} but with $a_{n+1}=1$.  Hence for the remainder of the paper we will assume that  \eqref{eqn:polyZglot} takes the form 
\begin{align}
\label{eqn:polyZ}
dz_t = [z_t^{n+1} + F(z_t, \bar{z}_t)] \, dt + \sigma \, dB_t
\end{align}
where $n\geq 1$, $F$, $\sigma$ and $B_t$ are as in equation \eqref{eqn:polyZglot}. By a simple change of variables, all results translate easily to the general system \eqref{eqn:polyZglot}.

\begin{theorem}
\label{thm:lyapfunmadness}
For each $\gamma \in (n, 2n)$ and $\delta=\delta_\gamma >0$
sufficiently small, there exist a function $\Psi$ so that $(\Psi,
\Psi^{1+\delta})$ and $(\Psi, |z|^\gamma)$ are  Lyapunov pairs  corresponding to the
dynamics \eqref{eqn:polyZ}.  Moreover, $\Psi$ satisfies the following bounds for $|z| \geq K$
\begin{align*}
 c|z|^{\gamma-n}\leq \Psi(z) \leq d|z|^{\gamma-n +\frac{n}{2}+1}
\end{align*}   
for some positive constants $c,d, K$.    
\end{theorem}

\begin{remark}
We will see that for a ``very large" set $X\subset \C\cap \{ |z| \geq K \}$, the bound 
\begin{align*}
C |z|^{\gamma-n}\leq \Psi(z) \leq D |z|^{\gamma-n}
\end{align*} 
holds for all $z\in X$ for some $C,D>0$.  The only region $Y$ in which
$\Psi$ grows faster than a constant times $|z|^{\gamma-n}$ satisfies
the property
\begin{align*}
\lim_{R\rightarrow \infty}\lambda(Y \cap \{|z|>R \}) = 0
\end{align*}   
where $\lambda$ denotes Lebesgue measure on $\C$.  As will be apparent
later, any increase in growth in $\Psi$ is exactly compensated by the
decrease in the measure of the set $Y \cap \{|z|>R \}$ as
$R\rightarrow \infty$.  Although we did not state it this way in
Theorem \ref{thm:lyapfunmadness}, we could have also chosen the second
function in the pair $(\Psi, |z|^\gamma)$ to have this property.
\end{remark}

Translating back to the general system \eqref{eqn:polyZglot}, notice by
combining 
Lemma \ref{l:existInvM} and Theorem \ref{LyapConverge}, we see that Theorem
\ref{thm:lyapfunmadness} implies Theorem \ref{thm:convergence} as well
as all results of Theorem \ref{thm:decayStMeasure} except
\begin{align}
\int_\C (1+|z|)^\gamma \, d \mu (z) = \infty  \,\, \text{ whenever } \,\, \gamma \geq 2 n.
\end{align}  
To prove this last point, we will show the following stronger result.  

\begin{theorem}
\label{thm:limit_bound}
Let $\rho(x,y)$ denote the invariant probability density function of
\eqref{eqn:polyZ} with respect to Lebesgue measure on $\R^2$.  Then there exist positive constants $c, K$ such that   
\begin{align}
\label{eqn:limit_bound}
|(x,y)|^{2n+2} \rho(x, y) \geq c\,\, \text{ for } \, \, |(x,y)|\geq K.    
\end{align}
\end{theorem}

In this paper, our focus is to prove Theorem \ref{thm:lyapfunmadness}
under the following simplifying assumption:
\begin{assumption}
\label{assumption:nlot} In equation \eqref{eqn:polyZ}, either $F$ is a constant function or  
$$F(z,\bar{z})=\mathcal{O}(|z|^{\lfloor \frac{n}{2} \rfloor -1})\text{ as }|z|\rightarrow \infty$$  
\end{assumption}
\noindent for $n\geq 2$.  We do this in order to highlight the general procedure used to yield our
Lyapunov pairs and to avoid substantial complexities
created by the presence of large lower-order terms.  The full proof of
Theorem \ref{thm:lyapfunmadness} and the proof of Theorem
\ref{thm:limit_bound} are given in Part II \cite{HerzogMattingly2013II}.


\section{Building Lyapunov Functions: The Key Initial Steps}
\label{sec:GenOver}

In this section, we make some beginning observations which will help us get started with constructing a Lyapunov function $\Psi$ for the system \eqref{eqn:polyZ}.  Everything done in this section applies to equation \eqref{eqn:polyZ} even if we do not employ Assumption~\ref{assumption:nlot}.   

\subsection{The Coordinate and Time Changes}
\label{sec:timeChange}
When building $\Psi$, it is paramount that one first pick a convenient coordinate system in which to work.  For equation \eqref{eqn:polyZ}, there are at least three choices: standard Euclidean coordinates $(x,y)$, the two-dimensional complex system $(z, \bar{z})$, and polar coordinates $(r, \theta)$.  Notice, however, since stability of the process \eqref{eqn:polyZ} (or any $\R^k$-valued process for that matter) is completely determined by the distribution of the radial component $r$, the polar system $(r, \theta)$ is arguably most natural. 

\begin{remark}
Even though $\Psi$ will be constructed using polar coordinates, all desired bounds obtained using $\Psi$ in $(r, \theta)$ will translate easily back to bounds in either $(z, \bar{z})$ or $(x,y)$ coordinates, as we will set $\Psi(r, \theta) \equiv 0$ on $\{ r\leq 1\}$.  In particular, since the the process $(x_t, y_t)$, $x_t =\text{Re}(z_t)$ and $y_t = \text{Im}(z_t)$ where $z_t$ solves \eqref{eqn:polyZ}, is an It\^{o} diffusion with $C^\infty$ coefficients and has a uniformly elliptic diffusion matrix, we can indeed apply all results in Section \ref{sec:conseq_lyap}.       
\end{remark}

In light of the above, observe that the generator of the Markov
process defined by \eqref{eqn:polyZ} has the following form when
written in the variables $(r, \theta)$:
\begin{align}
\label{eqn:gen_pol}
\mcl = r^{n+1} \cos(n\theta) \partial_r + r^n
\sin(n\theta) \partial_{\theta} + \frac{\sigma^2}{2 }\partial_{r}^2 +
\frac{\sigma^2}{2r^2}\partial_{\theta}^2 + r^n P(r, \theta)\partial_r
+ r^n Q(r, \theta) \partial_{\theta}
\end{align} 
where $P(r, \theta) = \sum_{k=0}^{n+2} r^{k-n-2} f_k(\theta) $ and $Q(r,
\theta)=\sum_{k=0}^{n+1} r^{k-n-2} g_k(\theta)$ for some collection of
smooth real-valued  functions $f_k$ and $g_k$ which are $2\pi$-periodic.  In order to encapsulate all terms in the generator, we certainly do not need the $k=0$ terms in $P$ and $Q$.  However when proving Theorem \ref{thm:limit_bound}, we will need certain
stability properties of a diffusion process related to the formal
adjoint $\mathcal{L}^*$.  Because there is one additional term in the
generator of this diffusion, we will construct the appropriate
Lyapunov pairs assuming  the slightly more general form of $P$ and $Q$ above.


As suggested by the appearances of $r^n$ in \eqref{eqn:gen_pol}, it is
helpful to pull out a number of factors of $r$ so that the resulting
underlying dynamics is stabilized at infinity.  More precisely, write ${\mathcal{L} = r^n L}$ where
\begin{align}
L =r \cos(n\theta) \partial_r +
\sin(n\theta) \partial_{\theta}   + \frac{\sigma^2}{2 r^n }\partial_{r}^2 + \frac{\sigma^2}{2r^{n+2}}\partial_{\theta}^2 + P(r, \theta)\partial_r + Q(r, \theta) \partial_{\theta} .  
\end{align}
We will see that using the operator $L$ instead of $\mcl$ itself to
define $\Psi$ results in a number of simplifications, the most notable
of which is that the asymptotic flow along $L$ is much more
straightforward than that of $\mcl$.  Notice that this is expected
since the stochastic dynamics $(r_t, \theta_t)$ defined by $L$ moves
according to a slower clock at infinity than the process $(R_t,
\Theta_t)$ determined by $\mcl$.  Indeed, observe that $(r_{T_t},
\theta_{T_t})= (R_t, \Theta_t)$ where $T_t$ is the time change
\begin{align*}
T_t=\int_0^t R_s^n ds .  
\end{align*}
Considering, too, the nature of Lyapunov functions, all results obtained in terms of $L$ will translate easily back to the original operator $\mcl$ since $\mcl = r^n L$ and $r^n >0$.

\subsection{The General Structure of $\Psi$}
\label{subsec:Bas_structI}

We now take a look at some of the characteristics that our $\Psi$ will
exhibit.  As dictated by the dynamics and suggested by previous works
\cite{AKM_12, GHW_11}, it is easiest construct Lyapunov functions
piecewise.  More precise reasons for why this is the case are given in
the following section and in the sequel \cite{HerzogMattingly2013II}, but here we focus on, at
least abstractly, how $\Psi$ will look in our context.

We begin by partitioning $\R^2$ into the open ball of radius $r^*>0$
about zero, denoted by $B_{r^*}(0)$, and a collection of closed regions
$\{\mathcal{S}_i : i =0,\dots,l\}$, the union of which captures all
routes to infinity.  

\begin{definition}
  We say that a collection of subsets $\mathcal{S}=\{\mathcal{S}_i : i =0,\dots,l\}$ is a
  \textsc{Good Radial Partition of} $\mathcal{U}\subset \R^2$ if the following conditions hold:
\begin{enumerate}
\item Each $\mathcal{S}_i$ is closed and there exists an $r^* >0$ so that 
  \begin{equation*}
    \bigcup_{j=0}^l  \mathcal{S}_j = \mathcal{U}\cap \{ (r,\theta) : r \geq   r^*\}
  \end{equation*}
\item For any $i\neq j$, $\textrm{interior}(\mathcal{S}_i) \cap
  \textrm{interior}(\mathcal{S}_j) = \emptyset$.
\item For all distinct $i$,$j$, and $k$, $\mathcal{S}_i \cap
  \mathcal{S}_j \cap \mathcal{S}_k=\emptyset .$
\item  For any $i\neq j$,  $\mathcal{S}_i \cap \mathcal{S}_j$ is
  either empty or  a collection of disjoint curves, each of which can be written  as $\{(r,f(r)) : r
  \geq r^*\}$ for
  some smooth function $f$.
\end{enumerate}
\end{definition}

\begin{definition}
\label{def:naturalextension}
Let $\Lambda : \R \rightarrow [0,1]$ be a $C^\infty$ function with $\Lambda(r)=0$ for $r\leq r^*$ and $\Lambda (r) =1$ for $r\geq 2 r^*$.  If $\mathcal{S}=\{\mathcal{S}_i : i
=0,\dots,l\}$  is a good radial partition of $\mathcal{U}$ and $f_i\colon\mathcal{S}_i \rightarrow\R$ are $C^2$, we define the
\textsc{natural extension of the} $f_i$'s \textsc{to}  $\mathcal{U}$ by
\begin{align*}
  F(r,\theta) =
  \begin{cases}
    0 & \text{if $(r,\theta) \in B_{r^*}(0)\cap \mathcal{U}$}\\
    \Lambda(r) f_i(r, \theta) & \text{if } (r, \theta) \in \text{interior}(\mathcal{S}_{i})\\
   \frac{\Lambda(r)}{2}( f_i(r,\theta)+  f_j(r,\theta)) & \text{if $(r,\theta) \in
     \mathcal{S}_i\cap  \mathcal{S}_j$}.\\
  \end{cases}
\end{align*}
\end{definition}

In Section~\ref{sec:theconstruction}, we will succeed in constructing
a good radial partition $\mathcal{S}=\{\mathcal{S}_i : i =0,\dots,l\}$
of $\R^2$ and two collections of functions $\{\psi_i\colon
\mathcal{S}_i \rightarrow (0,\infty) : i=0,\dots,l\}$ and
$\{\varphi_i\colon \mathcal{S}_i \rightarrow (0,\infty) :
i=0,\dots,l\}$ such that all functions are continuous and the $\psi_i$ are $C^2$ on their domains, which we
recall were assumed to be closed.  Additionally, the pairs $(\psi_i, \varphi_i)$ will be such that $\psi_i(r, \theta) \wedge \varphi_i(r, \theta) \rightarrow \infty$ as $r\rightarrow \infty$, $(r, \theta) \in \mathcal{S}_i$,  and will satisfy the following bound on $\mathcal{S}_i$ with respect to $L=r^{-n} \mathcal{L}$
\begin{align}\label{localLyop}
  (L \psi_i)(r,\theta) \leq -m_i \varphi_i(r,\theta) +b_i
  \end{align}
  for some constants $m_i,b_i>0$. The fact that each $\psi_i$ is $C^2$
  on the closed set $\mathcal{S}_i$ implies that $L \psi_i$ is
  continuous on $\mathcal{S}_i$ up to and including its boundary.
  Undoing the time change and using the fact that the number of
  inequalities is finite, it follows easily that on $\mathcal{S}_i$
\begin{align}\label{eq:localLyap}
  (\mcl \psi_i)(r,\theta) \leq  -m[ r^n \varphi_i(r,\theta)] +b 
\end{align}
for some global choice of constants $m,b>0$.

Let $\Psi$ and $\Phi$ be the natural extensions to $\R^2$ of the $\psi_i$'s and
$\varphi_i$'s respectively. By equation \eqref{eq:localLyap} and
Remark~\ref{rem:C2}, it is clear that $\Psi$ and $\Phi$ are locally
a Lyapunov pair on the interior of $\mathcal{S}_i $ for each $i=0,1,\dots,l$.
Unfortunately, we will see that this approach does not naturally
produce a $\Psi$ which is $C^2$.  Rather, $\Psi$ will only be globally
continuous as it is possible that the first and second derivatives may
not match along the boundaries between the regions $\mathcal{S}_i$ and
$\mathcal{S}_j$. This prevents us from applying It\^{o}'s formula in a
straightforward way to show that $\Psi$ is a Lyapunov function in the
sense of this paper.

A typical way around this difficulty is to smooth the function $\Psi$
along these interfaces rendering it $C^2$.  However when doing this,
special care must be taken to preserve the Lyapunov property expressed
in \eqref{eq:localLyap}. This often leads to long and less than intuitive calculations.  This was the approach taken in
\cite{AKM_12, GHW_11}. Here we take a different path.

To deal with the issue at hand, we employ a generalization of
It\^{o}'s Formula due to Peskir \cite{Peskir_07}.  This result allows
us to apply the It\^{o} differential to functions which are not $C^2$
along a collection of nonintersecting curves expressing $\theta$ as a function of $r$.
We now state a corollary of Peskir's formula in the context of our
problem.  A more general and detailed treatment is given in Section~\ref{II-sec:peskir}
 of Part II  \cite{HerzogMattingly2013II} along with a repackaged
 proof of a slightly weaker result suiting the needs of this paper.
The proof of the following corollary is a direct consequence of
Theorem~\ref{II-thm:genIto} of that section provided one
establishes the key jump conditions \eqref{eq:fluxCond} along the
curves of non-differentiability.
\begin{corollary}
\label{cor:fluxIto}
 Let  $\mathcal{S}=\{ \mathcal{S}_i :i=0,\dots,l\}$
  be a good radial partition of $\R^2$ and suppose that  $\{\psi_i\colon
\mathcal{S}_i \rightarrow (0,\infty) : i=0,\dots,l\}$ is a collection
of $C^2$ functions and $\{\varphi_i\colon \mathcal{S}_i \rightarrow
(0,\infty) : i=0,\dots,l\}$ is a collection of continuous functions such
that for each $i\in \{0,\dots,l\}$ the estimate in
\eqref{localLyop} holds.  Furthermore, assume that the natural extension $\Psi$ of the $\psi_i$'s is everywhere continuous and satisfies the flux condition 
 \begin{align}
   \label{eq:fluxCond}
   \lim_{\substack{(R,\Theta) \rightarrow (r,\theta)\\ \Theta > \theta} }\partial_\Theta \Psi(R,\Theta)-
     \lim_{\substack{(R,\Theta) \rightarrow (r,\theta)\\ \Theta < \theta}}\partial_\Theta
       \Psi(R,\Theta)\leq 0\,  
 \end{align}
 for all $(r, \theta) \in \R^2$ with $r\geq r^*$.  If $\Phi$ denotes the natural extension of the
$\varphi_i$'s then  $(\Psi, \Phi)$ is a Lyapunov pair on $\R^2$.
\end{corollary}

\begin{remark} The condition \eqref{eq:fluxCond} speaks to the
  convexity along the curves where $\Psi$ is not differentiable. Hence
  it is related to the classical generalization of It\^{o}'s formula
  to functions which are the difference of two convex functions.
\end{remark}

\subsection{Reduction of the Construction to the Principal Wedge}
\label{sec:reductions}
First observe that any system which can be described by
\eqref{eqn:gen_pol} remains a system which can be described by
\eqref{eqn:gen_pol} (with perhaps different $f_k$'s and $g_k$'s) after being rotated by $\theta \mapsto \theta +
\frac{2 k\pi}{n}$ for any integer $k$ .  In particular, we now
note how we can use this fact to reduce the construction of our Lyapunov pair from
$\mathbf{R}^2\setminus B_{r^*}(0)$ to the \emph{principal wedge}
\begin{align}
 \mathcal{R}=\{ r\geq r^*,\, -\tfrac{\pi}{n} \leq \theta \leq \tfrac{\pi}{n} \}.
\end{align}
Defining the remaining wedges by
\begin{align*}
  \mathcal{R}_k=\{ (r,\theta) :   (r,\theta- \tfrac{2 k\pi}{n}) \in \mathcal{R} \},
\end{align*}
we will now see that our construction on $\mathcal{R}$ will allow us to also build
a Lyapunov Function on all of the of the other $\mathcal{R}_k$'s. This
is the content of the following proposition, which is a straightforward consequence of the above observation.

\begin{proposition}
\label{prop:localimglobal}
Fix $n\geq 1$ in \eqref{eqn:polyZ}. Assume that  there exists positive
constants $\gamma$, $\delta$ and $p$ so that for any system of the form
\eqref{eqn:polyZ}, there exists a good radial partition
$\{\mathcal{S}_i : i =0,\dots,l\}$ of $\mathcal{R}$ and a collection of $C^2$-functions $\{\psi_i\colon
\mathcal{S}_i \rightarrow (0,\infty) : i=0,\dots,l\}$ satisfying the bound 
\begin{align*}
  \mathcal{L} \psi_{i}(r,\theta) \leq -m \big[r^\gamma \vee \psi_i^{1+\delta}(r,\theta)] + b
\end{align*}
on $\mathcal{S}_i$ for $i=0,\dots,l$ and some positive constants $m$
and $b$.  Furthermore, assume that the natural extension $\Psi$ of the $\psi_i$'s satisfies the flux condition \eqref{eq:fluxCond} for all $(r, \theta) \in \mathcal{R}$ and is such that $\Psi(r,\theta)=r^p$ for all $(r, \theta) \in
\mathcal{R}$ with $|\theta- \tfrac{\pi}{n}|\wedge |\theta +
\tfrac{\pi}{n}| \leq \epsilon$ for some $\epsilon>0$.  Then 
$(\Psi,\Psi^{1+\delta})$ and $(\Psi, |z|^\gamma)$ are Lyapunov pairs
corresponding to the dynamics
\eqref{eqn:polyZ}. 
\end{proposition}

\begin{remark}
Since $\Psi(r, \theta)=r^p$ in a neighborhood of $\theta = \pm \frac{\pi}{n}$ and the $f_k$'s and $g_k$'s in equation \eqref{eqn:gen_pol} are $2\pi$-periodic, we may rotate $\Psi$, initially defined only on the principle wedge $\mathcal{R}$, by integer multiples of $\frac{2\pi }{n}$ to produce the desired globally-defined Lyapunov pairs.  Moreover, after such rotations the flux condition \eqref{eq:fluxCond} will be satisfied globally.     
\end{remark}

\begin{remark}
  \label{rem:layout} 
  To prove Theorem \ref{thm:lyapfunmadness}, we construct $\Psi$ on
  $\mathcal{R}$ satisfying the properties above, the hypotheses of
  Proposition \ref{prop:localimglobal} with $\gamma \in (n, 2n)$
  arbitrary, and the following bound on $\mathcal{R}$
  \begin{align}
    \label{eqn:gpsibound}
    c \Lambda(r) r^{\gamma-n}\leq \Psi(r, \theta) \leq d \Lambda(r) r^{\gamma -n + \frac{n}{2}+1} 
  \end{align}    
  for some positive constants $c,d, r^*$.    Recall that $\Lambda:\R\rightarrow [0,1]$ is the smooth cutoff function introduced in Definition \ref{def:naturalextension}.  
\end{remark}


\section{The Construction of $\Psi$ on the Principal Wedge in a Simple
  Case}
\label{sec:theconstruction}
In this section, we will build $\Psi$ on $\mathcal{R}$ under
Assumption \ref{assumption:nlot}.  We will see that this assumption
assures that the lower-order terms collected in $F$ in the drift part of
\eqref{eqn:polyZ} play no role in the arguments.

The layout of this section is as follows.  First in Section
\ref{sec:PiecewiseI}, we study the asmyptotic behavior of $L$ as
$r\rightarrow \infty$.  This will help yield the fundamental building
blocks of the construction procedure: the so-called asymptotic
operators and their associated regions.  In Section \ref{sec:lfdefbp},
we explain the intuition behind how the local functions $\psi_i$ will
be defined in Section \ref{sec:mot_lya} as solutions to certain PDEs
involving these operators.  In Section \ref{sec:mot_lya}, we will also
see that each $\psi_i$ is smooth and non-negative on $\mathcal{S}_i$
and that $\psi_i(r, \theta) \rightarrow \infty$ as $r\rightarrow
\infty$ with $(r, \theta) \in \mathcal{S}_i$.

In Section \ref{sec:thedetails}, we will finish proving Theorem
\ref{thm:lyapfunmadness} under Assumption \ref{assumption:nlot} by
checking the details outlined in Remark \ref{rem:layout}.

\subsection{The Asymptotic Operators and Their Associated Regions}
\label{sec:PiecewiseI}

It is intuitively clear that certain terms in the operator $L$ are
asymptotically dominant over other terms as $r\rightarrow \infty$ and
such dominance can change from region to region in the plane.  Here
our goal is to elucidate these ideas by studying more
carefully $L$ along various paths to infinity.  Doing such analysis is
indispensable, as the dominant balances of terms in $L$ yielded from
it will be used to construct the local Lyapunov functions $\psi_i$ in
subsequent sections.

In order to parameterize various routes to infinity, we will make use of the scaling transformations
\begin{equation*}
  S^\lambda_\alpha \colon (r, \theta)  \mapsto(\lambda r, \lambda^{-\alpha} \theta).
\end{equation*}
for any $\lambda \geq 1$ and $\alpha \geq 0$.  In particular, we will determine, heuristically, the
behavior as $\lambda \rightarrow \infty$ of
\begin{multline*}
L\circ S_\alpha^\lambda(r, \theta)= r \cos(n\theta  \lambda^{-\alpha}
) \partial_r  +  \lambda^{\alpha}  \sin(n\theta  \lambda^{-\alpha}
) \partial_{\theta} + \lambda^{-2-n} \frac{\sigma^2}{2 r^n
}\partial_{r}^2 \\+
\lambda^{2\alpha-n-2}\frac{\sigma^2}{2r^{n+2}}\partial_{\theta}^2  +
\lambda^{-1} P(\lambda r, \lambda^{-\alpha} \theta)\partial_r +
\lambda^\alpha Q(\lambda r, \lambda^{-\alpha}\theta) \partial_{\theta}.   
\end{multline*}  
Because we have restricted the construction to the principal wedge $\mathcal{R}$, we only consider routes to infinity contained in $\mathcal{R}$.  The two cases $\alpha=0$ and $\alpha>0$ are qualitatively different,
so they are handled separately. 

Suppose first that $\alpha =0$.  Provided $\theta \neq 0$, we see that the fist two terms in $L\circ
S_0^\lambda(r, \theta)$ are unchanged and all other terms go to
zero as $\lambda \rightarrow \infty$.  More precisely, 
\begin{align*}
L\circ S_0^\lambda(r, \theta) = r\cos(n\theta) \partial_r + \sin(n\theta)  \partial_\theta + O(\lambda^{-1}) \text{ as } \lambda \rightarrow \infty.  
\end{align*} 
We thus conclude that the leading order behavior of $L$ as $r\rightarrow \infty$ in $\mathcal{R}$ along the rays traced out by $\lambda \mapsto
S_0^\lambda(r, \theta)$ (with $\theta \neq 0$) is given by 
\begin{align*}
T_1 = r\cos(n\theta) \partial_r + \sin(n\theta)  \partial_\theta \,.
\end{align*}   
Of course one is not simply restricted to radial paths. So long as one does not
asymptote to the line $\theta=0$ then the same dominate balances hold.
More precisely, $L \approx T_1 $ as $r\rightarrow \infty$ when the paths to infinity are restricted to a region $\mathcal{S}_1$ of the form      
\begin{align}
\mathcal{S}_1 = \{(r, \theta) \in \mathcal{R} \, : \, 0<\theta_1^*\leq |\theta| \leq \theta_0^*\leq \tfrac{\pi}{n}\}
\end{align}
for any fixed positive constants  $\theta_0^* >  \theta_1^*>0$.

The situation becomes more complicated if $|\theta|\rightarrow 0$ as
$r\rightarrow \infty$. To see what happens, we begin by considering $L\circ
S_\alpha^\lambda(r, \theta)$ as $\lambda \rightarrow \infty$ for
$\alpha>0$.  In this setting as $\lambda \rightarrow \infty$,
\begin{align*}
L \circ S_\alpha^\lambda(r, \theta)& = r \partial_r + n \theta \partial_{\theta} + \lambda^{d+\alpha-(n+2)}r^{d-(n+2)}g_d(0) \partial_\theta + \lambda^{2\alpha -(n+2)} \frac{\sigma^2}{2 r^{n+2}} \partial_\theta^2  + O(\lambda^{-1}) + o(\lambda^{d+\alpha-(n+2)}) 
\end{align*}
where $d\in \{ 0,1,\ldots, n+1\}$ is the largest index for which
$g_{d}(0)\neq 0$.  Recall that the $g_k$ are the coefficient functions of $Q(r,\theta)$ introduced in \eqref{eqn:gen_pol}.  If no such index
exists, then
\begin{align*}
\lambda^{d+\alpha-(n+2)} g_d(0) \partial_\theta + o(\lambda^{d+\alpha-(n+2)}) 
\end{align*}
is simply absent from the expression above and the following analysis still holds regardless.

First realize that if $\alpha>0$ is sufficiently small, the linearization of $T_1$ 
\begin{align*}
T_2 = r \partial_r + n \theta \partial_\theta
\end{align*}
gives the leading order asymptotic behavior  as $r\rightarrow
\infty$. Recalling the discussion of $T_1$ above, we see that
even when $\alpha=0$, we have that $L$ is asymptotically well
approximated by $T_2$ provided $|\theta|$ is small since 
\begin{equation*}
  L\circ S_0^\lambda(r, \theta) = r\partial_r +
  n\theta \partial_\theta + O(\lambda^{-1}) +O (\theta^2)\,.
\end{equation*}
In particular, one has $L \approx T_2$ as $r\rightarrow \infty$
provided the paths to infinity are restricted to a region of the form
\begin{align*}
\mathcal{S}_2 = \{(r, \theta) \in \mathcal{R} \, : \, b(r)\leq |\theta| \leq \theta_1^* \}
\end{align*}
were $\theta_1^*>0$ is small and the boundary curve $b(r)$ has the
property that $b(r) \rightarrow 0$ sufficiently slowly as
$r\rightarrow \infty$.  To define $b$ explicitly and also discover what happens to $L$ when $|\theta|\leq b(r)$, we must see for what powers of $\alpha$ other terms in the expansion $L \circ
S_\alpha^\lambda(r, \theta)$ become asymptotically relevant as
$\lambda \rightarrow \infty$.

We now claim that Assumption \ref{assumption:nlot} allows to disregard 
\begin{align*}
\lambda^{d +\alpha-(n+2)} r^{d-(n+2)}g_d(0) \partial_\theta + o(\lambda^{d + \alpha-(n+2)} )
\end{align*}             
in $L \circ S_\alpha^\lambda (r, \theta)$ as $\lambda \rightarrow
\infty$ for all choices of $\alpha\geq 0$.  Indeed if Assumption \ref{assumption:nlot} is satisfied, then it follows that $d\leq \lfloor \frac{n}{2}\rfloor $.  Hence, the value of
$\alpha \geq 0 $ where
\begin{align}
\label{eqn:lotI}
&\lambda^{d+\alpha-(n+2)}r^{d-(n+2)}g_d(0) \partial_\theta + \lambda^{2\alpha
  -(n+2)} \frac{\sigma^2}{2 r^{n+2}} \partial_\theta^2+
O(\lambda^{-1}) + o(\lambda^{d+\alpha-(n+2)})
\end{align}
is $O(1)$ as $\lambda \rightarrow \infty$ is precisely
\begin{align*}
\alpha = \frac{n+2}{2}. 
\end{align*}
But note that for $\alpha \geq \tfrac{n+2}{2}$, by Assumption \ref{assumption:nlot} the term 
\begin{align*}
\lambda^{2\alpha -(n+2)} \frac{\sigma^2}{2 r^{n+2}} \partial_\theta^2
\end{align*}
dominates all the remaining contributions in \eqref{eqn:lotI} in $\lambda$ as $\lambda \rightarrow \infty$.

If Assumption \ref{assumption:nlot} is not made, then the term 
\begin{align}
\label{eqn:menaceI}
\lambda^{d+\alpha -(n+2)} r^{d-(n+2)}g_d(0) \partial_\theta
\end{align}
initially dominates the remaining terms in \eqref{eqn:lotI} as
$\lambda\rightarrow \infty$.  However, at a certain threshold in
$\alpha$, \eqref{eqn:menaceI} can cancel with
$n\theta \partial_\theta$ implying that we must expand $L\circ
S_\alpha^\lambda(r,\theta)$ further asymptotically in $\lambda$ to
uncover the next lower-order term.  The total analysis in the general
case is quite involved and requires another novel idea.  This, in
addition to our desire to focus first on the general elements of the
construction, is why we save it for Part II \cite{HerzogMattingly2013II}.

Operating under Assumption \ref{assumption:nlot}, observe that
the above analysis suggests that the operator
\begin{align}
A = r \partial_r + n \theta \partial_\theta + \frac{\sigma^2}{2 r^{n+2}} \partial_\theta^2
\end{align}
satisfies $L \approx A$ as $r\rightarrow \infty$ in the remaining
portion of $\mathcal{R}$, namely ${\{(r,\theta)\in \mathcal{R}: |\theta|
\leq b(r)\}}$.  To
determine the correct choice of $b(r)$ note that
\begin{align*}
2\alpha- (n+2) \geq 0 \iff \alpha \geq \frac{n+2}{2}.  
\end{align*}
Specifically, the threshold $\alpha=\tfrac{n+2}{2}$ is precisely
where 
\begin{align*}
 r\partial_r + n\theta \partial_\theta+ \lambda^{2\alpha -(n+2)} \frac{\sigma^2}{2 r^{n+2}} \partial_\theta^2 =O(1)
\end{align*} 
as $\lambda \rightarrow \infty$.  Therefore, we now definitively   set 
\begin{align*}
  &\mathcal{S}_2 = \{(r, \theta) \in \mathcal{R}\, : \, r\geq r^*,\,  \eta^* r^{-\frac{n+2}{2}} \leq |\theta| \leq \theta_1^* \}\\
  &\mathcal{S}_3 = \{(r, \theta) \in \mathcal{R}\, : \, r\geq r^*,\,
  |\theta| \leq \eta^* r^{-\frac{n+2}{2}}, |\theta|\leq \theta_1^* \}
\end{align*}
where $\eta^*>0$ is a constant which will be chosen later.
Intuitively though, $\eta^*$ should be thought of as large so that in
$\mathcal{S}_2$ the term $\frac{\sigma^2}{2r^{n+2}} \partial_\theta^2$
is small in comparison to the rest of $A$.  Hence we still expect the
approximation $L \approx T_2$ as $r\rightarrow \infty$ to hold
when paths to infinity are restricted to $\mathcal{S}_2$ even though
$\frac{\sigma^2}{2r^{n+2}} \partial_\theta^2$ does not vanish on the
lower boundary curve.

\begin{remark}
Notice that for any choice of $\eta^*, \theta_1^*>0$, we may always
pick $r^*>0$ large enough so that the bound $|\theta|\leq \theta_1^*$
can be removed from the definition of $\mathcal{S}_3$.  In particular
after making this choice, $\mathcal{S}_1, \mathcal{S}_2,
\mathcal{S}_3$ are  a elements of a good radial partition of
$\mathcal{R}$ as discussed in Section~\ref{subsec:Bas_structI}
and Section~\ref{sec:reductions}.   
\end{remark}

In summary, under Assumption \ref{assumption:nlot} we have found the
\emph{asymptotic operators} $T_1$, $T_2$ and $A$ which ``approximate"
$L$ well for $r>0$ large in the regions $\mathcal{S}_1$,
$\mathcal{S}_2$, and $\mathcal{S}_3$ respectively.

\subsection{Overview of Local Lyapunov Function Construction}
\label{sec:lfdefbp}


As mentioned in Section~\ref{subsec:Bas_structI}, most notably
in Corollary~\ref{cor:fluxIto}, we will
construct our Lyapunov function in piecewise fashion. The regions used
in the construction are precisely the $\{\mathcal{S}_i : i=1,2,3\}$ from Section~\ref{sec:PiecewiseI}.
In particular, we will construct a $C^2$ function $\psi_i: \mathcal{S}_i \rightarrow [0, \infty)$
such that 
\begin{align}
  \label{eq:LyapPoisson}
  (L \psi_i)(r,\theta) \approx - c\, \varphi_i(r,\theta)
\end{align}
for $r$ large, $(r, \theta) \in \mathcal{S}_i$, and such that $\psi_i(r,
\theta) \rightarrow \infty$ as $r\rightarrow \infty$ in
$\mathcal{S}_i$.  In the expression above, $c >0$ is a constant and
$\varphi_i$ is a non-negative continuous function satisfying $\varphi_i(r,
\theta)\rightarrow \infty$ as $r\rightarrow \infty$, $(r, \theta) \in
\mathcal{S}_i$. We will then use the natural extensions of the
$\psi_i$'s and $\varphi_i$'s, as defined in
Section~\ref{subsec:Bas_structI}, to yield a Lyapunov pair.

To do this, we will heavily employ the 
asymptotic analysis carried out in Section \ref{sec:PiecewiseI}.  That is, we will
aim to construct $\psi_1, \psi_2, \psi_3$ satisfying \eqref{eq:LyapPoisson} but with
$L$ replaced by the respective asymptotic operators $T_1, T_2, A$.
The advantage of this methodology is that the operators $T_1, T_2, A$ are much simpler than $L$ yet they 
approximate $L$ well for $r\gg0$ in the appropriate region in space.  However, we must be careful to induce certain homogeneities in the $\psi_i$ 
so that the heuristic analysis in Section \ref{sec:PiecewiseI} can be made rigorous, at least when $L$ is applied to $\psi_1, \psi_2, \psi_3$.  Let us now further illustrate these points.

\begin{figure}
  \centering
  \begin{tikzpicture}[scale=2]
 \draw[thick,color=gray,opacity=.75] (-3.14,0) -- (3.14,0) node[right] {$\theta$};
 \draw[thick,color=gray,opacity=.75] (0,3) node[above] {$r$} --  (0,-.1) node[below]
{$0$};
 \draw[thick,color=gray] (-3.14,0.1) -- (-3.14,-0.1) node[below]
 {$\frac{-\pi}{n}$};
 \draw[thick,color=gray] (3.14,0.1) -- (3.14,-0.1) node[below]
 {$\frac{\pi}{n}$};
\draw[thick,color=gray] (-1.57,0.1) -- (-1.57,-0.1) node[below]
 {$\frac{-\pi}{2n}$};
 \draw[thick,color=gray] (1.57,0.1) -- (1.57,-0.1) node[below]
 {$\frac{\pi}{2n}$};

\draw[thin,pattern=dots,pattern color
  = black!70,opacity=.5]  ({2/3*3.14},3) --
({2/3*3.14},.5) -- (3.14,.5) --  (3.14,3) -- ({2/3*3.14},3) ;

\draw[thin,pattern=dots,pattern color
  = black!70,opacity=.5]  (-{2/3*3.14},3) --
(-{2/3*3.14},.5) -- (-3.14,.5) --  (-3.14,3) -- (-{2/3*3.14},3) ;

 \draw[thick,color=blue] ({2/3*3.14},3) -- ({2/3*3.14},0)
 node[below]{$\theta_0^*$};
\draw[thick,color=blue] ({-2/3*3.14},3) -- ({-2/3*3.14},0)
 node[below]{$-\theta_0^*$};
\draw[thin,pattern=crosshatch,pattern color
  = blue!60,opacity=.5]  ({2/3*3.14},3) --
({2/3*3.14},.5) -- (1/2,.5) --  (1/2,3) -- ({2/3*3.14},3) ;
\draw[thin,pattern=crosshatch,pattern color
  = blue!60,opacity=.5]  (-{2/3*3.14},3) --
(-{2/3*3.14},.5) -- (-1/2,.5) --  (-1/2,3) -- (-{2/3*3.14},3) ;

\draw[thick,color=blue] (1/2,3) -- (1/2,0)
 node[below]{$\theta_1^*$};
\draw[thick,color=blue] (-1/2,3) -- (-1/2,0)
 node[below]{$-\theta_1^*$};

\draw[color=black,thin,color=black,fill=black!30,opacity=.5]   ({-2/3*3.14},3) -- ({-2/3*3.14},.5) -- (
  -.519615242  ,.5) --  plot[domain=.5:3,samples=500,parametric]
 function{-.1*(1.5/t)**(3/2)  ,t } --  (-0.035355339,3);

\draw[color=black,thin,color=black,fill=black!30,opacity=.5]   ({2/3*3.14},3) -- ({2/3*3.14},.5) -- (
  .519615242  ,.5) --  plot[domain=.5:3,samples=500,parametric]
 function{.1*(1.5/t)**(3/2)  ,t } --  (0.035355339,3);

\node[color=red] at (0,.9) {$\mathcal{S}_3$};
\node[color=black] at (.33,.9) {$\mathcal{S}_2$};
\node[color=blue] at (1.33,.9) {$\mathcal{S}_1$};
\node[color=black] at ({5/6*3.14},0.9) {$\mathcal{S}_0$};

\node[color=red] at (0,.9) {$\mathcal{S}_3$};
\node[color=black] at (-.33,.9) {$\mathcal{S}_2$};
\node[color=blue] at (-1.33,.9) {$\mathcal{S}_1$};
\node[color=black] at (-{5/6*3.14},0.9) {$\mathcal{S}_0$};

%

\draw[thin,color=red,,pattern=grid,pattern color
  =red!60,opacity=.5,] plot[domain=.5:3,samples=500,parametric]
 function{-.1*(1.5/(3.5-t))**(3/2)  ,3.5-t } -- (-.519615242,1/2) --
 (.519615242,.5) -- plot[domain=.5:3,samples=500,parametric]
 function{.1*(1.5/(1*t))**(3/2)  , t } -- (0.035355339,3) -- (-0.035355339,3);

\draw[thick,color=blue] (-3.14,.5) -- (3.14,.5) node[right]{$r^*$};
\end{tikzpicture}
\caption{The regions $\mathcal{S}_0, \mathcal{S}_1, \mathcal{S}_2,
  \mathcal{S}_3$. In the diagram, $\theta_1^*>0$ is chosen much larger
  than in reality to make visualization easier. The regions
  $\mathcal{S}_1$, $\mathcal{S}_2$, and $\mathcal{S}_3$ are discussed
  in Section~\ref{sec:PiecewiseI}--\ref{sec:mot_lya} while
  $\mathcal{S}_0$ is only introduced in Section~\ref{sec:mot_lya}. }
  \label{fig:caseI}
\end{figure}
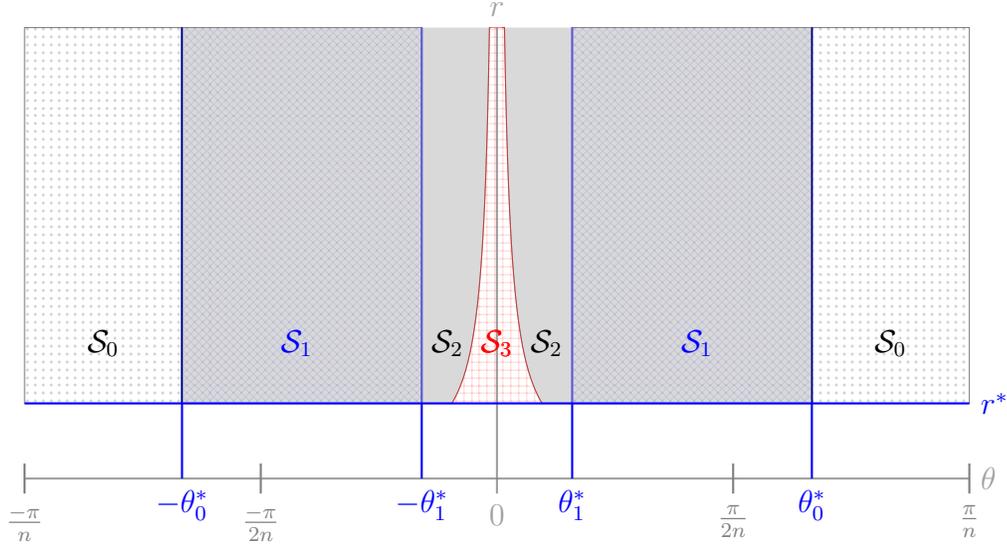

For $(r, \theta) \in \mathcal{S}_1$ with $r\gg0$, we have that $L \approx T_1$. This
suggests defining $\psi_1$ as the solution of the equation
\begin{align}
    \label{eq:LyapPoisson1}
  (T_1 \psi_1)(r,\theta) = - c\, \varphi_1(r,\theta)  
\end{align}
on $\mathcal{S}_1$ with the appropriate boundary conditions.  However, for
\eqref{eq:LyapPoisson1} to imply that \eqref{eq:LyapPoisson} holds for
large $r$ in $\mathcal{S}_1$, we need to know that the terms in
$(L-T_1) \psi_1$ are negligible asymptotically as $r\rightarrow
\infty$, $(r, \theta) \in \mathcal{S}_1$, when compared to those in
$T_1 \psi_1$ .  By the analysis of Section \ref{sec:PiecewiseI}, we
expect the terms $(L-T_1) \psi_1$ to be negligible if $\psi_1$ scales
homogeneously of degree $p>0$ under $S_0^\lambda$, for then $\psi_1(r, \theta) =r^p C(\theta)$ for some function $C(\theta)$, so
the action of $L$ on $\psi_1$ will mimic the action of $S_0^\lambda$
on $L$.

To see why we are able to construct $\psi_1$ so as to have this
homogeneously scaling property, first observe that $T_1$ scales
homogeneously under $S_0^\lambda$.  Therefore if one chooses
$\varphi_1$ in \eqref{eq:LyapPoisson1} to scale homogeneously under
$S_0^\lambda$, then as the solution to \eqref{eq:LyapPoisson1},
$\psi_1$ will, with the appropriate boundary data, also scale
homogeneously under $S_0^\lambda$ with the same scaling exponent as
$\varphi_1$. If we chose the scaling exponent to be positive and
$\varphi_1$ to be continuous, then it will be relatively easy to see
that both $\psi_1(r, \theta), \varphi_1(r, \theta)\rightarrow \infty$
as $r\rightarrow \infty$, $(r, \theta) \in \mathcal{S}_1$, and
\begin{align*}
(L \psi_1)(r, \theta) \leq -m \, \varphi_1(r,\theta) +b,
\end{align*}
for some $m,b>0$.

Jumping ahead to region $\mathcal{S}_3$, since $L\approx A$ for large
$r$, it makes sense to choose $\psi_3$ as the solution of the equation
\begin{align}
    \label{eq:LyapPoisson3}
  (A \psi_3)(r,\theta) = - c\, \varphi_3(r,\theta)
\end{align}
on $\mathcal{S}_3$ where $c>0$ and $\varphi_3(r, \theta) \rightarrow
\infty$ as $r\rightarrow \infty$, $(r, \theta) \in \mathcal{S}_3$.
Since $\lambda \mapsto S_{\frac{n+2}2}^\lambda(r^*,\theta)$ covers
$\mathcal{S}_3$ as $\theta$ varies in $\mathcal{S}_3$ and $A$ is invariant  under
$S_{\frac{n+2}2}^\lambda$, the same reasoning used above
suggests that we should choose $\varphi_3$ to be homogeneous of
positive degree under $S_{\frac{n+2}2}^\lambda$.  Again, it will
then follow that, with the appropriate boundary data, $\psi_3$ and all
of its derivatives are asymptotically homogeneous under
$S_{\frac{n+2}2}^\lambda$ and that $(L-A)\psi_3$ is negligible
relative to $A\psi_3$. Thus by the results of Section
\ref{sec:PiecewiseI}, we anticipate  the following bound
\begin{align*}
(L\psi_3)(r, \theta) \leq -m \varphi_3(r, \theta) + b,
\end{align*}  
on $\mathcal{S}_3$ for some positive constants $m,b$.

The set $\mathcal{S}_2$ serves as a transition region between the two
other sets $\mathcal{S}_1$ and $\mathcal{S}_3$. Hence $\psi_2$ must
connect $\psi_1$, which scales homogeneously under $S_0^\lambda$ in
$\mathcal{S}_1$, to $\psi_3$, which scales homogeneously under
$S_{\frac{n+2}2}^\lambda$ in $\mathcal{S}_3$. Thus we should setup the
equation so that $\psi_2$ and its derivatives asymptotically scale
homogeneously under both mappings, otherwise the $\psi_i$ together
could not be extended to a continuous function as required in
Corollary~\ref{cor:fluxIto}.  Requiring this duel scaling property is
further suggested by the fact that both paths of the form $\lambda
\mapsto S_0^\lambda(r^*,\theta)$ and $\lambda \mapsto
S_{\frac{n+2}2}^\lambda(r^*,\theta)$ are required to cover
$\mathcal{S}_2$.

Since $L \approx T_2$ for $r>0$ large in $\mathcal{S}_2$, we
take $\psi_2$ as the solution of
\begin{align}
    \label{eq:LyapPoisson2}
  (T_2 \psi_2)(r,\theta) = - c\, \varphi_2(r,\theta)
\end{align}
on $\mathcal{S}_2$. Since $T_2$ is homogeneous of degree zero under
$S_\alpha^\lambda$ for any $\alpha\geq 0$, choosing $\varphi_2 $ to
scale homogeneously under $S_\alpha^\lambda$ for all $\alpha\geq 0$
with positive degree will, with the right choice of boundary data,
lead us to a $\psi_2$ which will asymptotically scale  homogeneously under
$S_\alpha^\lambda$ for all $\alpha\geq 0$ with the same positive degree.   
 
\subsection{Defining the Local Lyapunov Functions}
\label{sec:mot_lya}

So far we have subdivided the principal wedge $\mathcal{R}$ into the
following three regions:
\begin{align*}
  &\mathcal{S}_1 = \{ r \geq r^* , \, 0<\theta_1^*\leq |\theta| \leq
  \theta_0^* \leq \tfrac{\pi}{n} \} \\
  & \mathcal{S}_2=\{ r \geq r^* , \, |\theta| \leq \theta_1^*,\,
  r^{\frac{n+2}2}
  |\theta| \geq \eta^* \}\\
  &\mathcal{S}_3=\{ r \geq r^* ,\, |\theta|\leq \theta_1^*, \,
  r^{\frac{n+2}2} |\theta| \leq \eta^* \}
\end{align*}
where $r^*, \,\eta^*>0$ (see Figure~\ref{fig:caseI}).  To initialize the construction procedure, we
will in fact need an additional region $\mathcal{S}_0$ given by
\begin{align*}
  \mathcal{S}_0 = \{ r\geq r^*, \,
  \theta_0^*\leq |\theta| \leq \tfrac{\pi}{n} \}
\end{align*} 
where we fix $\theta_0^* \in (\frac{\pi}{2n}, \frac{\pi}{n})$.
Because the vector field induced by $T_1$ points radially inward in
$\mathcal{S}_0$, defining
\begin{align}
\psi_0(r, \theta) = r^p, \, \, p>0,
\end{align}
and noting that $\psi_0$  clearly scales homogeneous under $S_0^\lambda$, it
will follow easily that
\begin{align*}
L \psi_0(r, \theta) \leq - m r^p + b,
\end{align*}  
on $\mathcal{S}_0$ for some positive constants $m,b$.  The function
$\psi_0$ will now serve as the boundary condition for the equation
satisfied by $\psi_1$ which is defined on the neighboring region $\mathcal{S}_1$.

\begin{remark}
\label{rem:parameter}  
  Each of the regions comes with a number of parameters such as
  $\theta_1^*>0$, $\eta^*>0$ and $r^*>0$. Instead of
  giving these constants specific values at the start, it is much
  easier to leave them as parameters because they will need to be
  adjusted throughout the construction.  To assure that each of these adjustments is
  consistent, we note  that throughout we will always pick $\theta_1^*>0$
  sufficiently small, then $\eta^*=\eta^*(\theta_1^*)>0$ sufficiently
  large, and then $r^*=r^*(\theta_1^*, \eta^*)>0$ sufficiently large.

\end{remark}


\subsubsection*{The Construction in $\mathcal{S}_1$}
Choosing $p \in (0,n)$, the function $\psi_1$ is defined as the
solution of the following equation
\begin{equation}
\label{caseI:psi_1}
  \begin{cases}
    (T_1 \psi_1)(r, \theta) &= - h_1 r^p |\theta|^{-q}\\
    \psi_1(r,\pm \theta_0^*) &= \psi_0(r,\pm \theta_0^*)
  \end{cases}
\end{equation}
on $\mathcal{S}_1$ where $h_1>0$ and $q \in (p/n, 1)$.  

\begin{remark}
The restrictions on $p,q$ stem from the dynamics in $\mathcal{S}_3$.  In particular, we will eventually see why they are needed.  
\end{remark}

Notice that the form of the righthand side of
\eqref{caseI:psi_1} scales homogeneously under $S_0^\lambda$ as
suggested by the considerations in Section~\ref{sec:lfdefbp}. The
dependence on $\theta$ is introduced to facilitate matching with
$\psi_2$ along the
boundary $\mathcal{S}_1 \cap \mathcal{S}_2$.  Also, it is important to point out that since we have
picked $\theta_0^* > \frac{\pi}{2n}$, the PDE given in
\eqref{caseI:psi_1} is not well-defined with the given boundary data
since some of the characteristics along $T_1$ cross $r=r^*$ before
reaching the lines $|\theta|=\theta^*_0$. Because it is convenient to
only give data on the lines $|\theta|=\theta^*_0$, we slightly modify the
domain of definition of the PDE to be
\begin{align*}
  \widetilde{ \mathcal{S}}_1=\Big\{(r, \theta) \in \mathcal{R}\, : \,
  0<\theta_1^*\leq |\theta|\leq \theta_0^*, \, r|\sin(n\theta_0^*)|^{\frac{1}{n}}\geq r^*\Big\}.
\end{align*}
\noindent With this modification, all characteristics now exit the domain
through the boundary $r\geq r^*, \, |\theta|=\theta^*_0$.  Thus,
solving \eqref{caseI:psi_1}, we see that for $(r,\theta) \in
\widetilde{\mathcal{S}}_1$
\begin{align}
\label{expr:psi_1}
\psi_1(r,\theta) = \frac{r^p}{|\sin(n\theta)|^{\frac{p}{n}}}\bigg( \sin(n\theta_0^*)^{\frac{p}{n}} + h_1 \int_{|\theta|}^{\theta_{0}^*} \frac{\sin(n\alpha)^{\frac{p}{n}-1}}{\alpha^q}\, d\alpha \bigg).  
\end{align}
In particular, we note that $\psi_1$ can be extended smoothly to all
of $\mathcal{S}_1$ and is a homogeneous function of degree $p$ under
$S_0^\lambda$.  Moreover, $\psi_1(r, \theta)\geq 0$ on $\mathcal{S}_1$
and $\psi_1(r, \theta) \rightarrow \infty$ as $r\rightarrow \infty$
with $(r, \theta) \in \mathcal{S}_1$.


\subsubsection*{The Construction in $\mathcal{S}_2$}
Let $\psi_2$ be defined on $\mathcal{S}_2$ by 
\begin{equation}
  \begin{cases}
  \label{caseI:psi_2}
    (T_2 \psi_2 )(r, \theta)= - h_2 r^p | \theta|^{-q}\\
    \psi_2(r,\pm \theta_1^*) = \psi_1(r,\pm \theta_1^*)
  \end{cases}
\end{equation}
where $h_2>0$.  This time, the PDE above is clearly well-defined.
In light of Section~\ref{sec:lfdefbp}, observe that the righthand side of \eqref{caseI:psi_2} scales
homogeneously under $S_\alpha^\lambda$ for all $\alpha\geq0$.

Using the method of characteristics we see that
\begin{align}
\label{expr:psi_2}
 \psi_2(r,\theta)= \bigg(
  (\theta_1^*)^{\frac{p}{n}} \psi_1(1,\theta_1^*) - h_2
  \frac{(\theta_1^*)^{\frac{p}n-q}}{qn -p}
  \bigg)\frac{r^p}{|\theta|^{\frac{p}{n}}} + \frac{h_2}{qn-p}
  \frac{r^p}{|\theta|^q}
\end{align}
In particular, we notice that $\psi_2$ is homogeneous under
$S_0^\lambda$ of degree $p$ and is the sum of two terms, each of which is
homogeneous under $S_\alpha^\lambda$ for every $\alpha\geq 0$ (though
each term has a different degree).
Moreover, on $\mathcal{S}_2$
\begin{align*}
\psi_2(r, \theta) \geq c r^p |\theta|^{-\frac{p}{n}}
\end{align*} 
for some $c>0$.  Hence $\psi_2(r, \theta) \geq 0$ on $\mathcal{S}_2$
and $\psi_2(r, \theta) \rightarrow \infty$ as $r\rightarrow \infty$
with $(r, \theta) \in \mathcal{S}_2$.


 \subsubsection*{The Construction in $ \mathcal{S}_3$}
 \label{sec:constructionnoise}
To define the final local Lyapunov function, we would like to assert that it is the solution  $\psi_3$ on $\mathcal{S}_3$ of the problem
\begin{equation}
\label{caseI:psi_3}
  \begin{cases}
   \big(A \psi_3\big)(r, \theta) = - h_3 r^{p_3} & \\
    \psi_3(r,\theta) = \psi_2(r,\theta), & r^{\frac{n+2}2}|\theta|= \eta^*
  \end{cases}
  \end{equation} 
  where $h_3>0$ and $p_3 = p + q\frac{n+2}{2}.$  The conditions in the results of Chapter 9 of \cite{Ok_98} are not met, however, so we cannot immediately apply them to see that solutions of the PDE exist and are unique.  Nevertheless, because the problem above can be essentially converted to a one dimensional problem, we will see that defining $\psi_3$ in this way is indeed permissible.       
  
To see why, let $\eta= r^{\frac{n+2}{2}}\theta$.  Then in the variables $(r, \eta)$, the PDE above transforms as follows:
\begin{equation}
\label{caseI:psi_3eta}
  \begin{cases}
   \big(\hat{A} \hat{\psi}_3\big)(r, \eta) = - h_3 r^{p_3} & \\
    \hat{\psi}_3(r,\eta) = \hat{\psi}_2(r,\eta), &| \eta| = \eta^*
  \end{cases}
  \end{equation}
  where $\hat{f}(r, \eta)= f(r, \theta(r, \eta))=f(r, \eta r^{-\frac{n+2}{2}})$ and $ \hat{A}= r\partial_r +\big(\frac{3}{2}n+1\big) \eta \partial_\eta + \frac{\sigma^2}{2} \partial_\eta^2 . $  Now let $\eta_t$ be the solution of the Gaussian SDE
 \begin{align}
\label{eq:eta}
d\eta_t = \big( \tfrac{3}{2}n+1\big)\eta_t\, dt + \sigma \, dW_t 
\end{align} 
and define ${\tau= \inf \{ t>0 \, : \, |\eta_t| = \eta^*\}}$.  Since $r_t$ is strictly increasing for all $t<\tau$ and $r_0 \geq r^*$, then formally solving the PDE \eqref{caseI:psi_3eta} produces  \begin{align}
  \label{expr:psi_3eta}
\psi_3(r, \theta(r, \eta)) &= c_1 r^{p_3} \E_{\eta} e^{p_3 \tau} +c_2
r^{p_2} \E_{\eta}e^{p_2\tau}- c_3 r^{p_3}
\end{align}
where $p_2 = p+ \frac{p}{n} \frac{n+2}{2}$ and 
\begin{align*}
  c_1 & = \frac{h_3}{p_3}+ \frac{1}{(\eta^*)^q}\frac{h_2}{qn-p}, \,\, \,c_2 = \frac{1}{(\eta^*)^{p/n}}\bigg[ (\theta_1^*)^{p/n} \psi_1(1,
  \theta_1^*) - h_2 \frac{(\theta_1^*)^{p/n -q}}{qn-p}\bigg],\, \, \, c_3= \frac{h_3}{p_3}.
\end{align*}
In a moment, we will show that the expression \eqref{expr:psi_3eta} itself makes sense, in that $\eta\mapsto \E_{\eta} e^{p_i \tau}\in C^2([-\eta^*, \eta^*])$ for $i=2,3$, and that the righthand side of \eqref{expr:psi_3eta} is the unique solution of \eqref{caseI:psi_3eta} which is bounded for fixed $r$ in $\eta$ on $[-\eta^*, \eta^*]$.  Hence, converting back to the variables $(r, \theta)$, 
\begin{align}
\label{expr:psi_3}
\psi_3(r, \theta)= c_1 r^{p_3} \E_{\eta(r, \theta)} e^{p_3 \tau} +c_2
r^{p_2} \E_{\eta(r, \theta)}e^{p_2\tau}- c_3 r^{p_3}
\end{align}  
is the unique solution of the original PDE \eqref{caseI:psi_3} with this boundedness property in $\eta$.

To establish the necessary claims, we prove the following result in Appendix~\ref{appendixA}:
\begin{lemma}\label{lem:exp_webf}
  Fix a constant $c\in \R$, let $\eta^* >|c|$ and
  define the stopping time 
  \begin{align*}
  \tau_c = \inf_{t>0} \{ t>0 \, : \, \eta_t \notin
  [-\eta^*+c, \eta^*+c]\}.
  \end{align*} 
  If $G_{a,c}(\eta) := \E_{\eta}e^{a \tau_c}$ and $0<a< \frac{3}{2}n+1$, then for all $\eta^*$ large enough we have the following: 
  \begin{enumerate}
 \item $G_{a,c}\in C^{\infty}([-\eta^*+c, \eta^*+c])$.  Moreover, 
\begin{align}
\label{web:asmI}
G_{a,c}'(\pm \eta^* +c ) = \mp \frac{2a}{3n+2}(\eta^*)^{-1} +
o((\eta^*)^{-1}) \text{ as } \eta^* \rightarrow \infty.
\end{align}   
\item \label{conclusionb}$G_{a,c}$ is the unique solution of \eqref{G:eqI}.  
\end{enumerate}
\end{lemma}

\begin{remark}
  In this article we only need the case $c=0$. However in the sequel
  \cite{HerzogMattingly2013II}, we will need the full strength of
  Lemma \ref{lem:exp_webf}; that is, all results above when
  $c\neq 0$.  
  \end{remark}

\begin{remark}
Applying conclusion \ref{conclusionb}) of the result, we see that the righthand side of \eqref{expr:psi_3eta} solves \eqref{caseI:psi_3eta}.  
\end{remark}
\begin{remark}
Given that $p\in (0, n)$ and $q \in (p/n, 1)$, we see that $$p_2 <p_3< \frac{3}{2}n+1 .$$  
Hence by the lemma above, $\psi_3(r, \theta)\in C^\infty(\mathcal{S}_3)$ for all $\eta^*>0$ large enough.  Note that this choice of $\eta^*$ is consistent with Remark~\ref{rem:parameter}.        
\end{remark}

\begin{proof}[Proof of Lemma \ref{lem:exp_webf}]
See Appendix~\ref{appendixA}.  
\end{proof}

Considering the framework of the construction procedure outlined in Section~\ref{sec:lfdefbp}, upon taking another look at the expression \eqref{expr:psi_3} we see that $\psi_3$ is
the sum of three terms, each of which is homogeneous under the scaling
transformation $S^\lambda_{\frac{n+2}{2}}$.  Also, it is not hard to see
that on $\mathcal{S}_3$
\begin{align*}
\psi_3(r, \theta) \geq c r^{p_2} \E_{\eta(r, \theta)} e^{p_2\tau}
\end{align*} 
for some $c>0$.  Hence, $\psi_{3}(r, \theta) \geq 0$ on
$\mathcal{S}_3$ and $\psi_{3}(r, \theta) \rightarrow \infty$ as
$r\rightarrow \infty $ with $(r, \theta)\in \mathcal{S}_3$.

\subsection{The Relationship Between the Scaling of $\mathcal{S}_1$ and $\mathcal{S}_3$ and the Origin of the Restriction on $p$}
\label{sec:scalingRelationDifRegion}

Now that the basic construction is
finished, let us take a moment to elucidate the relationship between
the scaling exponents $p$ and $p_2$. We will show that the shape of region
$\mathcal{S}_2$ dictates the relationship between the two. This is fundamental
to understanding the problem since we saw that the equation for $\psi_3$
places a restriction of the exponent $p_2$ which in turn cascades through
the remaining dependencies to place a restriction on $p$.

The function $\psi_2$ consists of two terms: one comes from the boundary
data propagated along the flow and the other from integrating the
right-hand side along the characteristics. Denoting the part of solution $\psi_2$ which comes from the
boundary data by $\tilde{\psi}_2$, setting $h_2=0$ in \eqref{expr:psi_2} gives that $\tilde{\psi}_2(r, \theta)= c r^p |\theta|^{-\frac{p}{n}}$ for some positive constant $c$.  Hence, $\tilde{\psi}_2$ is homogeneous of degree $p (1+ \alpha/n)$ under the scaling transformation $S_\alpha^\lambda$.  Since the lower boundary
of $\mathcal{S}_2$ is homogeneous under $S_\alpha^\lambda$ with $\alpha =\frac{n+2}{2}$, we see that $\tilde{\psi}_2$ must be
homogeneous under $S_{\frac{n+2}{2}}^\lambda$ of 
of degree
\begin{align*}
p_2 = p \Big( \frac{3n+2}{2n}\Big).
\end{align*}
Since we saw that $p_2$ was required to be less than $ \frac{3n+2}{2}$, we conclude
$p$ has to be less than
\begin{align*}
\Big( \frac{2n}{3n+2}\Big)\Big( \frac{3n+2}{2}\Big)=n
\end{align*}
which was the restriction placed on $p$ when it was introduced when $\psi_1$ was defined.  In summary, the solution of the exit problem associated to $\mathcal{S}_3$ is only
well defined if $p_2 <
\frac{3n+2}{2}$
by Lemma \ref{lem:exp_webf}. 

\section{Proof of Theorem~\ref{thm:lyapfunmadness}}
\label{sec:thedetails}

We now prove that under Assumption~\ref{assumption:nlot}, the
functions $\{\psi_i : i=0,1,2,3\}$ together with their corresponding
domains of definition $\{\mathcal{S}_i : i=0,\ldots,3\}$ satisfy the
hypotheses of Proposition~\ref{prop:localimglobal} with the
appropriate choice of the parameters $\theta_1^*, \eta^*, r^*, h_1,
h_2, h_3$.  Having done so, we will have proven
Theorem~\ref{thm:lyapfunmadness}, as the bound \eqref{eqn:gpsibound}
will follow almost immediately.

The layout of this section is as follows. First, we will deduce the
large $r$ asymptotics of the functions $\psi_1$, $\psi_2$, and
$\psi_3$, allowing us to validate the bound \eqref{eqn:gpsibound}.
Second, we will show that for all $r^*$, $\eta^*$ sufficiently large
and all $\theta_1^*$ sufficiently small (chosen, of course, in the consistent way mentioned in Remark~\ref{rem:parameter}), the boundary-flux conditions
given in \eqref{eq:fluxCond} are satisfied for some choice of the
positive parameters $h_1, h_2, h_3$.  Lastly, we will verify the local Lyapunov property from
\eqref{eq:localLyap}.  The second and third items in the agenda will
check the hypotheses of Proposition \ref{prop:localimglobal}.

Beginning with the large $r$ asymptotics, the following proposition
derives them quickly from the construction of the $\psi_i$'s and the
accompanying discussions.

\begin{proposition}\label{prop:bounds}
There exist positive constants $l_i,u_i$ such that 
\begin{equation}
  \label{eq:psiBounds}
  \begin{aligned}
l_1 r^p \leq &\psi_1(r, \theta) \leq u_1 r^p \qquad\qquad &&(r,
  \theta) \in \mathcal{S}_1\\ 
  l_2 \frac{r^p}{|\theta|^{\frac{p}{n}}} \leq  &\psi_2(r, \theta) \leq
  u_2 \frac{r^p}{|\theta|^{q}} && (r, \theta) \in \mathcal{S}_2\\ 
  l_3 r^{p_2} \leq &\psi_3(r, \theta) \leq u_3 r^{p_3} && (r, \theta)
  \in \mathcal{S}_3
\end{aligned}
\end{equation}
where the we recall that the constants $p, p_2, p_3$ satisfy $p_2=p+\frac{p}{n}\frac{n+2}{2}$ and $p_3 = p+ q\frac{n+2}{2}$ where $p\in (0,n)$ and $q\in (p/n,1)$.    
\end{proposition}
\begin{proof}[Proof of Proposition~\ref{prop:bounds}]
We begin with $\psi_1$. Since $\psi_1$ scales homogeneously under
$S_0^\lambda$ with degree $p$, for any $(r,\theta) \in \mathcal{S}_1$ we have that
\begin{align*}
  \psi_1(r,\theta) = \Big(\frac{r}{r^*}\Big)^p \psi_1(r^*,\theta)
\end{align*}
and hence 
\begin{align*}
\frac{m}{(r^*)^p}\,  r^p \leq   \psi_1(r,\theta) \leq \frac{M}{(r^*)^p}\,  r^p
\end{align*}
where $M=\sup\{  \psi_1(r^*,\theta) : \theta \in
[\theta_0^*,\theta_1^*]\}$ and $m=\inf\{  \psi_1(r^*,\theta) : \theta \in
[\theta_0^*,\theta_1^*]\}$. Since  $\psi_1(r^*,\theta)$ is continuous
$M \geq m > 0$.  The bounds on $\psi_3$ are handled in a completely
analogous way only using the scaling generated by $S^\lambda_{\frac{n+2}2}$
rather than $S^\lambda_0$. From \eqref{expr:psi_3}, we see that  the terms which make up $\psi_3$ do
not all scale with the same degree. Hence we obtain a upper bound of
$r^{p_3}$ and a lower bound of $r^{p_2}$.  Since the region $\mathcal{S}_2$ requires
both scalings to reach all points, we would need a slightly more complicated
construction which mixed the two scaling to obtain the bounds on
$\psi_2$ using just the abstract scaling. While this is not difficult, in light of the explicit representation
of $\psi_2$ given in \eqref{expr:psi_2}, we see that the quoted bounds
follow by inspection.
\end{proof}
With these estimates in hand, we turn to the more techincal of the two
remaining topics. 
\subsection{Boundary-flux conditions}
\label{sub:bfc}


\subsubsection*{Boundary between  $\mathcal{S}_0$ and $\mathcal{S}_1$}

Because $\psi_1(r,\theta)= r^p \psi_1(1,\theta)$ and 
\begin{align*}
  -h_1 r^p |\theta|^{-q} = \frac{\partial
    \psi_1}{\partial r}r \cos(n\theta)+ \frac{\partial
    \psi_1}{\partial \theta} \sin(n\theta)
\end{align*}
one has
\begin{align}\label{thetaDeriv0I}
  \frac{\partial \psi_1}{\partial \theta}=- r^p \Big(\frac{p
    \cos(n \theta) \psi_1(1,\theta) + h_1 |\theta|^{-q}}{ \sin(n \theta)}\Big).
\end{align}
Therefore combining  $\frac{\partial \psi_0}{\partial \theta}=0$ with
\eqref{thetaDeriv0I} produces
\begin{align*}
  \Big[ \frac{\partial \psi_0}{\partial \theta }- \frac{\partial
    \psi_1}{\partial \theta}\Big]_{\theta=\theta_0^*} = r^p
  \Big(\frac{p \cos(n \theta_0^*) \psi_1(1,\theta_0^*) + h_1
    (\theta_0^*)^{-q}}{ \sin(n \theta_0^*)}\Big).
\end{align*}
Since $\psi_j(r, \theta) = \psi_j(r, -\theta)$ on $\mathcal{S}_j$ for
$j=0,1$, we note also that
\begin{align}\label{eqn:bf01}
  \Big[ \frac{\partial \psi_1}{\partial \theta }- \frac{\partial
    \psi_0}{\partial \theta}\Big]_{\theta=-\theta_0^*} = \Big[
  \frac{\partial \psi_0}{\partial \theta }- \frac{\partial
    \psi_1}{\partial \theta}\Big]_{\theta=\theta_0^*}.
\end{align}
Since $\psi_1(1, \theta_0^*)=1$, $\sin(n\theta_0^*)>0$ and
$\cos(n\theta_0^*)<0$, picking
\begin{align}
0<h_1< p (\theta_0^*)^q |\cos(n\theta_0^*)| 
\end{align}
implies that the quantity \eqref{eqn:bf01} is negative.


\subsubsection*{Boundary between  $\mathcal{S}_1$ and $\mathcal{S}_2$}

Similar to the previous computations, observe that $\psi_2(r,
\theta)=r^p \psi_2(1, \theta)$ implies
\begin{align*}
  \frac{\partial \psi_2}{\partial \theta} = - r^p \bigg[ \frac{p
    \psi_2(1, \theta) + h_2 |\theta|^{-q}}{n \theta} \bigg].
\end{align*}
Since $\psi_1(1, \theta_1^*) = \psi_2(1, \theta_1^*)$, we then obtain 
\begin{align*}
  \bigg[ \frac{\partial \psi_1}{\partial \theta }- \frac{\partial
    \psi_2}{\partial \theta}\bigg]_{\theta=\theta_1^*} 
  &= -r^p \bigg[\frac{p
    \cos(n \theta_1^*) \psi_1(1,\theta_1^*) + h_1 (\theta_1^*)^{-q}}{
    \sin(n \theta_1^*)}- \frac{p \psi_1(1, \theta_1^*) + h_2
    (\theta_1^*)^{-q}}{n \theta_1^*}\bigg]\\ 
  &= -\frac{ r^p }{(\theta_1^*)^{q+1}} \bigg[ \bigg( \frac{p
    \cos(n\theta_1^*)}{\sin(n\theta_1^*)}- \frac{p}{n\theta_1^*}
  \bigg)\psi_1(1, \theta_1^*) (\theta_1^* )^{q+1} + \bigg(
  \frac{h_1}{\sin(n\theta_1^*)}-
  \frac{h_2}{n\theta_1^*}\bigg)(\theta_1^*)\bigg].
\end{align*}
Using the expression \eqref{expr:psi_1}, it is not hard to check that
$\psi_1(1, \theta_1^*) (\theta_1^*)^{q}$ is bounded as $\theta_1^*
\downarrow 0$.  Employing the Taylor expansions for
$\sin(n\theta_1^*)$ and $\cos(n\theta_1^*)$ about $\theta_1^*=0$, we
thus obtain
\begin{align} 
\label{eqn:bf2I}
\bigg[ \frac{\partial \psi_1}{\partial \theta }- \frac{\partial
  \psi_2}{\partial \theta}\bigg]_{\theta=\theta_1^*}\simeq -
\frac{r^p}{ |\theta_1^* |^{q+1}} \bigg( \frac{h_1}{n}-
\frac{h_2}{n}\bigg)
\end{align}
as $\theta_1^* \downarrow 0$ where $\simeq$ denotes asymptotic
equivalence.  Picking $h_2<h_1$, for all $\theta_1^* >0$ sufficiently small the quantity on the left-hand side
of \eqref{eqn:bf2I} is negative for $r\geq r^*$.  Since $\psi_j(r,
-\theta)= \psi_j(r, \theta)$ on $ \mathcal{S}_j$ for $j=1,2$, notice
that we also have the equality
\begin{align*}
  \bigg[\frac{\partial \psi_2}{\partial \theta}- \frac{\partial
    \psi_1}{\partial \theta } \bigg]_{\theta=-\theta_1^*}=\bigg[
  \frac{\partial \psi_1}{\partial \theta }- \frac{\partial
    \psi_2}{\partial \theta}\bigg]_{\theta=\theta_1^*}
\end{align*} 
Hence, the same choice of $\theta_1^*$ and $h_2>0$ results in a
negative sign for the flux across the boundary $\theta=-\theta_1^*$ as
well.


\subsubsection*{Boundary between $\mathcal{S}_2$ and $\mathcal{S}_3$}
\label{sec:whatisG}

Thus far it has been fairly straightforward to compute and analyze
fluxes across boundaries where noise plays no role.  In such cases, we
saw that we could find convenient expressions for $\partial_\theta
\psi_i$, $i=0,1,2$, simply by using the first-order PDEs these
functions satisfy.  A similar approach, however, does not work when
studying the flux across the boundaries between $\mathcal{S}_2$ and
$\mathcal{S}_3$ since the operator $A$ contains second-order partial
derivatives in $\theta$.  Therefore, to study $\partial_\theta
\psi_3$ along these interfaces, we opt to employ the somewhat explicit
expression \eqref{expr:psi_3} derived in the previous section.
Because there is no closed form expression for the functions
$G_{p_i}(\eta):= \E_\eta e^{p_i \tau}$, $i=2,3$, the analysis is slightly harder in this case.  We did see (at least in  the statement of Lemma~\ref{lem:exp_webf}), however, that analysis of $G'_{p_{i}}(\eta^*)$ is possible for large $\eta^*>0$.  One should have expected this because, by the computations
of Section \ref{sec:PiecewiseI}, the noise term in $A$ formally scales
away as $r^{\frac{n+2}{2}}|\theta| \rightarrow \infty$.  It turns out that this is all we need to see that the boundary flux terms have the right sign.

We now apply the Lemma~\ref{lem:exp_webf} to help control the flux terms across the boundaries
$r^{\frac{n+2}{2}}\theta= \pm\eta^*$.  By the symmetry
$G_{p_i}(\eta)=G_{p_i}(-\eta)$ for $\eta \in [-\eta^*, \eta^*]$, we
must only show that for $\eta^*>0$ sufficiently large, $h_3>0$ can be
chosen so that the flux across the boundary $r^{\frac{n+2}{2}}\theta =
\eta^*>0$ is negative for $r^*>0$ sufficiently large.  Observe that
\begin{align*}
  &\Big[\frac{\partial\psi_2 }{\partial \theta} -
  \frac{\partial\psi_3 }{\partial \theta}
  \Big]_{r^{\frac{n+2}{2}}\theta= \eta^*}\\ 
  &= -\Big(\frac{q h_2}{qn-p}\frac{1}{(\eta^*)^{q+1}}
  + \frac{1}{(\eta^*)^q}\frac{h_2}{qn-p}G_{p_3}'(\eta^*)+
  \frac{h_3}{p_3}G_{p_3}'(\eta^*)\Big)
  r^{p + \frac{n+2}{2}(q+1)} + o(r^{p + \frac{n+2}{2}(q+1)}) \text{ as } r\rightarrow
  \infty.
\end{align*} 
Recalling the assumption that $q>p/n$, observe that \eqref{web:asmI}
implies that for $\eta^*>0$ large enough and $h_3>0$ small enough
the righthand sisde of the above expression above is negative  for all 
sufficiently large $r^*$.

\subsection{The local Lyapunov property}

We now verify the local Lyapunov property given in
\eqref{eq:localLyap}. To do so, we will not need to change the values
of the $h_i$, $i=1,2$, set in the previous section.  We will need to, however, 
increase $r^*$, $\eta^*$ as well as decrease $\theta_1^*$, but this will consistent with all previous choices, including the choice of $h_3(\eta^*)$ made in the previous section, to assure that each boundary-flux term had the appropriate sign.

Letting $B$ denote the asymptotic operator corresponding to $L$ in
$\mathcal{S}_i$, this involves first writing
\begin{align*}
L\psi_i(r, \theta) = B\psi_i(r, \theta) + (L-B) \psi_i(r, \theta)
\end{align*} 
on $\mathcal{S}_i$.  Since $B\psi_i$ is of the desired form, all we
must do is estimate the remainder term $(L-B) \psi_i$ to see that
\begin{align*}
|(L-B) \psi_i| \ll |B\psi_i|  
\end{align*}        
as $r\rightarrow \infty$, $(r, \theta) \in \mathcal{S}_i$.  We proceed
region by region starting with:

\subsubsection*{Region $\mathcal{S}_0$}

Since $\psi_0(r, \theta)=r^p$, it is not hard to see that as
$r\rightarrow \infty$, $(r, \theta) \in \mathcal{S}_0$,
\begin{align}
\label{eqn:llpsi_1asI}
L \psi_0 (r, \theta) & =p r^{p} \cos(n\theta) + o(r^{p}).  
\end{align}
Since $\cos(n\theta) \leq -c <0$ for $(r, \theta) \in \mathcal{S}_0$
and some $c>0$, the relation \eqref{eqn:llpsi_1asI} implies that there
exist positive constants $c_0, d_0$ such that
\begin{align*}
L\psi_0(r, \theta) &\leq - c_0 r^{p} + d_0
\end{align*}   
for all $(r, \theta) \in \mathcal{S}_0$.  Undoing the time change, we
find easily that on $\mathcal{S}_0$
\begin{align}
\label{eqn:psi_0b2}
\mathcal{L}\psi_0(r, \theta) &\leq - C_0 r^{p+n} + D_0
\end{align}   
for some $C_0, D_0>0$.

\subsubsection*{Region $\mathcal{S}_1$}  
First observe that by definition of $\psi_1$   
\begin{align*}
L\psi_1 (r, \theta)&= T_1 \psi_1 (r, \theta) + (L-T_1) \psi_1(r, \theta)\\
&= - h_1 r^p |\theta|^{-q} + (L-T_1) \psi_1(r, \theta)  
\end{align*}
on $\mathcal{S}_1$.  To bound the remainder term $(L-T_1) \psi_1(r,
\theta)$, notice by \eqref{expr:psi_1} we may write $\psi_1(r, \theta)
= r^p g(\theta)$ where $g$ is a smooth and positive function in
$\theta$ for all $0<\theta_1^*\leq |\theta|\leq \theta_0^*$.  In
particular, since $0<\theta_1^*\leq |\theta|\leq \theta_0^*$ for $(r,
\theta) \in \mathcal{S}_1$, we see that as $r\rightarrow \infty$ with
$(r, \theta) \in \mathcal{S}_1$
\begin{align}
L \psi_1 (r, \theta)&=- h_1 \frac{r^p}{|\theta|^q} + o(r^p).  
\end{align}
Using the asymptotic formula above as well as positivity and
smoothness of $g$ on the domain for $\theta$ in $\mathcal{S}_1$, we
obtain the inequality
\begin{align*}
L \psi_1 (r, \theta)&\leq 
- c_1 \frac{r^p}{|\theta|^q} + d_1  
\end{align*}
on $\mathcal{S}_1$ for some constants $c_1, d_1>0$.  Undoing the time
change, we also find that
 \begin{align}\label{eqn:psi_1b1}
\mathcal{L} \psi_1 (r, \theta)&\leq 
- C_1 \frac{r^{p+n}}{|\theta|^q} + D_1  
\end{align} 
on $\mathcal{S}_1$ for some constants $C_1, D_1>0$.  

\subsubsection*{Region $\mathcal{S}_2$}

By definition of $\psi_2$, first observe that on $\mathcal{S}_2$
\begin{align*}
  L \psi_2 (r, \theta) &= T_2 \psi_2(r, \theta) + (L-T_2) \psi_2(r, \theta)= -h_2 \frac{r^p}{|\theta|^q} +(T_1-T_2)\psi_2(r, \theta) + (L-T_1)
  \psi_2(r, \theta).
\end{align*}
Using the Taylor expansions for $\sin(n\theta)$ and $\cos(n\theta)$
about $\theta=0$ notice that there exists a constant $C>0$ which is independent of $\theta_1^*>0$ such that
\begin{align*}
(T_1-T_2)\psi_2(r, \theta)&\leq C \theta^2\bigg[\Big(
  (\theta_1^*)^{\frac{p}{n}} \psi_1(1,\theta_1^*) + h_2
  \frac{(\theta_1^*)^{\frac{p}n-q}}{qn -p}
  \Big) \frac{r^p}{|\theta|^{p/n}} +  \frac{h_2}{qn-p} \frac{r^p}{|\theta|^q}\bigg]\\
  &\leq  C (\theta_1^*)^2\bigg[\Big(
  (\theta_1^*)^{\frac{p}{n}} \psi_1(1,\theta_1^*) + h_2
  \frac{(\theta_1^*)^{\frac{p}n-q}}{qn -p}
  \Big) \frac{r^p}{|\theta|^{p/n}} +  \frac{h_2}{qn-p} \frac{r^p}{|\theta|^q}\bigg] 
\end{align*}     
for all $(r, \theta) \in \mathcal{S}_2$.  In particular, since
$\psi_1(1, \theta_1^*)=O((\theta_1^*)^{-1})$ as $\theta_1^* \downarrow
0$, it follows that for all $\epsilon >0$, there exists $\theta_1^*>0$
small enough so that
\begin{align*}
(T_1-T_2)\psi_2(r, \theta)
&\leq \epsilon \frac{r^p}{|\theta|^q} 
\end{align*} 
for all $(r, \theta) \in \mathcal{S}_2$.  Therefore, in particular, we
may choose $\theta_1^*>0$ small enough so that
\begin{align*}
L \psi_2 (r, \theta) &\leq 
-\frac{h_2}{2} \frac{r^p}{|\theta|^q} + (L-T_1) \psi_2(r, \theta)
\end{align*} 
on $\mathcal{S}_2$.  To control the remaining term $(L-T_1) \psi_2(r,
\theta)$, recall that we are operating under Assumption \ref{assumption:nlot}.
Therefore, we find that there exists positive constants $C,D$
independent of $\eta^*, r^* $ such that on $\mathcal{S}_2$
\begin{align*}
  (L-T_1) \psi_2(r, \theta) &\leq\Big( \frac{C}{\eta^*} +
  \frac{D}{r^*}\Big)\frac{r^p}{|\theta|^q}.
\end{align*}  
Picking $\eta^*, r^*>0$ sufficiently large we see that there exist
constants $c_2,d_2>0$ such that
\begin{align*}
L \psi_2(r, \theta) & \leq - c_2 \frac{r^p}{|\theta|^q} + d_2
\end{align*}  
for all $(r, \theta) \in \mathcal{S}_2$.  Undoing the time change, we
obtain the bound
\begin{align}
\label{eqn:psi_2b2}
\mathcal{L} \psi_2(r, \theta) & \leq - C_2 \frac{r^{p+n}}{|\theta|^q} + D_2
\end{align} 
on $\mathcal{S}_2$ for some constants $C_2, D_2>0$.

\subsubsection*{Region $\mathcal{S}_3$}
First decompose $L\psi_3$ on $\mathcal{S}_3$ as follows
\begin{align*}
  L \psi_3(r, \theta) = A \psi_3(r, \theta) + (T-A) \psi_3(r, \theta)
  + (L-T) \psi_3(r, \theta)
\end{align*}
where 
\begin{align*}
  T= r\cos(n\theta) \partial_r + \sin(n\theta) \partial_\theta +
  \frac{\sigma^2}{2 r^{n+2}} \partial_\theta^2 .
\end{align*}
Hence  
\begin{align*}
  L \psi_3(r, \theta) &= -h_3 r^{p_3} + (T-A) \psi_3(r, \theta) + (L-T) \psi_3(r, \theta)\\&= -h_3 r^{p_3} + (T_1-T_2) \psi_3(r, \theta) + (L-T) \psi_3(r,
  \theta).
\end{align*}
Let us first see how to bound $(T_1-T_2)\psi_3$.  Again, making use of
the Taylor expansions for $\sin(n\theta)$ and $\cos(n\theta)$, we see
that there exists a constant $C>0$ which is independent of $\eta^*$
such that
\begin{align*}
  (T_1-T_2)\psi_3(r, \theta) &\leq C \theta^2 \big(|r\partial_r
  \psi_3(r, \theta)| + |\partial_\theta \psi_3(r, \theta) | \big).
\end{align*}
To control derivatives of $\psi_3$, recall the expression
\eqref{expr:psi_3}.  Applying Lemma \ref{lem:exp_webf}, we deduce the
existence of a constant $C=C(\eta^*)>0$ such that
\begin{align*}
  \theta^2 (|r\partial_r \psi_3(r, \theta)| + |\partial_\theta
  \psi_3(r, \theta)|)\leq C(\eta^*) r^{p_3-\frac{(n+2)}{2}}
\end{align*}  
on $\mathcal{S}_3$.  In particular, we have thus far obtained
\begin{align}
\label{eqn:psi3llboundi}
L \psi_3(r, \theta)\leq - h_3(\eta^*) r^{p_3} + C(\eta^*) r^{p_3-
  \frac{n+2}{2}} + (L-T) \psi_3(r, \theta)
\end{align}
for all $(r,\theta) \in \mathcal{S}_3$.  To estimate the remaining
term $(L-T) \psi_3$, proceed in a similar fashion using Assumption
\ref{assumption:nlot} to see that
\begin{align}
\label{eqn:psi3lot}
(L - T) \psi_3(r, \theta) \leq D(\eta^*) r^{p_3-1}
\end{align}  
on $\mathcal{S}_3$.  Putting \eqref{eqn:psi3llboundi} together with
\eqref{eqn:psi3lot} and picking $r^*>0$ large enough, there exist constants $c_3, d_3 >0$ such that on
$\mathcal{S}_3$
\begin{align}
\label{eqn:psi_3b1}
L \psi_3(r, \theta) \leq - c_3 r^{p_3} +d_3.
\end{align}
Undoing the time change, we also see that on $\mathcal{S}_3$
\begin{align}
\label{eqn:psi_3b2}
\mathcal{L} \psi_3(r, \theta) \leq - C_3 r^{p_3 +n} +D_3.
\end{align}
for some constants $C_3, D_3>0$.

\begin{remark}
  Using the bounds obtained in Proposition \ref{prop:bounds} and the
  inequalities \eqref{eqn:psi_0b2}, \eqref{eqn:psi_1b1},
  \eqref{eqn:psi_2b2}, and \eqref{eqn:psi_3b2} we can easily see that
  the bound for $\mathcal{L} \psi_i$, $i=0,1,2,3$, on $\mathcal{S}_i$
  required by Proposition \ref{prop:localimglobal} is satisfied.
  Because the boundary flux terms have the appropriate sign by the
  arguments of Section \ref{sub:bfc}, we have now finished proving
  Theorem \ref{thm:lyapfunmadness} under Assumption
  \ref{assumption:nlot}.
\end{remark}

\section{Conclusion}
\label{sec:conclusion}
We have given a general methodology for constructing Lyapunov
functions and applied it to study a family of equations in which the
underlying deterministic dynamics is stabilized under the addition of
noise. The method incorporates global information of the flow and
hence is well suited in the setting where stability results from
global rewiring of trajectories due to the addition of a small amount
of noise. The use of auxiliary PDEs to define our Lyapunov functions
was central to the construction as it allowed us to obtain radially
optimal results.  There are a number of points which, though technical, allow for a successful
completion of the argument. We always use homogeneous operators in our
local constructions, as this
allows us to create local Lyapunov functions  through the use of
auxiliary PDEs which are the sum of homogeneously scaling terms.
This greatly simplifies the general analysis. The homogeneous
operators are also drastically simplified from the original
generator. This makes many points of the analysis easier, often allowing
explicit representations of solutions. We also employ a
extension of It\^o theorem which allows us to avoid smoothing the
patched functions along interfaces.

Our construction of Lyapunov functions is closely related to the
construction of sub/super solutions to certain PDEs associated to the
SDEs considered. In particular, all of our results can be translated
to the existence of a normalizable solutions with polynomial decay at
infinity to the PDE $\mathcal{L}^* \rho=0$ where $\mathcal{L}$ is the
generator of the SDE \eqref{eqn:polyZ}.

In Part II \cite{HerzogMattingly2013II} of this paper, we will
consider the same class of problems but in a more general setting
where Assumption~\ref{assumption:nlot} does not hold. In the
conclusion of that paper we will give a number directions of possible future work.

\appendix
\section{}
\label{appendixA}

In this section, we will prove the remaining technical results needed in this work; that is, we will show Lemma~\ref{l:existInvM}, Theorem~\ref{LyapConverge} parts \ref{contraction}) and \ref{supcontraction}), Theorem~\ref{thm:entrance_time}, and Lemma~\ref{lem:exp_webf}.  We start by proving Lemma~\ref{l:existInvM}.   

\begin{proof}[Proof of Lemma \ref{l:existInvM}]
  To show non-explosivity of $\xi_t$, we follow the frame of the argument proving Theorem 3.5 in \cite{Has_12}.  Let $\tau_{\infty} =
  \lim_{n\uparrow \infty} \tau_n$ to be the explosion time of $\xi_t$.
  We need to show that $\P_{\xi_0} [\tau_\infty < \infty]=0$ for all
  $\xi_0 \in \R^k$.  By Definition \ref{def:Lyapunov}, there exist
  constants $m,b>0$ such that
\begin{align}
\label{eqn:stdlya}
 \E_{\xi_0}\Psi(\xi_{t\wedge \tau_n}) - \Psi(\xi_0) &\leq
\E_{\xi_0}\int_0^{t\wedge \tau_n}(- m \Phi(\xi_s) + b )\, ds
\end{align}    
for all $t\geq 0$, $n \in \N$.  Since $\Phi \geq 0$, we obtain the bound
\begin{align}
\label{eqn:psimombound}
  \E_{\xi_0}\Psi(\xi_{t\wedge \tau_n}) \leq \Psi(\xi_0) + b t
\end{align}
for all $t\geq 0$, $n \in \N$.  Because $\Psi(\xi) \rightarrow \infty$
as $|\xi|\rightarrow \infty$, for $n\in \N$ large enough the
inequality above implies
\begin{align*}
  \P_{\xi_0}[\tau_{n} \leq t] = \frac{\inf_{|x|\geq n} \Psi(x)\cdot
    \P_{\xi_0}[\tau_{n} \leq t]}{\inf_{|x|\geq n} \Psi(x)}&\leq
  \frac{\E_{\xi_0}[ \Psi(\xi_{\tau_n}) 1_{\{\tau_n \leq
      t\}}]}{\inf_{|x|\geq n} \Psi(x)} \leq \frac{\Psi(\xi_0) + b t }{\inf_{|x|\geq n} \Psi(x)}
\end{align*}
for all $t\geq 0$.  Using the fact that $\Psi(\xi) \rightarrow \infty$
as $|\xi| \rightarrow \infty$, Fatou's lemma gives
\begin{align*}
\P_{\xi_0}[\tau_{\infty} \leq t] =0 \qquad \forall t\geq 0,
\end{align*}  
showing that $\xi_t$ is non-explosive.


We now show the existence of an invariant probability measure $\pi$ following the argument for Proposition 5.1 in \cite{HM_09}.  Since $\Psi\geq 0$, the bound \eqref{eqn:stdlya} implies
\begin{align*}
 \E_{\xi_0}\int_0^{t\wedge \tau_n}  \Phi(\xi_s)  \,ds \leq  \frac{\Psi(\xi_0)}{m} + \frac{b}{m} t
\end{align*}   
for all $t\geq 0$, $n\in \N$.  Using nonnegativity of $\Phi$, the Monotone convergence theorem
and the fact that $\tau_n \uparrow \infty$ almost surely, we obtain
\begin{align}
\label{eqn:intboundgen}
\E_{\xi_0}\int_0^{t} \Phi(\xi_s) \,ds \leq \frac{\Psi(\xi_0)}{m} + \frac{b}{m} t
\end{align} 
for all $t\geq 0$.  Letting $A_R=\{\xi \in \R^k\,: \, \Phi(\xi) \leq
R \}$, we note that the bound above implies
\begin{align*}
\frac{1}{t} \int_0^t \P_{\xi_0}[\xi_s \in A_R^c] \, ds \leq
\frac{\Psi(\xi_0) + b  t}{m R  t}.   
\end{align*}
In particular, it follows that
the sequence of measures
\begin{align*}
\pi_t^{\xi_0}(\, \cdot \,)= \frac{1}{t}\int_0^t \P_{\xi_0} [\xi_s \in \, \cdot \,]\,ds, \,\, t\geq 1,
\end{align*}
is relatively compact in the weak topology.  The Krylov-Bogoliubov Theorem \cite{KB_37} (see also the proof of sufficiency on pages 65-66 of \cite{Has_12}) now implies that there exists a sequence of times $t_n \uparrow \infty$ such that $\pi_{t_n}^{\xi_0}$ converges weakly to a probability measure $\pi^{\xi_0}$ on $\R^k$.  Moreover, by construction, $\pi^{\xi_0}$ is an invariant probability measure corresponding to the Markov process $\xi_t$.

We have left to show that $\pi^{\xi_0}$ defined above satisfies
\begin{align*}
\int_{\R^k} \Phi(\xi) \, \pi^{\xi_0}(d\xi)< \infty.    
\end{align*}
By \eqref{eqn:intboundgen}, we note that the inequality
\begin{align*}
\frac{1}{t}\E_{\xi_0} \int_0^t [\Phi(\xi_s) \wedge R] \, ds \leq
\frac{\Psi(\xi_0)}{m t} + \frac{b}{m} 
\end{align*}
is valid for any $R, t>0$.  In particular, we obtain the inequality
\begin{align*}
\int_{\R^k} [\Phi(\xi) \wedge R]\, \pi^{\xi_0}(d\xi) \leq  \frac{b}{m} 
\end{align*}
for any $R>0$.  Applying the Monotone Convergence Theorem, taking $R\rightarrow \infty$ finishes the proof of the result.
\end{proof}


To prove Theorem~\ref{LyapConverge} parts \ref{contraction}) and \ref{supcontraction}) and Theorem~\ref{thm:entrance_time}, we need the following lemma.  

\begin{lemma}
\label{lem:gronwcomp}
Suppose that $\xi_t$ has a Lyapunov pair $(\Psi, \Phi)$.  Then for all $\xi_0 \in \R^k$ and all $0\leq s \leq t$ 
\begin{align}
\E_{\xi_0}\Psi(\xi_t) - \E_{\xi_0}\Psi(\xi_s) \leq \E_{\xi_0} \int_s^t g(\xi_u) \, du
\end{align}
where $\, \E_{\xi_0} \Psi(\xi_t)$ is finite for all $t\geq 0$ and $\,\E_{\xi_0} \int_s^t |g(\xi_u)| \, du$ is finite for all $s,t \geq 0$.  Moreover, $t\mapsto \E_{\xi_0}\Psi(\xi_t)$ is a right-continuous function on $[0, \infty)$ with no upward jumps.   
\end{lemma}

\begin{proof}
We first claim that $\E_{\xi_0}\int_0^t  |g(\xi_s)| \, ds< \infty $ for all $t\geq 0$.  Note that this will then easily prove that $ \E_{\xi_0}\int_s^t |g(\xi_u)|\, \,du$ is finite for all $s, t\geq 0$.  
Observe that by definition and the fact that $g\leq -m \Phi+ b\leq b$  we have
\begin{align*}
\E_{\xi_0}\Psi(\xi_{t\wedge \tau_n})& \leq  \Psi(\xi_0) + \E_{\xi_0}\int_0^{t\wedge \tau_n}  g(\xi_s) \, ds\\
&\leq   \Psi(\xi_0) + b t+  \E_{\xi_0}\int_0^{t\wedge \tau_n} g(\xi_s) 1\{ g(\xi_s) \leq 0\}\, \, ds
\end{align*}
for all $t\geq 0$, $n\in \N$.  Rearranging the above and using non-negativity of $\Psi$ produces the inequality
\begin{align}
\label{eqn:crazy}
 \E_{\xi_0} \int_0^{t\wedge \tau_n} -g(\xi_s) 1\{g(\xi_s) \leq 0 \} \, ds \leq   \Psi(\xi_0) + bt.  
\end{align}
Applying the Monotone Convergence Theorem and non-explosivity of $\xi_t$, we see that 
\begin{align*}
\lim_{n\rightarrow \infty} \E_{\xi_0} \int_0^{t\wedge \tau_n} - g(\xi_s) 1\{g(\xi_s) \leq 0 \} \, ds &=\E_{\xi_0} \int_0^t -g(\xi_s) 1\{ g(\xi_s) \leq 0\}\, ds < \infty.  
\end{align*}
Moreover, using again the fact that $g\leq b$ we have 
\begin{align*}
\E_{\xi_0} \int_0^t |g(\xi_s)| \, ds  = \E_{\xi_0}\int_0^t -g(\xi_s) 1\{g(\xi_s) \leq 0 \}\, ds + \E_{\xi_0}\int_0^t g(\xi_s) 1\{ g(\xi_s) >0\} \, ds < \infty, 
\end{align*}
establishing the claim.  To finish proving the bound, note that since $s\leq t$ and $\Psi \geq 0$ we find 
\begin{align*}
 \Psi(\xi_{t\wedge \tau_n}) -\Psi(\xi_{s\wedge \tau_n})  &= \Psi(\xi_{t}) 1\{ t \leq \tau_n\} - \Psi(\xi_s) 1\{ s\leq \tau_n\}+ \Psi(\xi_{\tau_n}) (1\{t> \tau_n \} -1\{s> \tau_n \}) \\
& \geq  \Psi(\xi_{t}) 1\{ t \leq \tau_n\} - \Psi(\xi_s) 1\{ s\leq \tau_n\}.
\end{align*}
Also note that by \eqref{eqn:psimombound}, continuity of $\Psi$, and Fatou's lemma, $\E_{\xi_0} \Psi(\xi_t) \in[ 0, \infty)$ for all $t\geq 0$.  Thus, by the Dominated Convergence Theorem,
\begin{align*}
\E_{\xi_0}\Psi(\xi_t) - \E_{\xi_0}\Psi(\xi_s) &= \lim_{n\rightarrow \infty}(\E_{\xi_0}\Psi(\xi_{t}) 1\{ t \leq \tau_n\} - \E_{\xi_0}\Psi(\xi_s) 1\{ s\leq \tau_n\})\\&\leq \limsup_{n\rightarrow \infty}( \E_{\xi_0}\Psi(\xi_{t\wedge \tau_n}) -\E_{\xi_0}\Psi(\xi_{s\wedge \tau_n}) ) \\
&= \limsup_{n\rightarrow \infty}\bigg[ \E_{\xi_0} \int_{s\wedge \tau_n}^{t\wedge \tau_n} g(\xi_u) \, du + \text{Flux}(\xi_0, t, n) - \text{Flux}(\xi_0, s, n) \bigg]\\
&\leq \limsup_{n\rightarrow \infty} \E_{\xi_0} \int_{s\wedge \tau_n}^{t\wedge \tau_n} g(\xi_u) \, du  =\E_{\xi_0} \int_s^t g(\xi_u) \, du,
\end{align*}
finishing the proof of the bound.  

We have left to show that $t\mapsto \E_{\xi_0}\Psi(\xi_t)$ is a right-continuous function on $[0, \infty)$ with no upward jumps.  To see that ${\E_{\xi_0}\Psi(\xi_t) \rightarrow \Psi(\xi_s)}$ as $t\rightarrow s^+$, note that since $g\leq -m \Phi + b \leq b$ and $\Psi\geq 0$ we have 
\begin{align*}
\liminf_{t\rightarrow s^+} \E_{\xi_0}  \Psi(\xi_t)\leq \limsup_{t\rightarrow s^+} \E_{\xi_0}  \Psi(\xi_t) \leq \limsup_{t\rightarrow s^+}\big(\E_{\xi_0}\Psi(\xi_s)  + b(t-s) \big) = \E_{\xi_0}\Psi(\xi_s).  
\end{align*} 
On the other hand, by Fatou's lemma, continuity of $\Psi$, and path continuity of $\xi_t$
\begin{align*}
\E_{\xi_0}\Psi(\xi_s) \leq \liminf_{t\rightarrow s^+} \E_{\xi_0 }\Psi(\xi_t).  
\end{align*}       
Therefore $\lim_{t\rightarrow s^+} \E_{\xi_0} \Psi(\xi_t) = \E_{\xi_0}\Psi(\xi_s)$ as claimed.  Hence $t\mapsto \E_{\xi_0}\Psi(\xi_t)$ is right-continuous on $[0, \infty)$.  Note also that the bound
\begin{align*}
\E_{\xi_0}\Psi(\xi_t) \leq \E_{\xi_0}\Psi(\xi_s) + b(t-s),
\end{align*}  
which is satisfied for all $0 \leq s \leq  t$, implies that for $t>0$
\begin{align*}
\liminf_{s\rightarrow t^{-}} \E_{\xi_0}\Psi(\xi_s) \geq \liminf_{s\rightarrow t^{-}} (\E_{\xi_0}\Psi(\xi_t) - b(t-s))= \E_{\xi_0} \Psi(\xi_t).
\end{align*}  
Hence, $t\mapsto \E_{\xi_0} \Psi(\xi_t)$ has no upward jumps.  
\end{proof}

We will also need the following ODE comparison result.  
\begin{proposition}
\label{prop:odecomp}
Fix $T\in(0, \infty)$ and let $f\in C(\R) $ be non-increasing.  Suppose that $\psi\in C([0, T])$ satisfies 
\begin{align}
\label{eqn:ppe}
\psi(t)= \psi(s) + \int_s^t f(\psi(u)) \, du
\end{align}
for all $s,t$ with $0 \leq s\leq t \leq T$.  If $\phi(t)$ is a right-continuous function on $[0, T]$ with no upward jumps satisfying $\phi(0)=\psi(0)$ and the inequality 
\begin{align}
\label{eqn:ppcomp}
\phi(t) \leq \phi(s) + \int_s^t f(\phi(u)) \, du   
\end{align}
for all $s, t$ with $0\leq s\leq t\leq T$, then $\phi(t) \leq \psi(t)$ for all $0 \leq t \leq T$.  
\end{proposition}
\begin{proof}
Suppose that there exists $T_0 \in (0, T]$ such that $\phi(T_0)- \psi(T_0) >0$.  Define $$S_0 = \sup\{ t\in [0, T_0]\, : \phi(t) -\psi(t)\leq 0\}$$ and observe that $\phi-\psi$ is also a right-continuous function with no upward jumps on $[0, T]$ as $\psi$ is a continuous function.  Hence it follows that $S_0\in [0, T_0)$, $\phi(S_0)- \psi(S_0)=0$ and $\phi(t)- \psi(t) >0$ for $t\in (S_0, T_0)$.  Now use that relations \eqref{eqn:ppe} and \eqref{eqn:ppcomp} and the fact that $f$ is non-increasing to see that for $t\in (S_0,T_0)$
\begin{align*}
0 < \phi(t) - \psi(t) \leq \int_{S_0}^t f(\phi(s)) - f(\psi(s)) \, ds \leq 0, 
\end{align*}   
showing that no such $T_0$ can exist.  Hence $\phi(t) \leq \psi(t)$ for all $t\in [0, T]$.      
\end{proof}

We are now able to prove Theorem~\ref{LyapConverge} parts \ref{contraction}) and \ref{supcontraction}) and Theorem~\ref{thm:entrance_time}.

\begin{proof}[Proof of Theorem~\ref{LyapConverge} \ref{contraction}), \ref{supcontraction})]
Letting $\mathscr{P}_t(\xi_0, \, \cdot \,)= \P_{\xi_0}[\xi_t \in \, \cdot \, ]$, we start by proving part \ref{contraction}), aiming to apply Theorem 1.3 of \cite{HM_11}.  To connect with their notation, fix $T_0>0$ and let $\mathcal{P}(\xi_0, \, \cdot\,)$ be the one step transition probability of the Markov chain $\xi_{nT_0}$, $n=0, 1, 2, \ldots$, and $\mathcal{P}^n$ the associated semi-group.  We first check that Assumption 1 of \cite{HM_11} is satisfied under our hypotheses.  By Lemma \ref{lem:gronwcomp}, we have for $0\leq s\leq t$
\begin{align*}
\E_{\xi_0} \Psi(\xi_{t}) \leq \E_{\xi_0}\Psi(\xi_s) + \E_{\xi_0}\int_s^{t} g(\xi_u) \, du &\leq \E_{\xi_0}\Psi(\xi_s) + \E_{\xi_0}\int_s^{t} (-m \Psi(\xi_u) + b) \, du\\
& = \E_{\xi_0} \Psi(\xi_s) + \int_{s}^t (-m \E_{\xi_0}\Psi(\xi_u) + b)  \, du 
\end{align*}   
where the final equality follows from Tonelli's theorem and the fact that each quantity above is finite.  Since $t\mapsto \E_{\xi_0} \Psi(\xi_t)$ is right-continuous on $[0, \infty)$ with no upward jumps, by Proposition \ref{prop:odecomp} we obtain the bound
\begin{align}
\label{eqn:semigroupbound}
\mathscr{P}_t \Psi(\xi_0) = \E_{\xi_0}\Psi(\xi_{t}) \leq e^{-m t} \Psi(\xi_0) + \frac{b}{m}   
\end{align}
for all $t\geq 0$.  In particular, 
\begin{align*}
\mathcal{P}\Psi(\xi_0)\leq e^{-m T_0} \Psi(\xi_0) + \frac{b}{m}.
\end{align*}
Since $e^{-mT_0} \in (0,1)$, this now validates Assumption 1 of \cite{HM_11}.  To check that Assumption 2 of \cite{HM_11} is satisfied, recall that since $\xi_t$ is an It\^{o} diffusion with smooth coefficients and that $\xi_t$ has a uniformly elliptic diffusion matrix, $\xi_t$ has a transition probability density function $p_t(\xi_0, \xi)$ with respect to Lebesgue measure $d\xi$ on $\R^k$ which is $C^\infty$ and strictly positive for $(t, \xi_0, \xi) \in (0, \infty) \times \R^k \times \R^k$.  In particular using the notation in Assumption 2 of \cite{HM_11}, if $R>2b/[a(1-e^{-mT_0})]$, $\mathcal{C}= \{\xi \, : \, \Psi(\xi) \leq R \}$ ($\mathcal{C}$ is compact as $\Psi$ is continuous) and $B= \{\xi  \, : \, |\xi| \leq 1 \}$, then for any Borel set $A\subset \R^k$
\begin{align*}
\inf_{\xi_0 \in \mathcal{C}} \mathcal{P}(\xi_0, A) = \inf_{\xi_0 \in \mathcal{C}} \int_{A} p_{T_0}(\xi_0, \xi) \, d\xi  &\geq   \inf_{\xi_0 \in \mathcal{C}} \int_{A\cap B} p_{T_0}(\xi_0, \xi) \, d\xi \geq \gamma \lambda(B) \frac{\lambda(A\cap B)}{\lambda(B)}
\end{align*}  
 where $\gamma= \min_{\xi_0 \in \mathcal{C},\, \xi \in B} p_{T_0}(\xi_0, \xi)  >0$ and $\lambda$ denotes Lebesgue measure.  That is, Assumption 2 of \cite{HM_11} is also satisfied.  Applying Theorem 1.3 of \cite{HM_11}, letting $w_\beta(\xi) = 1+ \beta \Psi(\xi)$, there exists $\alpha \in (0,1)$ and $\hat{\beta} >0$ such that for any two probability measures $\nu_1, \nu_2 \in \mathcal{M}_{w_\beta}(\R^k)$
 \begin{align*}
d_{w_{\hat{\beta}}}(\nu_1 \mathcal{P} ,\nu_2 \mathcal{P}) \leq \alpha d_{w_{\hat{\beta}}}( \nu_1, \nu_2).  
 \end{align*}
In particular, iterating this bound produces
\begin{align*}
d_{w_{\hat{\beta}}}(\nu_1 \mathcal{P}^n , \nu_2\mathcal{P}^n ) \leq \alpha^n d_{w_{\hat{\beta}}}(\nu_1, \nu_2)
\end{align*}
for all $n\geq 0$.  Now to get the bound for any $\beta>0$ (not just for some $\hat{\beta}>0$), first note that for all $\beta, \, \beta'>0$ we have a constant $C_{\beta, \beta'}>0$ such that 
\begin{align*}
d_{w_\beta}( \nu_1, \nu_2) \leq C_{\beta, \beta'} d_{w_{\beta'}}(\nu_1, \nu_2).
\end{align*}
Hence for any $\beta>0$, there exists a constant $C_\beta$ such that 
\begin{align*}
d_{w_\beta}(\nu_1 \mathcal{P}^n , \nu_2 \mathcal{P}^n ) \leq C_\beta \alpha^n d_{w_\beta}(\nu_1, \nu_2)
\end{align*}
for all $n\geq 0$.  To finish part \ref{contraction}), we have left to translate the above bound to the continuous time.  To do this, we first claim that $\nu \mathscr{P}_t \in \mathcal{M}_{w_\beta}(\R^k)$ for all $t\geq 0$ whenever $\nu \in \mathcal{M}_{w_\beta}(\R^k)$.  Indeed, this follows from the bound \eqref{eqn:semigroupbound} and Tonelli's Theorem as
\begin{align*}
\int_{\R^k} w_\beta(\xi) \nu \mathscr{P}_t(d\xi) = \int_{\R^k} \nu(d\xi) (\mathscr{P}_t w_\beta)(\xi) \leq \int_{\R^k} \Big(\beta \Psi(\xi) + 1+\beta \frac{b}{m}\Big)\nu(d\xi)< \infty.    
\end{align*} 
Therefore if $t= nT_0+ \delta$ for some integer $n\geq 0$ and $\delta \in [0, T_0)$ we have by the semigroup property, the claim and the Fubini-Tonelli Theorem
\begin{align*}
d_{w_\beta}(\nu_1 \mathscr{P}_t, \nu_2 \mathscr{P}_t) = d_{w_\beta}( (\nu_1 \mathscr{P}_\delta)\mathcal{P}^{n}, (\nu_2 \mathscr{P}_\delta)\mathcal{P}^{n})&\leq C_\beta \alpha^n d_{w_\beta}(\nu_1 \mathscr{P}_\delta, \nu_2 \mathscr{P}_\delta )\\
&=  C_\beta \alpha^n  \sup_{|\varphi|\leq w_\beta} \bigg[ \int_{\R^k} \mathscr{P}_\delta \varphi(\xi) (\nu_1(d\xi) -  \nu_2(d\xi))\bigg] \\ 
&\leq C_\beta' \alpha^n d_{w_\beta}( \nu_1, \nu_2)\\
& \leq C_\beta'' (\alpha^{1/T_0})^t d_{w_\beta}( \nu_1, \nu_2)
\end{align*} 
where in the penultimate inequality we have used the bound \eqref{eqn:semigroupbound}.  This finishes the proof of part \ref{contraction}) of the result.  

To prove part \ref{supcontraction}), first observe that if $(\Psi, \Psi^{1+\delta})$, $\delta >0$,  is a Lyapunov pair corresponding to $\xi_t$, then so is $(\Psi, \Psi)$ since $\Psi(\xi) \rightarrow \infty$ as $|\xi| \rightarrow \infty$.  In particular, the conclusion in part \ref{contraction}) also holds; that is, for $\beta >0$ fixed and $w=1+ \beta \Psi$ (here we choose not to emphasize the dependence on $\beta >0$), there exist positive constants $C, \eta$ such that 
\begin{align*}
d_w(\nu_1 \mathscr{P}_t, \nu_2 \mathscr{P}_t ) \leq C e^{-\eta t} d_w(\nu_1, \nu_2)
\end{align*}
for all $t\geq 0$ and all $\nu_1, \nu_2 \in \mathcal{M}_w(\R^k)$.  To improve this bound in the sense of the statement in part \ref{supcontraction}), we follow the reasoning given in Section 6 of \cite{AKM_12}.    That is, we will first show that if $(\Psi, \Psi^{1+\delta})$ is a Lyapunov pair corresponding to $\xi_t$ and $t_0 >0$, then there exists a constant $K_{t_0}>0$ such that 
\begin{align*}
(\mathscr{P}_t \Psi)(\xi_0) \leq K_{t_0}
\end{align*}
for all $t\geq t_0$ and all $\xi_0 \in \R^k$.  Applying Lemma \ref{lem:gronwcomp} and using the fact that $g\leq -m \Psi^{1+\delta} + b$ for some constants $m,b>0$, we see that    
\begin{align*}
\E_{\xi_0} \Psi(\xi_t) - \E_{\xi_0}\Psi(\xi_s) \leq \E_{\xi_0} \int_s^t g(\xi_u ) \, du &\leq \E_{\xi_0} \int_{s}^t (-m \Psi(\xi_u)^{1+\delta} + b)\\
& \leq \int_s^t (- m [\E_{\xi_0}\Psi(\xi_u)]^{1+\delta} + b) \, du
\end{align*}
where the last inequality follows Tonelli's Theorem and Jensen's inequality.  Now let $$h_{\xi_0}(t):= \E_{\xi_0}\Psi(\xi_t) \,\,\, \text{ and } \,\,\, T= \inf\{ t>0 \,: \, h_{\xi_0}(t) \leq (2b m^{-1})^{\frac{1}{1+\delta}}\}.$$  Observe that for all $0\leq s \leq t \leq T$ we have shown that
\begin{align*}
h_{\xi_0}(t) \leq h_{\xi_0}(s) + \int_s^t (-m h_{\xi_0}(u)^{1+\delta} + b)\, du \leq h_{\xi_0}(s) - \int_s^t \frac{m}{2} h_{\xi_0}(u)^{1+\delta} \, du.  
\end{align*} 
Hence, in particular, this bound implies that for all times $t\geq T$ 
\begin{align*}
h_{\xi_0}(t) \leq (2b m^{-1})^{\frac{1}{1+\delta}} 
\end{align*}
as the map $t\mapsto h_{\xi_0}(t)$ is strictly decreasing whenever $h_{\xi}(t) \geq  (2b m^{-1})^{\frac{1}{1+\delta}}$.  Moreover, by Proposition \ref{prop:odecomp}, for all $t\in [0, T]$
\begin{align*}
h_{\xi_0}(t) = \E_{\xi_0}\Psi(\xi_t) \leq \frac{1}{(\frac{m \delta t}{2} + \Psi(\xi_0)^{-\delta})^{1/\delta}}.   
\end{align*}
Therefore, for all times $t\geq t_0 >0$ we have that 
\begin{align*}
\mathscr{P}_t \Psi(\xi_0)= \E_{\xi_0}\Psi(\xi_t) \leq K_{t_0}:=\max \bigg\{  (2b m^{-1})^{\frac{1}{1+\delta}}, \frac{1}{(\frac{m\delta t_0}{2})^{1/\delta}}\bigg\},
\end{align*}
as claimed.

To finish proving the result, note by Tonelli's Theorem: If $\nu$ is a probability measure on $\R^k$ and $t>0$, then 
\begin{align*}
\int_{\R^k} w(\xi) (\nu \mathscr{P}_t)(d\xi) = \int_{\R^k} (\mathscr{P}_t w)(\xi) \nu(d\xi)\leq K_t \int_{\R^k} \nu(d\xi) < \infty; 
\end{align*}
that is, $\nu \mathscr{P}_t \in \mathcal{M}_w(\R^k)$ for $t>0$ and any probability measure $\nu$ on $\R^k$.  Also, after applying the proof of Proposition 6.2 of \cite{AKM_12}, we see that for any probability measures $\nu_1, \nu_2$ on $\R^k$ and any $t\geq 1$
\begin{align*}
d_w( \nu_1 \mathscr{P}_t, \nu_2 \mathscr{P}_t) \leq (1+\beta K_1) \| \nu_1 - \nu_2\|_{TV}.
\end{align*}       
Hence, combining this bound with the bound obtained in part \ref{contraction}) we see that for any $t\geq 1$ and any two probability measures $\nu_1, \nu_2 $ on $\R^k$
\begin{align*}
d_w(\nu_1 \mathscr{P}_t, \nu_2 \mathscr{P}_t) = d_w( \nu_1\mathscr{P}_1 \mathscr{P}_{t-1} , \nu_2 \mathscr{P}_1\mathscr{P}_{t-1} )&\leq C_\beta e^{-\eta (t-1)} d_w( \nu_1 \mathscr{P}_1, \nu_2 \mathscr{P}_1)\\
& \leq C_\beta' e^{-\eta t} \| \nu_1 - \nu_2 \|_{TV}
\end{align*}
as $\nu_i \mathscr{P}_1 \in \mathcal{M}_w(\R^k)$.  Note that this now finishes the proof of part \ref{supcontraction}) and the result.   
\end{proof}

\begin{proof}[Proof of Theorem \ref{thm:entrance_time}]
We first show that for all $\gamma >0$ large enough 
\begin{align*}
\P_{\xi_0}[\upsilon_\gamma < \infty]=1
\end{align*}
for all $\xi_0 \in \R^k$.  Let $(\Psi, \Psi^{1+\delta})$ be a Lyapunov pair corresponding to $\xi_t$.  Since $\Psi(\xi) \rightarrow \infty$ as $|\xi | \rightarrow \infty$, we may pick $\gamma >0$ large enough so that for all $|\xi | \geq \gamma$ 
\begin{align*}
- m \Psi^{1+\delta}(\xi) + b \leq - 1, 
\end{align*}
where $m,b>0$ are the constants given in the definition of Lyapunov pair.  In particular, we have that for $\gamma >0$ large enough
\begin{align*}
\inf_{|\xi| \geq \gamma}\Psi(\xi)  \P_{\xi_0}[\upsilon_\gamma \wedge \tau_n \geq t] &\leq  \E_{\xi_0} \Psi(\xi_{t\wedge \upsilon_\gamma \wedge \tau_n})\\
&\leq\Psi(\xi_0)+ \E_{\xi_0} \int_0^{t\wedge \upsilon_\gamma \wedge \tau_n} (-m \Psi(\xi_s)^{1+\delta} +b) \, ds \\
& \leq \Psi(\xi_0)+ \E_{\xi_0} \int_0^{t\wedge \upsilon_\gamma \wedge \tau_n} -1 \, ds \\
& \leq \Psi(\xi_0)- t\P_{\xi_0} [\upsilon_\gamma \wedge \tau_n \geq t]   .
\end{align*}
Rearranging the previous inequality produces the following bound
\begin{align*}
\P_{\xi_0 }[\upsilon_\gamma \wedge \tau_n \geq t ]\leq \frac{\Psi(\xi_0)}{\inf_{|\xi| \geq \gamma} \Psi(\xi) + t}
\end{align*}
which holds for all $t\geq  0$ and $\gamma >0$ large enough.  First take $n\rightarrow \infty$, applying non-explosivity of $\xi_t$, and then take $t\rightarrow \infty$ to see that for $\gamma >0$ large enough
\begin{align*}
\P_{\xi_0}[\upsilon_\gamma = \infty] = 0,
\end{align*}
finishing the proof of the first part of the result. 

To prove the second part of the result, we first note by the proof of Theorem \ref{LyapConverge} part \ref{supcontraction}), $(t, \xi_0) \mapsto \E_{\xi_0}\Psi(\xi_t)$ is bounded on $[t_0, \infty) \times \R^k$ for any $t_0>0$.  Hence, there exists a constant $K>0$ such that for all $t\geq t_0>0$ and $\xi_0 \in \R^k$: 
\begin{align*}
K \geq \E_{\xi_0} \Psi(\xi_t) &\geq \E_{\xi_0} 1_{\{\upsilon_\gamma\geq t\}} \Psi(\xi_t)\\
&\geq \inf_{|\xi|\geq \gamma} \Psi(\xi) \cdot \P_{\xi_0} [\upsilon_\gamma \geq t].  
\end{align*}
Thus for $\gamma >0$ large enough and $t\geq t_0$ 
\begin{align*}
\P_{\xi_0} [\upsilon_\gamma \geq t ] \leq \frac{K}{\inf_{|\xi|\geq \gamma} \Psi(\xi)}.
\end{align*}
Since $\Psi(\xi) \rightarrow \infty$ as $|\xi|\rightarrow \infty$, it follows that for each $\epsilon,t  >0$ there exists a $\gamma >0$ such that 
\begin{align*}
\sup_{\xi_0\in \mathbb{R}^d} \P_{\xi_0} [\upsilon_\gamma\geq t]\leq \epsilon  
\end{align*}
finishing the proof of the result.    
\end{proof}

\begin{proof}[Proof of Lemma \ref{lem:exp_webf}]
Fixing $c\in \R$ and $a\in \big(0, \frac{3}{2}n+1\big)$, we first study the solution of the boundary-value problem
\begin{align}\label{G:eqI}
&\frac{\sigma^2}{2}G''_{a,c}(\eta) + \Big(\frac{3}{2}n +1 \Big)\eta
G'_{a,c}(\eta) + a G_{a,c}(\eta) = 0\\ 
\nonumber &G_{a,c}(-\eta^*+c)=G_{a,c}(\eta^* +c) = 1 
\end{align} 
and show all conclusions of Lemma \ref{lem:exp_webf} without assuming that we may write $G_{a,c}(\eta)= \E_{\eta} e^{a \tau_c}.$  In particular, we leave the proof that $G_{a,c}(\eta)= \E_{\eta} e^{a \tau_c}$ until the end of the argument.  To further understand solutions of \eqref{G:eqI}, we 
transform the equation to Weber's equation.  To this end, we write  
\begin{align}
G_{a,c}(\eta) = e^{-\beta\eta^2/4} H(\sqrt{\beta} \eta)
\end{align}
where $\beta =(3n+2)/\sigma^2 $ and note that $H$ satisfies
\begin{align}\label{w:eqnI}
H''(v) -\Big(\frac{v^2}{4} + \frac{1}{2} - \frac{2a}{\sigma^2 \beta} \Big) H(v) =0.
\end{align}
The two linearly independent general solutions of \eqref{w:eqnI}, denoted by 
\begin{align*}
U\Big(\frac{2a}{\sigma^2 \beta}-1/2, \pm i v\Big), 
\end{align*}
have the following integral representations (cf. Chapter 12.5 of \cite{NIST:DLMF})  
\begin{align*} 
  U\Big(\kappa-1/2, \pm i v\Big) =
  \frac{e^{\frac{v^2}{4}}}{\Gamma\big( \kappa\big)}
  \int_{0}^{\infty} t^{-1 + \kappa} e^{-t^2/2 \mp i
    v t} \, dt , \, \, \, v \in \mathbb{R}.
\end{align*}
where we have introduced $\kappa=\frac{2a}{\sigma^2 \beta}$ in the
interest of brevity.
Using these expressions, the boundary conditions given in
\eqref{G:eqI} and the assumption $0<a<\frac{3}{2}n+1$ we may formally
write  
\begin{align}
\label{eqn:exprGac}
  G_{a,c}(\eta)   = \frac{D-B}{AD-BC} \int_0^\infty f(t) \cos(\sqrt{\beta}\eta t) \, dt + \frac{A-C}{AD-BC} \int_0^\infty f(t) \sin(\sqrt{\beta}\eta t) \, dt  
\end{align} 
where $f(t) = t^{-1+ \kappa} e^{-\frac{t^2}{2}}$ and 
\begin{xalignat*}{2}
A &= \int_0^\infty f(t) \cos(\sqrt{\beta} (\eta^*+c)t)\, dt, & 
B &=  \int_0^\infty f(t) \sin(\sqrt{\beta} (\eta^*+c)t)\, dt \\
C&= \int_0^\infty f(t) \cos(\sqrt{\beta} (\eta^*-c)t)\, dt, & 
D&= -\int_0^\infty f(t) \sin(\sqrt{\beta} (\eta^*-c)t)\, dt.     
  \end{xalignat*}
The only question pertaining to the validity of \eqref{eqn:exprGac} is
that $AD-BC$ could possibly be zero.  We will now show that this is
not the case for all $\eta^*>|c|$ sufficiently large.  It will then follow
easily that $G_{a,c} \in C^\infty([-\eta^* + c, \eta^*+c])$ for all
$\eta^*>|c|$ large enough by \eqref{eqn:exprGac}.  Write
\begin{align*}
  A &= \Gamma(\kappa) e^{-\beta \frac{(\eta^*+c)^2}{4}} \text{Re}\, U\big(\kappa- \frac{1}{2}, i \sqrt{\beta}(\eta^*+c)\big) \\
  B &= - \Gamma(\kappa) e^{-\beta
    \frac{(\eta^*+c)^2}{4}} \text{Im}\, U\big(\kappa- \frac{1}{2}, i \sqrt{\beta}(\eta^*+c)\big)  \\ 
  C&=  \Gamma(\kappa) e^{-\beta
    \frac{(\eta^* -c)^2}{4}} \text{Re}\, U\big(\kappa- \frac{1}{2}, i \sqrt{\beta}(\eta^*-c)\big) \\ 
  D&= \Gamma(\kappa) e^{-\beta
    \frac{(\eta^*-c)^2}{4}} \text{Im}\, U\big(\kappa - \frac{1}{2}, i \sqrt{\beta}(\eta^*-c)\big) .
\end{align*}
One can then use the asymptotic formula for $U(a, z)$ as $z\rightarrow
\infty$ in Section 12.9 of \citep{Olver:2010:NHMF} to deduce that as
$\eta^* \rightarrow \infty$ one has
\begin{xalignat*}{2}
A &= \Gamma( \kappa)\frac{\cos(\frac{\pi}{2}  \kappa)}{(\sqrt{\beta}\eta^*)^{\kappa}}
\Big\{1+ O\big(\tfrac1{\eta^*}\big) \Big\},   &
B&=  \Gamma( \kappa)\frac{\sin(\frac{\pi}{2} \kappa)}{(\sqrt{\beta}\eta^*)^{\kappa}}
\Big\{1+ O\big(\tfrac1{\eta^*}\big) \Big\},   \\
C&= \Gamma(\kappa)\frac{\cos(\frac{\pi}{2}\kappa)}{(\sqrt{\beta}\eta^*)^{\kappa}}
\Big\{1+ O\big(\tfrac1{\eta^*}\big) \Big\},   & 
D &= - \Gamma(\kappa)\frac{\sin(\frac{\pi}{2}  \kappa)}{(\sqrt{\beta}\eta^*)^{\kappa}}
\Big\{1+ O\big(\tfrac1{\eta^*}\big) \Big\}.   
\end{xalignat*}
From these formulas and the fact that $0<\kappa<1$, we can easily conclude that for $\eta^*$ large enough
$AD-BC\neq 0$.  To see the claimed asymptotic formula for $G_{a,c}'$,
we may differentiate under the integrals in \eqref{eqn:exprGac} to
obtain
\begin{multline*}
  G_{a,c}'(\pm \eta^* +c)  = -\sqrt{\beta}\frac{D-B}{AD-BC}
  \int_0^\infty t f(t) \sin(\sqrt{\beta}(\pm\eta^*+c) t) \, dt \\ 
+ \sqrt{\beta} \frac{A-C}{AD-BC} \int_0^\infty t f(t)
  \cos(\sqrt{\beta}(\pm\eta^*+c) t) \, dt.
\end{multline*}
Using a similar trick, we may write the functions 
\begin{align*}
\int_0^\infty t f(t) \sin(\sqrt{\beta}(\pm\eta^*+c) t) \, dt\,\,
\text{ and }\,\,\int_0^\infty  t f(t) \cos(\sqrt{\beta}(\pm\eta^*+c)
t) \, dt  
\end{align*}
in terms of the function $U$ as
\begin{align*}
  \int_0^\infty t f(t) \sin(\sqrt{\beta}(\pm\eta^*+c) t) \, dt 
  = -\kappa\Gamma(\kappa) e^{- \frac{\beta}{4}(\pm \eta^* +c)^2} \text{Im} \,
  U\big(\kappa+ \frac{1}{2}, i\sqrt{\beta}(\pm
  \eta^* +c) \big),
  \end{align*}
and
  \begin{align*}
    \int_0^\infty t f(t) \cos(\sqrt{\beta}(\pm\eta^*+c) t) \, dt 
  =\kappa \Gamma(\kappa) e^{- \frac{\beta}{4}(\pm \eta^* +c)^2} \text{Re} \,
  U\big(\kappa+ \frac{1}{2}, i\sqrt{\beta}(\pm
  \eta^* +c) \big).
  \end{align*}
Again, applying the asymptotic formula for $U(a, z)$ as $z\rightarrow
\infty$ in Section 12.9 of \cite{NIST:DLMF} with those derived for $A,
B, C,D$, we can obtain the claimed asymptotic formulas for $G'_{a,c}$.  To see the symmetry $G_{a,0}(-\eta)=G_{a,0}(\eta)$ for $\eta \in
[-\eta^* +c, \eta^* +c]$, set $c=0$ in \eqref{eqn:exprGac} to see that  
\begin{align*}
G_{a,0}=\frac{\int_{0}^{\infty} f(t) \cos( \sqrt{\beta} \eta t)
    \, dt}
  {\int_{0}^{\infty} f(t) \cos( \sqrt{\beta} \eta^* t) \, dt}.
\end{align*}

Finally, to see that $G_{a,c}(\eta) = \E_{\eta} e^{a\tau_c}$ we first show that $\E_{\eta} e^{a \tau_c} < \infty$.  Indeed, since $\eta_t$ with $\eta_0=\eta$ is normally
  distributed with mean $e^{(\frac{3}{2}n+1)t}\eta $ and variance
$$ \frac{\sigma^2}{3n+2}(e^{(3n+2)t}-1),$$ we obtain
\begin{align*}
  \E_{\eta} e^{a \tau_c } = \int_{0}^{\infty} \P_{\eta}
  \{e^{a\tau_c}>t\}\, dt
  &\leq  2 + \int_{2}^{\infty} \P_{\eta} \{\tau_c>a^{-1}\log(t) \}\, dt\\
  &\leq  2+ \int_{2}^{\infty} \P_{\eta} \{\eta_{a^{-1} \log(t)}\in
  [-\eta^* +c ,\eta^* +c] \}\, dt\\ 
  &\leq 2 + K \int_{2}^\infty \frac{1}{\sqrt{t^{\frac{3n+2}{a}}-1}}\,
  dt
\end{align*}
for some constant $K>0$.  Notice that the last integral above is
finite since ${a\in \big(0, \frac{3}{2} n+1\big)}$.  To finish, since $G_{a,c}$ is bounded on $[-\eta^* + c, \eta^*+c]$, we may apply Dynkin's formula to obtain
\begin{align*}
\E_\eta e^{a (\tau_c \wedge t) }G_{a,c}(\eta_{\tau_c \wedge t}) = G_{a,c}(\eta)
\end{align*}
for all finite times $t\geq 0$.  Since $\E_\eta e^{a\tau_c} < \infty$ and $G_{a,c}$ is bounded on $ [-\eta^* + c, \eta^*+c]$, we may apply dominated convergence and take $t\rightarrow \infty$ to see that $G_{a,c}(\eta)= \E_{\eta} e^{a \tau_c}$.

\end{proof}

\bibliographystyle{plain}
\bibliography{polyZ}

\end{document}